\documentclass{article}
\usepackage{graphicx}
\usepackage{hyperref}
\usepackage{amsmath}
\usepackage{amssymb}
\usepackage{marginnote}
\usepackage{cmll}
\usepackage{mathtools}
\usepackage{algpseudocode}
\usepackage{algorithm}
\usepackage{bussproofs}
\usepackage{xcolor}
\usepackage{forest}
\usepackage[margin=2.6cm]{geometry}

\usepackage{amsthm}

\newcommand{\aninductiveinversecorrespondent}{a pre-inverse correspondent}
\newcommand{\cryptoinductiveinversecorrespondent}{inverse correspondent}
\newcommand{\inductiveinversecorrespondent}{pre-inverse correspondent}
\newcommand{\Inductiveinversecorrespondent}{Pre-inverse correspondent}

\newcommand{\Diamondblack}{\mathbin{\blacklozenge}}
\newcommand{\metaand}{\ \with \ }
\newcommand{\bigmetaor}{\bigparr}

\newcommand{\bigmetaand}{\bigwith}

\newcommand{\metaor}{\ \parr\ }
\newcommand{\metanot}{{\sim}}
\newcommand{\val}[1]{{[\![}{#1}{]\!]}}
\newcommand{\descr}[1]{{(\![}{#1}{]\!)}}
\newcommand{\atprop}{\mathsf{AtProp}}
\newcommand{\nomset}{\mathrm{NOM}}
\newcommand{\cnomset}{\mathrm{CNOM}}

\newcommand{\nomh}{\mathbf{h}}
\newcommand{\nomi}{\mathbf{i}}
\newcommand{\nomj}{\mathbf{j}}
\newcommand{\nomk}{\mathbf{k}}
\newcommand{\cnoml}{\mathbf{l}}
\newcommand{\cnomm}{\mathbf{m}}
\newcommand{\cnomn}{\mathbf{n}}
\newcommand{\cnomo}{\mathbf{o}}
\newcommand{\pureu}{\mathbf{u}}
\newcommand{\purev}{\mathbf{v}}
\newcommand{\purew}{\mathbf{w}}
\newcommand{\jty}{J^{\infty}}
\newcommand{\mty}{M^{\infty}}

\newcommand{\pdlab}{{>}\!\!\!-}
\newcommand{\starfor}{\slash_\star}
\newcommand{\starback}{\backslash_\star}
\newcommand{\circfor}{\slash_\circ}
\newcommand{\circback}{\backslash_\circ}
\newcommand{\upsE}{@}
\newcommand{\crit}[1]{\mathsf{crit}(#1)}
\newcommand{\flattify}[1]{\mathsf{Flat}(#1)}

\newcommand{\leclass}{\mathbb{LE}}
\newcommand{\langbase}{\mathcal{L}_{\mathrm{LE}}}

\newcommand{\langalba}{\langbase^+} 
\newcommand{\langineq}{\langbase^{\leq}}
\newcommand{\langmeta}{\langbase^{\mathrm{FO}}}
\newcommand{\langpolarities}{\mathcal{L}_{\mathrm{P}}^{\mathrm{FO}}}

\newcommand{\gengeomlevel}[1]{\mathrm{GA}_{#1}}

\newcommand{\addmeet}{\wedge}
\newcommand{\addjoin}{\vee}
\newcommand{\mulmeet}{\otimes}
\newcommand{\muljoin}{\oplus}

\newcommand{\adjunction}[4]{#1\negmedspace: #2\rightleftarrows #3:\negmedspace #4}

\DeclareRobustCommand 
\Compactldots{\mathinner{\ldotp\mkern-3mu\ldotp\mkern-3mu\ldotp}}
\renewcommand{\ldots}{%
  \Compactldots
}

\theoremstyle{plain}
\newtheorem{thm}{Theorem}[section]
\newtheorem{theorem}[thm]{Theorem}

\newtheorem{corollary}[thm]{Corollary}
\newtheorem{prop}[thm]{Proposition}

\newtheorem{lemma}[thm]{Lemma}

\theoremstyle{definition}
\newtheorem{definition}[thm]{Definition}
\newtheorem{notation}[thm]{Notation}
\newtheorem{example}[thm]{Example}

\newtheorem{remark}[thm]{Remark}

\title{\bf Unified inverse correspondence for LE-logics}
\author{Alessandra Palmigiano\textsuperscript{1,2} \and Mattia Panettiere\textsuperscript{1}}
\date{\small\textsuperscript{1}School of Business and Economics, Vrije Universiteit, Amsterdam, The Netherlands \\
\textsuperscript{2}Department of Mathematics and Applied Mathematics, University of Johannesburg, South Africa}

\begin{document}

\maketitle

\begin{abstract}
We generalize Kracht's theory of internal describability from classical modal logic to the family of all logics canonically associated with varieties of normal lattice expansions (LE algebras). 
We work in the purely algebraic setting of perfect LEs; the formulas playing the role of Kracht's formulas in this generalized setting pertain to a first order language whose atoms are special inequalities between terms of perfect algebras.
Via duality, formulas in this language can be equivalently translated into first order conditions in the frame correspondence languages of several types of relational semantics for LE-logics.
\end{abstract}

\setlength{\abovedisplayskip}{1.7mm}
\setlength{\belowdisplayskip}{1.7mm}
\setlength{\abovedisplayshortskip}{1.7mm}
\setlength{\belowdisplayshortskip}{1.7mm}

\section{Introduction}
\label{sec:introduction}
The present paper continues a line of investigation initiated in \cite{CoPan22} and aimed at uniformly generalizing Kracht's theory of internal describability \cite{krachtphdthesis} from classical modal logic to large families of nonclassical logics, such as the logics canonically associated with varieties of normal lattice expansions \cite{gehrke2001bounded}, which are sometimes referred to as {\em non-distributive logics} or {\em LE-logics}. 

In \cite{krachtphdthesis}, Kracht characterized  the syntactic shape of a class of formulas in the first order language of Kripke frames, each of which is the first-order correspondent of some  Sahlqvist formula  of classical modal logic. This remarkable and very well known result, which can be understood as the `converse of Sahlqvist correspondence theory' (cf.~\cite[Section 3.7]{blackburn2002modal}),  does not have many follow-up results in the literature (with the notable exception of Kikot's \cite{kikot}, which extends Kracht's result from Sahlqvist formulas to  {\em inductive formulas} \cite{GoVa2006} of classical modal logic). In particular, to our knowledge, no generalizations exist of Kracht's theory to any specific non-classical logic language, let alone to families thereof, unlike   Sahlqvist correspondence theory, which has given rise to a rich literature, especially investing non-classical frameworks \cite{holliday-possibility,kurtonina,britzconradiemorton,conradierobinsonhybrid,vanBenthem2012-VANSCF,celanijansana_sahlqvistpositivemodallogic}. Moreover,  a shift from a model-theoretic to an algebraic perspective on the mechanism of Sahlqvist correspondence made  it possible to  define the class of Sahlqvist and inductive formulas/inequalities uniformly across  a broad range of logical languages, based on the order-theoretic properties of the algebraic interpretations of the logical connectives in each language, and to extend the algorithm SQEMA, which computes the first order correspondents of  inductive formulas of  classical modal logic \cite{CoGoVa2006}, to the algorithm ALBA \cite{CoPa:distr,CoPa:non-dist}, performing the same task as SQEMA for a broad spectrum of nonclassical languages which includes the LE-logics, 
and their expansions with fixed points \cite{CoGhPa14,conradie2017constructive}.
These results make an overarching  theory referred to as {\em unified correspondence} \cite{CoGhPa14}.
The lack of extant results generalizing Kracht's theory to nonclassical logics is likely due to the specific challenges  they present, one  of them being that non-classical logics do not typically have one single established relational semantics, but rather, many different relational semantic frameworks exist for the same logic, each of which is associated with significantly different first order languages. 

The present paper aims at uniformly characterizing the syntactic  shape of the first order correspondents of inductive inequalities in arbitrary LE-languages, 
and at introducing an algorithm for computing the LE-axiom to which each given first order condition corresponds. The uniformity we seek is relative to {\em every} type of relational semantics of  {\em every} LE-logic (see Definition \ref{def:lelogic}). Achieving such uniformity would justify the present theory to be regarded as the `converse' of unified correspondence for LE-logics \cite{CoPa:non-dist}, and hence to be referred to as {\em unified inverse correspondence} for LE-logics.

Rather than following Kikot's proof strategy, which hinges on finding the appropriate relaxation of the conditions on Kracht's formulas so as to cover the wider class of inductive formulas, 
the proof strategy of the present contribution hinges on the  syntactic characterization of the first order correspondents of a certain {\em subclass} of very simple Sahlqvist inequalities in the given LE-language expanded with the residuals of all connectives (cf.\ Definition \ref{def:Crypto:inductive}).

\paragraph{Structure of the paper.} In Section \ref{sec:starting_example}, we illustrate the procedure for computing the LE-axiom associated with a given first order condition with an example.
In Section \ref{sec:preliminaries} we recall notions about lattice expansion logics, canonical extensions of lattice expansions, ALBA's language, (refined) inductive inequalities, and geometric implications. 
In Section \ref{ssec:ALBAcorr}, we discuss an alternative presentation of the algorithm ALBA in Section \ref{ssec:albadirect}, and we use it to show that ALBA outputs are (equivalent to) generalized geometric implications in Section \ref{ssec:flattening}. 
In Section \ref{sec:inverse} we present an algorithm for inverse correspondence that targets inductive inequalities. Specifically, in Section \ref{ssec:inductiveinversecorrespondents}, we introduce an inverse correspondence procedure which transforms certain first order formulas into very simple Sahlqvist formulas in the language expanded with the residuals of all non-propositional connectives in each coordinate (i.e.~this language is the counterpart, in the context of LE-languages, of  the tense language for classical modal logic). 
Then, in Section \ref{sec:cryptoinductive}, we characterize a proper subclass of very simple Sahlqvist formulas in the expanded language each of which is equivalent to an inductive inequalities in the original LE-language. 
In Section \ref{sec:instances}, we show how the general algebraic definitions can be instantiated to a few semantic contexts, and we provide some additional examples in Section \ref{sec:moreexamples}, including an example that deals with heterogeneous (many-sorted) lattice expansions. We conclude in Section \ref{sec:conclusions}.

\section{Informal illustration of the `inverse correspondence' procedure}
\label{sec:starting_example}

The strategy enabling a uniform approach to correspondence for arbitrary languages hinges on a {\em two-step process} in which (a) the elimination of the monadic second order quantification is effected in an algebraic environment, thanks to the order-theoretic properties of the algebraic interpretation of the logical connectives; after this elimination, (b) conditions in this algebraic environment are translated into the first order correspondents in any frame correspondence language via discrete duality (cf.\ \cite[Section 2.2]{CoPa:non-dist}).

The strategy adopted in this paper to define a unified inverse correspondence algorithm consists in taking the converse route to the steps ($\mathrm{a}$) and ($\mathrm{b}$) described above: firstly, $(\mathrm{b}')$ a condition in the first order language of a certain relational structure is equivalently rewritten in the extended language of ALBA $\langalba$, which is an algebraic language extending a given LE-language $\langbase$ and is naturally interpreted in the complex of algebras of the given relational structures (see Definition \ref{def:albametalang}). Then, ($\mathrm{a}'$) the translated condition in $\langalba$ is processed and is equivalently rewritten as an $\langbase$-axiom. In particular, step ($\mathrm{a}'_1$) can be subdivided in two sub-steps. Firstly, ($\mathrm{a}'_1$) the translated condition in $\langalba$ is equivalently transformed into a very simple Sahlqvist $\langbase^\ast$-inequality of a suitable syntactic shape (cf.\ Definition \ref{def:Crypto:inductive}), where $\langbase^\ast$ is the extension of $\langbase$ with the residuals of all connectives in each coordinate. Secondly, ($\mathrm{a}'_2$) this $\langbase^\ast$-inequality is equivalently transformed into an inductive $\langbase$-inequality.

Since, as discussed above, the frame correspondence languages associated with various semantics can be significantly different, the suitable syntactic shape is not characterized in terms of any of these languages, but is rather characterized in $\langalba$. Hence, the generalization of Kracht formulas is defined in $\langalba$, and it can be extended to the language of any specific semantics by characterizing the formulas in that language which equivalently translate to this shape $\langalba$.

In what follows, we  illustrate these two steps with examples. For the sake of simplicity, we set the LE-language $\mathcal{L}_{\mathrm{LE}}$ to be the $\{\backslash\}$-fragment of the language of the basic full Lambek calculus \cite{GaJiKoOn07}, and consider the well known  semantics of $\mathcal{L}_{\mathrm{LE}}$ given by the class of magmas $\mathbb{F} = (W, \cdot)$, i.e., sets endowed with a binary operation. As is well known, any such $\mathbb{F}$ gives rise to the complex $\mathcal{L}_{\mathrm{LE}}^\ast$-algebra $\mathbb{F}^+ = (\mathcal{P}(W), \circ, \backslash, \slash)$, where $\mathcal{L}_{\mathrm{LE}}^\ast$ denotes the expansion of $\mathcal{L}_{\mathrm{LE}}$ with the residuals of $\backslash$ in each coordinate, $\mathcal{P}(W)$ denotes the powerset of $W$, and for all $A, B \subseteq W$, 
%
\[A \circ B = \{ a \cdot b : a \in A, b \in B\},\quad A\backslash B=\{u\mid \forall a(a\in A\Rightarrow a\cdot u\in B)\},\quad A/B=\{u\mid \forall b( b\in B\Rightarrow u\cdot b\in A)\}.\]
Consider the  following equation  holding on a magma $\mathbb{F} = (W, \cdot)$: 
\begin{equation}
\label{eqn:informal_discussion:startingeq}
w \cdot (x \cdot z) = (w \cdot x) \cdot (w \cdot z).
\end{equation}  
Since, in this context, $\mathbb{F}$ plays the role  of a Kripke frame, it makes sense to equivalently rewrite the equation above  as the following condition in the frame correspondence language of the class of magmas:
\begin{equation}
\label{eqn:informal_discussion:firstordermagma}
    \forall x,y,z,w,t \;\left(
y = x \cdot z\metaand t = w \cdot y \Rightarrow \exists w_1\exists w_2 ( t = w_1 \cdot w_2 \metaand w_1 = w \cdot x \metaand w_2 = w \cdot z )
\right)
\end{equation}
The condition above can be translated as a condition in the (extended) language of the {\em perfect} residuated algebra $\mathbb{F}^+$ as follows: For any variable $u$ ranging over $W$, let  $\nomj_u\coloneqq \{u\}$ and $\cnomm_u\coloneqq W \setminus \{u\}$. Hence, for any $x, y, z\in W$,
\begin{equation}    \label{eqn:informal_discussion:useful_equivalences}
x = y\cdot z 
\quad \text{ iff } \quad
\nomj_x \leq \nomj_y \circ \nomj_z 
\quad \text{ iff } \quad
\nomj_y \circ \nomj_z \nleq \cnomm_x   
\quad \text{ iff } \quad
\nomj_z \nleq \nomj_y \backslash \cnomm_x   
\quad \text{ iff } \quad
\nomj_y \backslash \cnomm_x \leq \cnomm_z,
\end{equation}
and analogous chains of equivalences can be obtained by residuating on the first coordinate of $\circ$. Therefore, \eqref{eqn:informal_discussion:firstordermagma} can  equivalently be rewritten as follows (thus completing step (b') described above):
\[
\forall x,y,z,w,t \;\left(
\nomj_x \backslash \cnomm_y \leq \cnomm_z
\metaand
\nomj_w \backslash \cnomm_t \leq \cnomm_y
\Rightarrow
\exists w_1\exists w_2 (
\nomj_t \leq \nomj_{w_1} \circ \nomj_{w_2}
\metaand
\nomj_{w_1} \leq \nomj_{w} \circ \nomj_x
\metaand
\nomj_{w_2} \leq \nomj_w \circ \nomj_z
)
\right),
\]
which is readily shown to be equivalent to
\begin{equation}
\label{eqn:informal_discussion:firstordermagma_dle_compact}
\forall x,z,w,t \;\left(
\nomj_x \backslash (\nomj_w \backslash \cnomm_t) \leq \cnomm_z
\Rightarrow
\nomj_t \leq (\nomj_{w} \circ \nomj_x)\circ(\nomj_w \circ \nomj_z)
\right),
\end{equation}
using the fact that atoms (resp.~coatoms) in $\mathbb{F}^+$ are completely join-prime (resp.\ meet-prime) and join-dense (resp.\ meet-dense) in $\mathbb{F}^+$, and moreover $\backslash$ is completely meet-preserving in its second coordinate, and $\circ$ is completely join-preserving in both  coordinates.\footnote{For instance, let us show how to rewrite $\nomj_x \backslash (\nomj_w \backslash \cnomm_t) \leq \cnomm_z$. Since coatoms are meet-dense in $\mathbb{F}^+$, $\nomj_x \backslash (\nomj_w \backslash \cnomm_t) \leq \cnomm_z$ holds iff $\nomj_x \backslash \bigwedge \{ \cnomm_y : (\nomj_w \backslash \cnomm_t) \leq \cnomm_y \} \leq \cnomm_z$, which, by the fact that $\backslash$ is completely meet preserving in its second coordinate, is equivalent to $\bigwedge \{\nomj_x \backslash \cnomm_y : (\nomj_w \backslash \cnomm_t) \leq \cnomm_y \} \leq \cnomm_z$. Since $\cnomm_z$ is completely meet prime, the inequality is equivalent to $\exists y((\nomj_w \backslash \cnomm_t) \leq \cnomm_y \metaand \nomj_x \backslash \cnomm_y \leq \cnomm_y)$.}
Consider now the following condition, which is  obtained by taking the 
contrapositive of \eqref{eqn:informal_discussion:firstordermagma_dle_compact} and then applying the equivalences in \eqref{eqn:informal_discussion:useful_equivalences}: 
\[
\forall x,z,w,t \;\left(
(\nomj_{w} \circ \nomj_x)\circ(\nomj_w \circ \nomj_z) \leq \cnomm_t
\Rightarrow
\nomj_z \leq \nomj_x \backslash (\nomj_w \backslash \cnomm_t)
\right).
\]
Using the properties mentioned above, one readily shows that the condition above is equivalent to the following inequality, in which the variable $\cnomm_t$ has been eliminated:
\begin{equation}
\label{eqn:informal_discussion:alba_output}
\forall x,z,w \;\;\;
\nomj_z \leq \nomj_x \backslash (\nomj_w \backslash ((\nomj_{w} \circ \nomj_x)\circ(\nomj_w \circ \nomj_z))).
\end{equation}
This inequality can be recognized as the ALBA-output of an inequality pertaining to a specific subclass of Sahlqvist $\mathcal{L}_{\mathrm{LE}}^\ast$-inequalities (the {\em scattered very simple Sahlqvist $\mathcal{L}_{\mathrm{LE}}^\ast$-inequalities}, cf.\ Definition \ref{def:sahlqvist_and_vss}). By the general theory (cf.~Lemma \ref{lemma:vss_substitute_nomcnom_var}), the ALBA output of any scattered very simple 
Sahlqvist inequality 
can be obtained by replacing every propositional variable 
by a nominal or a conominal variable. Hence, condition \eqref{eqn:informal_discussion:alba_output} is equivalent to the following scattered very simple Sahlqvist $\mathcal{L}_{\mathrm{LE}}^\ast$-inequality holding in $\mathbb{F}^+$:
\begin{equation}
\label{eqn:informal_discussion:lstarvss}
\forall p,q,r \;\;\;
p \leq q \backslash (r \backslash ((r \circ q)\circ(r \circ p))),
\end{equation}
Using ALBA-rules, \eqref{eqn:informal_discussion:lstarvss} can be further processed  so as to equivalently rewrite it as a proper inductive $\langbase$-inequality. 
\smallskip

{{\centering
\begin{tabular}{rrl}
&$
\forall p,q,r$&$
\left(
p \leq q \backslash (r \backslash ((r \circ q)\circ(r \circ p)))
\right)
$\\
iff&$
\forall p,q,r,a$&$
\left(
r \leq a
\Rightarrow
p \leq q \backslash (r \backslash ((a \circ q)\circ(a \circ p)))
\right)
$\\
iff &$
\forall p,q,r,a,b$&$
\left(
r \leq a
\metaand
a \circ q \leq b
\Rightarrow
p \leq q \backslash (r \backslash (b\circ(a \circ p)))
\right)
$\\
iff &$
\forall p,q,r,a,b,c$&$
\left(
r \leq a
\metaand
a \circ q \leq b
\metaand
b\circ(a \circ p) \leq c
\Rightarrow
p \leq q \backslash (r \backslash c)
\right)
$ \\ 
iff &$
\forall p,q,r,a,b,c$&$
\left(
r \leq a
\metaand
q \leq a \backslash b
\metaand
p \leq a \backslash (b\backslash c)
\Rightarrow
p \leq q \backslash (r \backslash c)
\right)
$ \\ 
iff &$
\forall p,q,a,b,c$&$
\left(
q \leq a \backslash b
\metaand
p \leq a \backslash (b\backslash c)
\Rightarrow
p \leq q \backslash (a \backslash c)
\right)
$ \\
iff &$
\forall p,a,b,c$&$
\left(
p \leq a \backslash (b\backslash c)
\Rightarrow
p \leq (a \backslash b) \backslash (a \backslash c)
\right)
$ \\
iff &$
\forall a,b,c$&$
a \backslash (b\backslash c) \leq (a \backslash b) \backslash (a \backslash c),
$ 
\end{tabular}
\par}}
\smallskip
\noindent which completes step (a') described above. Notice that the $\langbase$-inequality $a \backslash (b\backslash c) \leq (a \backslash b) \backslash (a \backslash c)$ is a proper inductive inequality which is not Sahlqvist for any order-type.

The syntactic shape of \eqref{eqn:informal_discussion:lstarvss} is further characterized in Definition \ref{def:Crypto:inductive}, as a class of very simple Sahlqvist $\langbase^\ast$-inequalities which can be equivalently rewritten, via ALBA-rules \cite[Section 4]{CoPa:non-dist}, as $\langbase$-inductive inequalities.



\section{Preliminaries}
\label{sec:preliminaries}

The present section collects definitions and basic facts  introduced in  \cite{CoPa:non-dist,inductivepolynomial,simpson1994,negri2016sysrul,johnstone:elephant}. 

\subsection{Lattice expansion logics}
\label{ssec:LElogics}
The framework of (normal) {\em LE-logics} (those logics canonically associated with varieties of normal  general---i.e.~not necessarily distributive---lattice expansions \cite{gehrke2001bounded, CoPa:non-dist}, and sometimes referred to as {\em non-distributive logics}) has been introduced as an overarching environment for developing duality, correspondence, and canonicity results for a wide class of logical systems covering the best known and most widely applied  nonclassical logics in connection with one another (cf.~Example \ref{example:LElangs} below); however, the resulting framework turned out to be very important also for studying the proof theory of these logics \cite{greco2018unified,chen2021syntactic}, precisely thanks to the systematic connections between proof-theoretic properties and correspondence-theoretic properties (more about this in the conclusions). The exposition in the present subsection is based on \cite[Section 1]{CoPa:non-dist},  \cite[Section 3]{conradie2020non},
\cite[Sections 2 and 3]{greco2018algebraic}, to which we refer the interested reader for more detail. 

\begin{definition}[Order type]
\label{def:ordertype}
An {\em order-type} over $n\in \mathbb{N}$ is an $n$-tuple $\varepsilon\in \{1, \partial\}^n$. For every order-type $\varepsilon$, we denote its {\em opposite} order-type by $\varepsilon^\partial$, that is, $\varepsilon^\partial_i = 1$ iff $\varepsilon_i=\partial$ for every $1 \leq i \leq n$. For any lattice $\mathbb{A}$, we let $\mathbb{A}^1: = \mathbb{A}$ and $\mathbb{A}^\partial$ be the dual lattice, that is, the lattice associated with the converse partial order of $\mathbb{A}$. For any order-type $\varepsilon$ over $n$, we let $\mathbb{A}^\varepsilon: = \Pi_{i = 1}^n \mathbb{A}^{\varepsilon_i}$. For a given $s \in \{1, \partial\}$, we identify $-s$ with $1$ if $s=\partial$, and with $\partial$ otherwise.
\end{definition}

\begin{definition}[LE-language]
\label{def:lelang}
The language $\mathcal{L}_\mathrm{LE}(\mathcal{F}, \mathcal{G})$ (sometimes abbreviated as $\mathcal{L}_\mathrm{LE}$) takes as parameters a denumerable set $\atprop$ of proposition letters, and two sets of connectives $\mathcal{F}$ and $\mathcal{G}$.  Each $f\in \mathcal{F}$ and $g\in \mathcal{G}$ has arity $n_f\in \mathbb{N}$ (resp.\ $n_g\in \mathbb{N}$) and is associated with some order type $\varepsilon_f$ over $n_f$ (resp.\ $\varepsilon_g$ over $n_g$).\footnote{Unary $f$ (resp.\ $g$) connectives are typically denoted  $\Diamond$ (resp.\ $\Box$) if their order type is 1, and $\lhd$ (resp.\ $\rhd$) if their order type is $\partial$.}
The terms (formulas) of $\mathcal{L}_\mathrm{LE}$ are defined recursively as follows:
\[
\phi ::= p \mid \bot \mid \top \mid (\phi \wedge \phi) \mid (\phi \vee \phi) \mid f(\overline{\phi}) \mid g(\overline{\phi})
\]
where $p \in \atprop$, $f \in \mathcal{F}$, $g \in \mathcal{G}$. 
 
We will follow the standard rules for the elimination of
parentheses. Terms in $\mathcal{L}_\mathrm{LE}$ will be denoted either by $s,t$, or by lowercase Greek letters such as $\varphi, \psi, \gamma$ etc. 
The purpose of grouping LE-connectives in the families $\mathcal{F}$ and $\mathcal{G}$ is to identify -- and refer to -- the two types of order-theoretic behaviour relevant for the development of this theory. Roughly speaking, connectives in $\mathcal{F}$  and in $\mathcal{G}$ can be thought of as the logical counterparts of generalized operators, and of generalized dual operators, respectively. We adopt the convention that unary connectives bind more strongly than the others.
\end{definition}

\begin{example}
\label{example:LElangs}
Several examples of LE languages are listed in the following table
\smallskip

{{\centering
\begin{tabular}{|ll|l|l|l|}
\hline
Logic & Abbr.& \ $\mathcal{F}$-oper. & \ $\mathcal{G}$-oper. & Refs \\
\hline
Classical modal logic & $\mathcal{L}_{\mathrm{CML}}$    & $\wedge \ \Diamond \ \neg \ \top$ & $\vee \ \rightarrow \ \neg \ \Box \ \bot$ & \\
Intuitionistic logic & $\mathcal{L}_{\mathrm{IL}}$      & $\wedge \ \top$ & $\vee \ \rightarrow \ \bot$ & \\
Bi-intuitionistic logic & $\mathcal{L}_{\mathrm{Bi}}$   & $\wedge \ \pdlab \ \top$ & $\vee \ \rightarrow \ \bot$ & \cite{rauszer1974formalization} \\
Distributive modal logic & $\mathcal{L}_{\mathrm{DML}}$ & $\wedge \ \Diamond \ \lhd \ \top$ & $\vee \ \Box \ \rhd \  \bot$ & \cite{GNV,CoPa:non-dist} \\
Positive modal logic & $\mathcal{L}_{\mathrm{PML}}$     & $\wedge \ \Diamond \ \top$ & $\vee \ \Box \ \bot$ & \cite{Dunn:Pos:ML}\\
Bi-intuitionistic modal logic \quad & $\mathcal{L}_{\mathrm{BIML}}$ & $\wedge \ \pdlab \ \Diamond \ \top$ & $\vee \ \rightarrow \ \Box \ \bot$ & \cite{wolter1998CoImplication}\\
Intuitionistic modal logic & $\mathcal{L}_{\mathrm{IML}}$ & $\wedge \ \Diamond \ \top$ & $\vee \ \rightarrow \ \Box \ \bot$ & \cite{fischerservi1977modal}\\
Semi-De Morgan logic & $\mathcal{L}_{\mathrm{SDM}}$ & $\wedge \ \top$ & $\vee \ \rhd \ \bot$ & \cite{sankappanavar1987semi}\\
\hline
Orthologic & $\mathcal{L}_{\mathrm{OL}}$ & $\top$ & $\rhd \ \bot$ & \cite{Goldblatt:Ortho:74}\\
Full Lambek calculus & $\mathcal{L}_{\mathrm{FL}}$ & $\circ \ 1$ & $\backslash \ \slash $& \cite{la61,GaJiKoOn07}\\
Lambek-Grishin calculus & $\mathcal{L}_{\mathrm{LG}}$ & $\circ \ \starback \ \starfor \ 1$ & $\star \ \circback \ \circfor \ \upsE$ & \cite{Moortgat}\\
Mult.-Add.\ linear logic & $\mathcal{L}_{\mathrm{MALL}}$ & $\circ \ 1$ & $ \circback \ \upsE$ & \cite{GaJiKoOn07}\\
\hline
\end{tabular} \par}} 
\smallskip

\noindent where the order type of the connectives is specified in the following tables

{{\centering
\begin{tabular}{|c|c|}
\hline
Symbol & \quad $\varepsilon_h$ \quad \\
\hline
$\top \ \bot \ 1 \ \upsE$ & \quad () \quad \\
$\Diamond \ \Box$ & \quad (1) \quad \\
\quad $\rhd \ \lhd \ \neg$ \quad & \quad  ($\partial$)  \quad\\
\hline
\end{tabular}
\quad\quad
\begin{tabular}{|c|c|}
\hline
Symbol & \quad $\varepsilon_h$ \quad \\
\hline
$\wedge \ \circ \ \vee \ \star$ & \quad $(1, 1)$ \quad \\
$\circback \rightarrow \starback$ & \quad $(\partial, 1$) \quad \\
\quad $\circfor \pdlab \starfor$ \quad & \quad $(1, \partial)$ \quad \\
\hline
\end{tabular}
\par
}}\end{example}

\begin{definition}[Lattice expansion]
For any pair of families of connectives $\mathcal{F}$ and $\mathcal{G}$ as above, a {\em lattice expansion} is a triple $\mathbb{A} = (\mathbb{L}, \mathcal{F}^{\mathbb{A}}, \mathcal{G}^{\mathbb{A}})$ such that $\mathbb{L}$ is a bounded lattice, $\mathcal{F}^{\mathbb{A}} = \{f^{\mathbb{A}} : f \in \mathcal{F} \}$ and $\mathcal{G}^{\mathbb{A}} = \{g^{\mathbb{A}} : g \in \mathcal{G} \}$, such that every $f^\mathbb{A} \in \mathcal{F}^\mathbb{A}$ (resp.\ $g^\mathbb{A} \in \mathcal{G}^\mathbb{A}$) is a $n_f$ (resp.\ $n_g$)-ary operation on $\mathbb{A}$. A lattice expansion is {\em normal} if all the operations ${f}^\mathbb{A} \in \mathcal{F}^\mathbb{A}$ and ${g}^\mathbb{A} \in \mathcal{G}^\mathbb{A}$ are finitely join and meet preserving respectively, i.e. for any $a, b \in \mathbb{A} ^{\varepsilon_{f}}$ and $c, d \in \mathbb{A} ^{\varepsilon_{g}}$, such that $a$ and $b$ (resp.\ $c$ and $d$) differ in only one coordinate,
\[
\begin{array}{lcl}
f^\mathbb{A}(a) = 
\bot \quad \mbox{if some } a_i = \bot^{\varepsilon_f(i)} 
& \mbox{and} &
g^\mathbb{A}(c) = 
\top \quad \mbox{if some } c_i = \top^{\varepsilon_g(i)} \\
f^\mathbb{A}(a \vee^{\varepsilon_f} b) = 
f^\mathbb{A}(a) \vee f^\mathbb{A}(b) 
& \mbox{and} &
g^\mathbb{A}(c \wedge^{\varepsilon_g} d) = 
g^\mathbb{A}(c) \wedge g^\mathbb{A}(d),
\end{array}
\]
where $\bot = \bot^1 = \top^\partial$ and $\top = \top^1 = \bot^\partial$, and $\vee^{\varepsilon_h}$ and $\wedge^{\varepsilon_h}$ are the meet and the join of the lattice $\mathbb{A}^{\varepsilon_h}$ for some $h \in \mathcal{F}\cup \mathcal{G}$. In the remainder of the paper, we will also denote $\wedge^1 = \wedge = \vee^\partial$, and $\wedge^\partial = \vee = \vee^1$.
\end{definition}

In the present paper, all the lattice expansions we consider are normal; hence will often drop the adjective `normal'. The class $\leclass$ of normal lattice expansions (for a fixed LE-language) is equational, and thus forms a variety. Lattice-based logics of this kind are usually not expressive enough to build implication-like terms which encode the entailment relation. Hence, since entailment cannot be recovered from the tautologies, it is necessary to take the whole relation as definition for these logics.

\begin{definition}[LE logic]
\label{def:lelogic}
For any $\mathcal{L}_{\mathrm{LE}}(\mathcal{F},\mathcal{G})$, an $\mathcal{L}_{\mathrm{LE}}$-logic is a set of sequents $\varphi\vdash\psi$, with $\varphi,\psi\in\mathcal{L}_{\mathrm{LE}}$, containing the axioms encoding the lattice structure
\[
p\vdash p, \quad 
\bot\vdash p, \quad 
p\vdash \top, \quad
p\vdash p\vee q, \quad
q\vdash p\vee q, \quad
p\wedge q\vdash p, \quad
p\wedge q\vdash q, 
\]
and for each $f \in \mathcal{F}$ and $g\in\mathcal{G}$ the axioms encoding normality
\[
\begin{array}{lr}
f(p_1,\ldots, \bot^{\varepsilon_f(i)},\ldots,p_{n_f}) \vdash \bot,\quad\quad\quad\quad
\top\vdash g(p_1,\ldots, \top^{\varepsilon_g(i)},\ldots,p_{n_g}),\\
f(p_1,\ldots, p\vee^{\varepsilon_f(i)} q,\ldots,p_{n_f}) \vdash f(p_1,\ldots, p,\ldots,p_{n_f})\vee f(p_1,\ldots, q,\ldots,p_{n_f}),\\
g(p_1,\ldots, p,\ldots,p_{n_g})\wedge g(p_1,\ldots, q,\ldots,p_{n_g})\vdash g(p_1,\ldots, p\wedge^{\varepsilon_g(i)} q,\ldots,p_{n_g}),\\
\end{array}
\]
and is closed under the following inference rules
\[
\frac{\phi\vdash \chi\quad \chi\vdash \psi}{\phi\vdash \psi}
\quad
\frac{\phi\vdash \psi}{\phi(\chi/p)\vdash\psi(\chi/p)}
\quad
\frac{\chi\vdash\phi\quad \chi\vdash\psi}{\chi\vdash \phi\wedge\psi}
\quad
\frac{\phi\vdash\chi\quad \psi\vdash\chi}{\phi\vee\psi\vdash\chi}
\]
where $\phi(\chi/p)$ denotes uniform substitution of $\chi$ for $p$ in $\phi$,  and for each connective $f \in \mathcal{F}$ and $g \in \mathcal{G}$ (in this context, $\vdash^1$ is $\vdash$ and $\vdash^\partial$ is $\dashv$),
\[
\begin{array}{l}
     \dfrac{\phi\vdash\psi}{f(\phi_1,\ldots,\phi,\ldots,\phi_{n_f})\vdash^{\varepsilon_f(i)} f(\phi_1,\ldots,\psi,\ldots,\phi_{n_f})} \quad\quad 
     \dfrac{\phi\vdash\psi}{g(\phi_1,\ldots,\phi,\ldots,\phi_{n_g})\vdash^{\varepsilon_g(i)} g(\phi_1,\ldots,\psi,\ldots,\phi_{n_g})}
\end{array}
\]
The minimal $\mathcal{L}_{\mathrm{LE}}(\mathcal{F}, \mathcal{G})$-logic is denoted by $\mathbf{L}_\mathrm{LE}(\mathcal{F}, \mathcal{G})$, or simply by $\mathbf{L}_\mathrm{LE}$ when $\mathcal{F}$ and $\mathcal{G}$ are clear from the context.
\end{definition}

For every LE $\mathbb{A}$, $\vdash$ is interpreted as $\leq$, and $\mathbb{A}\models \varphi \vdash \psi$ iff $h(\varphi)\leq h(\psi)$ for every homomorphism $h$ over the algebra of formulas over $\atprop$ to $\mathbb{A}$. The notation $\leclass \models \varphi \vdash \psi$ denotes that $\mathbb{A} \models \varphi \vdash \psi$ for every algebra $\mathbb{A}$ in $\leclass$. Since the minimal LE logics $\mathbf{L}_\mathrm{LE}$ are self-extensional (e.g. the interderivability is a congruence relation over the algebra of formulas), a standard Lindebaum-Tarski construction shows their completeness with respect the corresponding class of lattice expansions $\leclass$, namely $\varphi \vdash \psi \in \mathbf{L}_\mathrm{LE}$ iff $\leclass \models \varphi \vdash \psi$.

\begin{definition}[Fully residuated language]
\label{def:fully_residuated_language}
To any LE language $\mathcal{L}_\mathrm{LE}(\mathcal{F}, \mathcal{G})$, is associated the corresponding {\em fully residuated language} $\mathcal{L}^*_\mathrm{LE} = \mathcal{L}_\mathrm{LE}(\mathcal(F)^*,\mathcal{G}^*)$, where $\mathcal{F}^*$ (resp.\ $\mathcal{G}^*$) is obtained by expanding $\mathcal{F}$ (resp.\ $\mathcal{G}$) with connectives $f^\sharp_1,\ldots,f^\sharp_{n_f}$ (resp.\ $g^\flat_1,\ldots,g^\flat_{n_g}$) for each $f \in \mathcal{F}$ (resp.\ $g \in \mathcal{G}$), the intended meaning of which are the right adjoints of $f$ (resp.\ left adjoints of $g$).
\end{definition}

\begin{definition}
\label{def:lelogicresiduals}
For any LE-language $\mathcal{L}_\mathrm{LE}$, {\em the basic $\mathcal{L}_\mathrm{LE}$-logic with residuals} is obtained by specializing Definition \ref{def:lelogic} to $\mathcal{L}_\mathrm{LE}$ and closing under the following rules for any $f \in \mathcal{F}$ and $g \in \mathcal{G}$,
\smallskip

{{\centering
\begin{tabular}{cc}
\AxiomC{$f(\varphi_1,\ldots,\phi,\ldots, \varphi_{n_f}) \vdash \psi$}
\doubleLine
\LeftLabel{}
\UnaryInfC{$\phi\vdash^{\varepsilon_f(i)} f^\sharp_i(\varphi_1,\ldots,\psi,\ldots,\varphi_{n_f})$}
\DisplayProof
&
\AxiomC{$\phi \vdash g(\varphi_1,\ldots,\psi,\ldots,\varphi_{n_g})$}
\doubleLine
\RightLabel{}
\UnaryInfC{$g^\flat_i(\varphi_1,\ldots, \phi,\ldots, \varphi_{n_g})\vdash^{\varepsilon_f(i)} \psi$}
\DisplayProof
\end{tabular}
\par}}

\noindent The double line in the rules means that they can be read both from top to bottom and from bottom to top. Let $\mathbf{L}_\mathrm{LE}^*$ be the the minimal $\mathcal{L}_\mathrm{LE}$ logic with residuals.
\end{definition}
\subsection{Canonical extensions of lattice expansions}
\label{ssec:lecanonicalextensions}

The J\'onsson-Tarski expansion of Stone representation theorem proves that every Boolean algebra $\mathbb{A}$ with operators canonically embeds into the complex algebra of its ultrafilter frame, which is also called the {\em canonical extension} of $\mathbb{A}$. Such a canonical extension has certain algebraic properties which allow to characterize it up to isomorphism; hence, it is possible to define canonical extensions in purely algebraic terms as follows.

\begin{definition}
Let $\mathbb{A}$ be a lattice. The canonical extension $\mathbb{A}^\delta$ of $\mathbb{A}$ is a complete lattice which contains $\mathbb{A}$ and such that:
\begin{enumerate}
\item $\mathbb{A}$ is {\em dense} in $\mathbb{A}^\delta$, i.e., every element of $\mathbb{A}^\delta$ is the meet of joins and join of meets of elements of $\mathbb{A}$.\footnote{For this reason, sometimes canonical extensions are referred to as $\Delta_1$ completions.} 
\item $\mathbb{A}$ is {\em compact} in $\mathbb{A}^\delta$, i.e., for all $S, T \subseteq \mathbb{A}$, if $\bigwedge S\leq \bigvee T$ then $\bigwedge S'\leq \bigvee T'$ for some finite $S'\subseteq S$, $T'\subseteq T$.
\end{enumerate}
\end{definition}

Given a lattice $\mathbb{A}$, its canonical extension, besides being unique up to an isomorphism fixing $\mathbb{A}$, always exists (cf.\ \cite[Propositions 2.6 and 2.7]{GH01}), and is a perfect lattice \cite[Corollary 2.10]{DGP}, i.e., $\mathbb{A}^\delta$ is both completely join-generated by the set $J^{\infty}(\mathbb{A}^\delta)$ of the completely
join-irreducible elements of $\mathbb{A}^\delta$, and completely meet-generated by the set $M^{\infty}(\mathbb{A}^\delta)$ of
the completely meet-irreducible elements of $\mathbb{A}^\delta$.\footnote{The proof of existence of the canonical extension is constructive, while the proof of perfectness requires the axiom of choice.}

The join (resp.\ meet) closure in $\mathbb{A}^\delta$ of $\mathbb{A}$ is denoted by $O(\mathbb{A}^\delta)$ (resp.\ $K(\mathbb{A}^\delta)$), and its elements are called the {\em open} (resp.\ {\em closed}) elements of $\mathbb{A}$. Since the density condition can be restated as $O(\mathbb{A}^\delta)$ being meet-dense and $K(\mathbb{A}^\delta)$ being join-dense, it immediately follows that $J^{\infty}(\mathbb{A}^{\delta}) \subseteq K(\mathbb{A}^\delta)$ and $M^{\infty}(\mathbb{A}^{\delta})\subseteq O(\mathbb{A}^\delta)$.

The canonical extension of an LE $\mathbb{A}$ will be defined as a suitable expansion of the canonical extension of the underlying lattice of $\mathbb{A}$, which also extends the operators to $\mathbb{A}^\delta$.
Since,
${(\mathbb{A}^\partial)}^\delta = {(\mathbb{A}^{\delta})}^\partial$ and ${(\mathbb{A}_1\times \mathbb{A}_2)}^\delta = \mathbb{A}_1^\delta\times \mathbb{A}_2^\delta$ (cf.\ \cite[Theorem 2.8]{DGP}),
it follows that  ${(\mathbb{A}^\partial)}^\delta$ can be  identified with ${(\mathbb{A}^{\delta})}^\partial$,  ${(\mathbb{A}^n)}^\delta$ with ${(\mathbb{A}^{\delta})}^n$, and
${(\mathbb{A}^\varepsilon)}^\delta$ with ${(\mathbb{A}^{\delta})}^\varepsilon$ for any order type $\varepsilon$. Hence,
in order to extend operations of any arity which are monotone or antitone in each coordinate from a lattice $\mathbb{A}$ to its canonical extension, it is sufficient to provide the definition of such extensions only for {\em monotone} and {\em unary} operations:
\begin{definition}
For every unary, order-preserving operation $f : \mathbb{A} \to \mathbb{A}$, the $\sigma$-{\em extension} of $f$ is defined firstly by declaring, for every $k\in K(\mathbb{A}^{\delta})$ and every $u \in \mathbb{A}^\delta$
\[
f^\sigma(k):= \bigwedge\{ f(a)\mid a\in \mathbb{A}\mbox{ and } k\leq a\}, 
\quad\quad\quad \mbox{and} \quad\quad\quad 
f^\sigma(u):= \bigvee\{ f^\sigma(k)\mid k\in K(\mathbb{A}^{\delta})\mbox{ and } k\leq u\}
\]
The $\pi$-{\em extension} of $f$ is defined firstly by declaring, for every $o\in O(\mathbb{A}^{\delta})$ and every $u\in \mathbb{A}^{\delta}$,
\[
f^\pi(o):= \bigvee\{ f(a)\mid a\in \mathbb{A}\mbox{ and } a\leq o\},
\quad\quad\quad \mbox{and} \quad\quad\quad
f^\pi(u):= \bigwedge\{ f^\pi(o)\mid o\in O(\mathbb{A}^{\delta})\mbox{ and } u\leq o\}.
\]
\end{definition}
It is easy to see that the $\sigma$- and $\pi$-extensions of $\varepsilon$-monotone maps are $\varepsilon$-monotone. More remarkably,  the $\sigma$-extension (resp.\ $\pi$-extension) of a map which sends (finite) joins or meets in the domain to (finite) joins (resp.\ meets) in the
codomain sends {\em arbitrary} joins
or meets in the domain to {\em arbitrary} joins (resp.\ meets) in the codomain.

\begin{definition}
The canonical extension of an
$\mathcal{L}_\mathrm{LE}$-algebra $\mathbb{A} = (L, \mathcal{F}^\mathbb{A}, \mathcal{G}^\mathbb{A})$ is the   $\mathcal{L}_\mathrm{LE}$-algebra
$\mathbb{A}^\delta: = (L^\delta, \mathcal{F}^{\mathbb{A}^\delta}, \mathcal{G}^{\mathbb{A}^\delta})$ such that $f^{\mathbb{A}^\delta}$ and $g^{\mathbb{A}^\delta}$ are defined as the
$\sigma$-extension of $f^{\mathbb{A}}$ and as the $\pi$-extension of $g^{\mathbb{A}}$ respectively, for all $f\in \mathcal{F}$ and $g\in \mathcal{G}$.
\end{definition}
The canonical extension of an LE $\mathbb{A}$ can be shown to be a {\em perfect} LE:
\begin{definition}
\label{def:perfect LE}
An LE $\mathbb{A} = (L, \mathcal{F}^\mathbb{A}, \mathcal{G}^\mathbb{A})$ is perfect if $L$ is a perfect lattice, and moreover the following infinitary distribution laws are satisfied for each $f\in \mathcal{F}$, $g\in \mathcal{G}$, $1\leq i\leq n_f$ and $1\leq j\leq n_g$: for every $S\subseteq L$,

{{\centering
\begin{tabular}{c  }
$f(x_1,\ldots, \bigvee^{\varepsilon_f(i)} S, \ldots, x_{n_f}) =\bigvee \{ f(x_1,\ldots, x, \ldots, x_{n_f}) \mid x\in S \}$ \\[1mm]
$g(x_1,\ldots, \bigwedge^{\varepsilon_g(i)} S, \ldots, x_{n_g}) =\bigwedge \{ g(x_1,\ldots, x, \ldots, x_{n_g}) \mid x\in S \}$ 
\end{tabular}
\par}}
\end{definition}
  
\noindent The algebraic completeness of the logics $\mathbf{L}_\mathrm{LE}$ and $\mathbf{L}_\mathrm{LE}^*$ and the embedding of LEs into their canonical extensions immediately give completeness of $\mathbf{L}_\mathrm{LE}$ and $\mathbf{L}_\mathrm{LE}^*$ w.r.t.\ the appropriate class of perfect LEs.

\subsection{Refined inductive inequalities}

We recall the definition of {\em refined inductive} and {\em Sahlqvist} $\langbase$-inequalities which were first introduced in \cite{inductivepolynomial} as a refinement of the definitions of inductive and Sahlqvist inequalities in \cite{CoPa:non-dist}.

\begin{definition}[\textbf{Signed Generation Tree}]
    \label{def: signed gen tree}
    The \emph{positive} (resp.\ \emph{negative}) {\em generation tree} of any $\mathcal{L}_\mathrm{LE}$-term $s$ is defined by labelling the root node of the generation tree of $s$ with the sign $+$ (resp.\ $-$), and then propagating the labelling on each remaining node as follows:
    \begin{itemize}
        \item For any node labelled with $ \lor$ or $\land$, assign the same sign to its children nodes.
        \item For any node labelled with $h\in \mathcal{F}\cup \mathcal{G}$ of arity $n_h\geq 1$, and for any $1\leq i\leq n_h$, assign the same (resp.\ the opposite) sign to its $i$th child node if $\varepsilon_h(i) = 1$ (resp.\ if $\varepsilon_h(i) = \partial$).
    \end{itemize}
    Nodes in signed generation trees are \emph{positive} (resp.\ \emph{negative}) if are signed $+$ (resp.\ $-$).
\end{definition}

Signed generation trees will be mostly used in the context of term inequalities $s\leq t$. In this context we will typically consider the positive generation tree $+s$ for the left-hand side and the negative one $-t$ for the right-hand side. We will also say that a term-inequality $s\leq t$ is \emph{uniform} in a given variable $p$ if all occurrences of $p$ in both $+s$ and $-t$ have the same sign, and that $s\leq t$ is $\varepsilon$-\emph{uniform} in a (sub)array $\overline{p}$ of its variables if $s\leq t$ is uniform in $p$, occurring with the sign indicated by $\varepsilon$, for every $p$ in $\overline{p}$.

For any term $s(p_1,\ldots p_n)$, any order type $\varepsilon$ over $n$, and any $1 \leq i \leq n$, an \emph{$\varepsilon$-critical node} in a signed generation tree of $s$ is a leaf node $+p_i$ with $\varepsilon_i = 1$ or $-p_i$ with $\varepsilon_i = \partial$. An $\varepsilon$-{\em critical branch} in the tree is a branch from an $\varepsilon$-critical node.

\begin{notation}
Given a set of variables $V = \{p_1,\ldots,p_n\}$ and an order type $\varepsilon$ over $n$, we often abuse notation and treat $\varepsilon$ as a function from $V$ to $\{1, \partial\}$, and so, for instance, we write $\varepsilon(p_i) = \partial$ to mean $\varepsilon_i = \partial$.
\end{notation}

\begin{definition}
    \label{def:good:branch}
    A branch in a signed generation tree $\ast s$, with $\ast \in \{+, - \}$, is a \emph{refined good branch} if it is the concatenation of three (possibly empty) paths $P_1$, $P_2$, and $P_3$, such that $P_1$ is a path from the leaf consisting (apart from variable nodes) only of $+g$ and $-f$ nodes, $P_2$ consists (apart from variable nodes) only of $+\wedge$ and $-\vee$ nodes, and $P_3$ consists (apart from variable nodes) only of $+f$ and $-g$ nodes. Non unary nodes in signed generation trees labelled with $+g$ or $-f$ are called \emph{SRR (Syntactically Right Residual)} nodes.
    In general, $+f$, $-g$, $+\vee$, and $-\wedge$ are referred to as {\em skeleton} nodes, while $-f$, $+g$, $+\wedge$, and $-\vee$ nodes are referred to as {\em PIA} nodes. A skeleton (resp.\ PIA) node which is not $+\vee$ or $-\wedge$ (resp.\ $+\wedge$ or $-\vee$) is a {\em definite} skeleton (resp.\ PIA) node; hence in a refined good branch $P_1$ consists of definite PIA nodes, $P_2$ of $+\wedge$ and $-\vee$ nodes, and $P_3$ of definite skeleton nodes.
\end{definition}

\begin{definition}[Refined inductive inequalities \cite{inductivepolynomial}]
    \label{Inducive:Ineq:Def}
    For any order type $\varepsilon$ and any irreflexive and transitive relation $<_\Omega$ on $p_1,\ldots p_n$, the signed generation tree $*s$ $(* \in \{-, + \})$ of a term $s(p_1,\ldots p_n)$ is \emph{refined $(\Omega, \varepsilon)$-inductive} if
    \begin{enumerate}
        \item every $\varepsilon$-critical branch is refined good (cf.\ Definition \ref{def:good:branch});
        \item every SRR node lying on  some $\varepsilon$-critical branch is ancestor of exactly one $\varepsilon$-critical leaf with variable $p$ (so it is contained in exactly one $\varepsilon$-critical branch), and all the other variables $q \neq p$ under its scope are such that $q <_\Omega p$. In what follows, we will sometimes refer to SRR nodes lying on   $\varepsilon$-critical branches as {\em SRR critical nodes}, and to branches which are not refined good as {\em bad branches}.
    \end{enumerate}
    
    We refer to $<_{\Omega}$ as the \emph{dependency order} on the variables. An inequality $s \leq t$ is \emph{refined $(\Omega, \varepsilon)$-inductive} if the signed generation trees $+s$ and $-t$ are refined $(\Omega, \varepsilon)$-inductive. An inequality $s \leq t$ is \emph{refined inductive} if it is refined $(\Omega, \varepsilon)$-inductive for some $\Omega$ and $\varepsilon$, and no variable occurs uniformly in it.
\end{definition}

\begin{remark}
\label{remark:refinedinductive}
In \cite{CoPa:non-dist}, the original definition of {\em inductive LE-inequality} is slightly more general than the one above. However, each inductive inequality is equivalent to a conjunction of refined inductive inequalities (as in Definition \ref{Inducive:Ineq:Def}). Indeed, every refined inductive inequality is inductive. Moreover, by substituting all positive (resp.\ negative) uniform variables with $\top$ (resp.\ $\bot$), and by exhaustively distributing every connective over $\wedge$ and $\vee$, any $(\Omega, \varepsilon)$-inductive inequality can be rewritten as an inequality $s_1 \vee \cdots \vee s_n \leq t_1 \wedge \cdots \wedge t_m$, where the signed generation trees of each $s_i$ and $t_j$ are refined $(\Omega, \varepsilon)$-inductive. The latter inequality is equivalent to the conjunction of all the inequalities $s_i \leq t_j$ for $i \in \{1, \ldots, n \}$ and $j \in \{1, \ldots, m \}$.
\end{remark}

\begin{notation}\label{notation: placeholder_variables}
We will often need to use {\em placeholder variables} to e.g.~specify the occurrence of a subformula within a given formula. In these cases, we will write e.g.~$\varphi(!z)$ (resp.~$\varphi(!\overline{z})$) to indicate that the variable $z$ (resp.~each variable $z$ in  vector $\overline{z}$) occurs exactly once in $\varphi$. Accordingly, we will write $\varphi[\gamma / !z]$  (resp.~$\varphi[\overline{\gamma}/!\overline{z}]$   to indicate the formula obtained from $\varphi$ by substituting $\gamma$ (resp.~each formula $\gamma$ in $\overline{\gamma}$) for the unique occurrence of (its corresponding variable) $z$ in $\varphi$. Also, in what follows, we will find it sometimes useful to group placeholder variables together according to certain assumptions we make about them. So, for instance, we will sometimes write e.g.~$\varphi(!\overline{x}, !\overline{y})$ to indicate that the variables in $x$ are positive in $\varphi$, and the variables in $y$ are negative in $\varphi$, or we will write e.g.~$f(!\overline{x}, !\overline{y})$ to indicate that $f$ is monotone (resp.~antitone) in the coordinates corresponding to every variable $x$ in  $\overline{x}$ (resp.~$y$ in  $\overline{y}$).  We will provide further explanations as to the intended meaning of these groupings whenever required. Finally, we will also extend these conventions to inequalities or sequents, and thus write e.g.~$(\phi\leq \psi) [\overline{\gamma}/!\overline{z}, \overline{\delta}/!\overline{w}] $ to indicate  the inequality obtained from $\varphi\leq \psi$ by substituting each formula $\gamma$ in $\overline{\gamma}$ (resp.~$\delta$ in $\overline{\delta}$) for the unique occurrence of its corresponding variable $z$ (resp.~$w$) in $\varphi\leq \psi$. 
Finally, we will extend this vectorial notation for substitution to the meta-language defined in Definition \ref{def:albametalang}: thus, we will write e.g.\ $(\varphi_1\leq\psi_1 \metaand \varphi_2 \leq \psi_2)[\overline{\chi}/!\overline{x}]$ for $(\varphi_1\leq\psi_1)[\overline{\chi}/!\overline{x}] \metaand (\varphi_2\leq\psi_2)[\overline{\chi}/!\overline{x}]$.
\end{notation}

\begin{notation}
\label{notation:compactinductive}
In what follows, we refer to  a formula $\chi$ such that $+\chi$ (resp.\ $-\chi$)  consists only of skeleton nodes as a \emph{positive} (resp.\ \emph{negative}) \emph{skeleton}; and we dub formulas $\zeta$ as positive (resp.\ negative) (definite) PIA if there is a path from a leaf to the root of $+\zeta$ (resp.\ $-\zeta$) consisting only of (definite) PIA nodes.\footnote{In the context of analytic inductive inequalities \cite{chen2021syntactic}, definite positive (resp.\ negative) PIA formulas coincide with definite negative (resp.\ positive) skeleton formulas.}

Taking inspiration from \cite{chen2021syntactic}, we will often write refined $(\Omega, \varepsilon)$-inductive inequalities as follows: 
\begin{equation}
\label{eqn:Inductive:Compact}
(\varphi\leq \psi)[\overline{\eta_a}/!\overline{a}, \overline{\eta_b} / !\overline{b}][\overline{\alpha}/!\overline{x}, \overline{\beta}/!\overline{y},\overline{\gamma}/!\overline{z}, \overline{\delta}/!\overline{w}],
\end{equation}
where $(\varphi\leq \psi)[!\overline a, !\overline b]$ contains only definite skeleton nodes, is positive (resp.\ negative) in $!\overline a$ (resp.\ $!\overline b$), and is \emph{scattered}, i.e.\ each variable occurs at most once in each polarity; 
each $\eta_a$ in $\overline{\eta_a}$ (resp.\ $\eta_b$ in $\overline{\eta_b}$) is a finite meet (resp.\ join) containing variables in $\overline x$ and $\overline z$ (resp.\ $\overline y$ and $\overline w$); $(\varphi\leq \psi)[\overline{\eta_a}/!\overline{a}, \overline{\eta_b} / !\overline{b}]$ contains every variable in $\overline x$, $\overline y$, $\overline z$, and $\overline w$ exactly once;
each $\alpha$ in $\overline\alpha$ (resp.\ $\beta$ in $\overline\beta$) is a positive (resp.\ negative) definite PIA, and for each variable $v$ in $\overline x$ (resp.\ $\overline y$) there is a formula in $\overline\alpha$ (resp.\ $\overline\beta$) where $v$ is $\varepsilon$-critical. The formulas $\varphi$ and $\psi$ are {\em the skeleton} of the inequality. The length of $\overline a$, $\overline b$, $\overline x$, $\overline y$, $\overline z$, and $\overline w$ is often denoted by $n_a$, $n_b$, $n_x$, $n_y$, $n_z$, and $n_w$, respectively.
\end{notation}

\begin{definition}
\label{def:sahlqvist_and_vss}
A refined inductive inequality $(\varphi\leq \psi)[\overline{\eta_a}/!\overline{a}, \overline{\eta_b} / !\overline{b}][\overline{\alpha}/!\overline{x}, \overline{\beta}/!\overline{y},\overline{\gamma}/!\overline{z}, \overline{\delta}/!\overline{w}]$ is {\em Sahlqvist} if the formulas in $\overline\alpha$ and $\overline\beta$ do not contain any PIA connective in a critical branch whose arity is greater than one (i.e., an SRR node). The inequality is {\em very simple Sahlqvist} if each formula in $\overline\alpha$, $\overline\beta$, $\overline{\eta_a}$, and $\overline{\eta_b}$ is a propositional variable.
\end{definition}

We report some examples showcasing the definition of (refined) inductive inequality taken from \cite{inductivepolynomial}.
\begin{example}
\label{example:inductive}
The Frege inequality $p \rightarrow (q \rightarrow r) \leq (p\rightarrow q) \rightarrow (p \rightarrow r)$ in the language of intuitionistic logic is $(\Omega, \varepsilon)$-inductive, e.g., for $\varepsilon(p)=\varepsilon(q)=\varepsilon(r) = 1$ and $p<_\Omega q <_\Omega r$. In this language, $\rightarrow$ is an operator in $\mathcal{G}$.
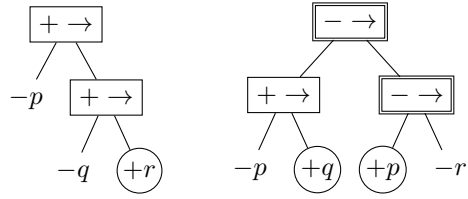
\smallskip

{{\centering
\begin{forest}
[ $+\rightarrow$, draw
    [ $-p$ ]
    [ $+\rightarrow$, draw
        [$-q$]
        [$+r$, circle, draw, inner sep=1pt]
    ]
]
\end{forest}
\quad\quad
\begin{forest}
[ $-\rightarrow$, double, draw
    [ $+\rightarrow$, draw
        [$-p$]
        [$+q$, circle, draw, inner sep=1pt]
    ]
    [ $-\rightarrow$, double, draw
        [$+p$, circle, draw, inner sep=1pt]
        [$-r$]
    ]
]
\end{forest}
\par}}
\smallskip

Critical occurrences are drawn inside circles, $-f$ and $+g$ nodes are drawn inside rectangles, and $+f$ and $-g$ nodes are inside double rectangles. 
Every critical branch is refined good (indeed every branch is refined good), and the critical variable under each SRR node is greater (w.r.t.~$<_\Omega$) than the other variables occurring under the same node.
The inequality above is also $(\Omega', \varepsilon')$-inductive for $\varepsilon'(p)=\varepsilon'(q) = 1$, $\varepsilon'(r)=\partial$ and $p<_{\Omega'} q$.

The inequality $p\wedge \Box(\Diamond p\rightarrow \Box q)\leq \Diamond\Box\Box q$ in the language of distributive modal logic\footnote{In \cite{GoVa2006} it was shown that this inequality is not semantically equivalent to any Sahlqvist inequality in classical modal logic. Thus the class of inductive inequalities is not just a syntactic proper extension of the Sahlqvist class, but also a semantic one.} is inductive w.r.t.~the order-type $\varepsilon(p) = \varepsilon(q) = (1, 1)$ and $p <_\Omega q$.
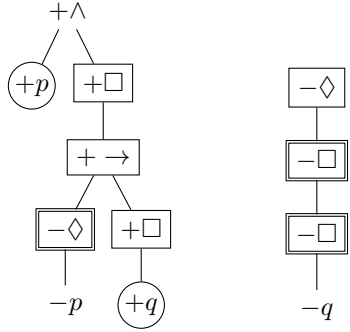
\smallskip

{{\centering
\begin{forest}
[ $+\wedge$,
    [$+p$, circle, draw, inner sep=1pt]
    [$+\Box$, draw
        [$+\rightarrow$, draw
            [$-\Diamond$, double, draw
                [$-p$]
            ]
            [$+\Box$, draw
                [$+q$, circle, draw, inner sep=1pt]
            ]
        ]
    ]
]
\end{forest}
\quad\quad\quad\quad
\begin{forest}
[ $-\Diamond$, draw
    [$-\Box$, double, draw
        [$-\Box$, double, draw
            [$-q$]
        ]
    ]
]
\end{forest}
\par}}
\end{example}

\begin{example}
\label{eg:longpartial}
The $\mathcal{L}_\mathrm{DML}$-inequality $\Box(\Box q \rightarrow \Diamond ((\Box q \wedge q) \rightarrow p) ) \leq \Diamond p \vee \rhd q$ is refined $(\Omega,\varepsilon)$-inductive for $\varepsilon(p) = \partial$, $\varepsilon(q)= 1$, and any $\Omega$. Its signed generation tree is the following
\smallskip

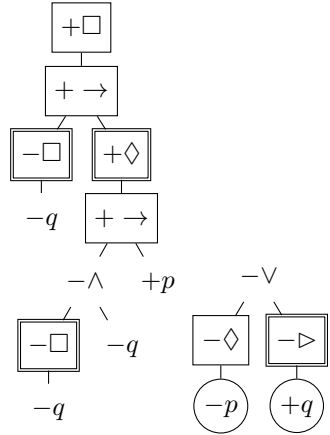
{{\centering
\begin{forest}
l sep = 0, for children = {l*= 0.6},
  for tree = {
    align = center,
    l sep = 5, for children = {l*= 0.6}, 
  },
[$+\Box$, draw
    [$+\rightarrow$, draw
        [$-\Box$, double, draw
            [$-q$]
        ]
        [$+\Diamond$, double, draw
            [$+\rightarrow$, draw
                [$-\wedge$
                    [$-\Box$, double, draw
                        [$-q$]
                    ]
                    [$-q$]
                ]
                [$+p$]
            ]
        ]
    ]
]
\end{forest}
\begin{forest}
l sep = 0, for children = {l*= 0.6},
  for tree = {
    align = center,
    l sep = 5, for children = {l*= 0.6}, 
  },
[$-\vee$
   [$-\Diamond$, draw
        [$-p$, circle, draw, inner sep=1pt]
   ]
   [$-\rhd$, double, draw
        [$+q$, circle, draw, inner sep=1pt]
   ]
]
\end{forest}
\par}}
\end{example}
\subsection{The language of ALBA}
\label{ssec:ALBAlang}

ALBA is a calculus for correspondence that has been shown to successfully compute the first order correspondents of LE-inductive inequalities (cf.\ Definition \ref{Inducive:Ineq:Def}) by exploiting the order theoretic properties of the LE setting \cite{CoPa:non-dist}. The present section introduces the first order language in which ALBA's correspondence is achieved.

ALBA manipulates inequalities and quasi-inequalities\footnote{A {\em quasi-inequality} of $\langbase$ is an expression of the form $\bigmetaand_{i = 1}^n s_i\leq t_i \Rightarrow s\leq t$, where $s_i\leq t_i$ and $s\leq t$ are $\langbase$-inequalities for each $i$.} in the expanded language $\langalba$, which is built up on the base of the lattice constants $\top, \bot$ and an enlarged set of propositional variables 
$\nomset\cup \cnomset\cup \mathsf{AtProp}$ (the variables $\nomh, \mathbf{i}, \mathbf{j}, \nomk$ in $\nomset$ are referred to as \emph{nominals}, and the variables $\cnoml, \mathbf{m}, \mathbf{n}, \cnomo$ in $\cnomset$ as \emph{conominals}), 
closing under the logical connectives of $\langbase^*$. The natural semantic environment of $\langalba$ is given by {\em perfect} $\langbase$-{\em algebras} $\mathbb{A}$, which are based on complete lattices, both completely join-generated by their completely join-irreducible elements $J^\infty(\mathbb{A})$ and completely meet-generated by their completely meet-irreducible elements $M^\infty(\mathbb{A})$, and such that every connective has residuals in each coordinate.
Nominals and conominals range over the sets of the completely join-irreducible elements and the completely meet-irreducible elements of perfect LEs, respectively. 

We will extensively write $\pureu, \purev, \purew$ to indicate generic \emph{pure variables}, i.e., variables in $\nomset\cup\cnomset$. The language of $\langalba$-inequalities is referred to as $\langineq$. An inequality in $\langineq$ whose variables are all pure is a \emph{pure inequality}.

\begin{remark}
\label{remark:perfect_actually_not_perfect}
As remarked in \cite[Section 9]{CoPa:non-dist}, the rules of ALBA for LE-logics never use the complete join-irreducibility (resp.~meet-irreducibility) of the elements of $J^\infty(\mathbb{A})$ (resp.\ $M^\infty(\mathbb{A})$), but only use the fact that $J^\infty(\mathbb{A})$ (resp.\ $M^\infty(\mathbb{A})$) join-generate (resp.\ meet-generate) a perfect algebra $\mathbb{A}$. Thus, the results of the present paper transfer straightforwardly also to semantic settings consisting of or corresponding to algebras which are not necessarily perfect, but are complete (in the sense that their underlying lattice is complete, and  all their operations are residuated in each coordinate) and  are join-generated (resp.\ meet-generated) by some designated sets. This is for instance the case of constructive canonical extensions, whose designated families of join-generators and meet-generators are the sets of closed and open elements, respectively. This is also the case of arbitrary polarities (i.e.~polarities which are not necessarily reduced and separated \cite{gehrke2006generalized}).
\end{remark}

\begin{definition}[ALBA language for inverse correspondence]
\label{def:albametalang}
ALBA's language for inverse correspondence $\langmeta$ is generated by the following rules:
\[
\xi ::= \nomj \nleq \cnomm \ |\ s \leq t \ |\ \xi \metaand \xi  \ |\ \xi \metaor \xi  \ |\ \metanot \xi  \ |\ \xi \Rightarrow \xi \ | \ \forall \nomj\ \xi  \ |\  \forall \cnomm\ \xi  \ |\ \exists \nomj \ \xi  \ |\ \exists \cnomm\ \xi,
\]
where $\metaand$ (resp.\ $\metaor$) denotes the meta-linguistic conjunction (resp.\ disjunction), $\Rightarrow$ (resp.\ $\metanot$) denotes the meta-linguistic implication (resp.\ negation), and $s\leq t$ is an inequality in $\langineq$ of shape $\nomj \leq \varphi[\overline\pureu]$ (resp.\ $\psi[\overline\pureu] \leq \cnomm$) where $\varphi$ (resp.\ $\psi$) is a skeleton formula (cf.\ Notation \ref{notation:compactinductive}), and the leaves of the signed generation tree $+\varphi[\overline\pureu]$ (resp.\ $-\psi[\overline\pureu]$) which are positive are labelled with nominals, while the negative ones are labelled with conominals.
\end{definition}

\begin{remark}
\label{remark:flatinequalities}
In classical and distributive modal logic, the inequalities equivalent to relational atoms are inequalities of shape as in Definition \ref{def:albametalang} which are {\em flat}, in the sense that they contain just one connective. For instance, in classical modal logic, an inequality such as $\nomj \leq \Diamond \nomi$ can be thought of as a relational atom $Rxy$, whenever $\nomj$ is interpreted as $\{x\}$ and $\nomi$ as $\{y\}$. Similarly, an atom such as $\Box\cnomm \leq \cnomn$ is interpreted as $Rxy$ whenever $\cnomm$ is interpreted as $\{y\}^c$, and $\cnomn$ as $\{x\}^c$. Allowing for arbitrary length skeleton formulas is analogous to allowing for relational atoms corresponding to compositions of relations in the classical (and distributive) modal setting. Furthermore, we will show how it is possible to break these inequalities down into simpler formulas containing only flat inequalities. We have chosen to allow for arbitrary skeleton formulas for the sake of compactness, and ease of presentation and use. In Section \ref{sssec:flatteningskeleton}, we show how the results of the paper can be refined to the case in which at most one connective is allowed in the inequalities.

Furthermore, inequalities such as $\nomj \nleq \cnomm$ provide a link between nominals and conominals which is expressible in the natural first order languages of classical and distributive modal logic. Indeed, in the classical setting one of such inequalities is interpreted as $\{x\} \nsubseteq \{y\}^c$, which is equivalent to $x=y$.
\end{remark}
\subsection{Geometric implications}
\label{ssec:geometricimpl}

In the remainder of the section, we fix a first order language $\mathcal{L}_\mathrm{FO}$ where $\&$, $\parr$, and $\Rightarrow$ denote conjunctions, disjunctions, and implications, respectively.

\begin{definition}[Geometric formula]
\label{def:geometric_formula}
A {\em geometric formula} (also known as {\em coherent formula} \cite[D 1.1.3, item (d)]{johnstone:elephant}, and {\em positive formula}) is a $\mathcal{L}_{\mathrm{FO}}$-formula built up from atoms, conjunctions, disjunctions, and existential quantifiers.\footnote{In the literature, the terms `geometric formula' and `positive formula' are sometimes used to denote the same concept, but allowing also for infinite disjunctions \cite{Wraith1979}.  Indeed, this notion originates from a characterization of the formulas which are preserved by {\em geometric morphisms}, i.e. adjoint functor pairs between topoi, such that the left adjoint preserves finite limits (and infinite co-limits). This idea is the generalization of the fact that any continuous map $f:X \to Y$ between topological spaces gives rise to an adjunction $\adjunction{f^*}{\mathsf{Sh}(Y)}{\mathsf{Sh}(X)}{f_*}$ between the toposes of sheaves of sets in these spaces, and the inverse image functor $f^*$ preserves finite limits (and infinite co-limits as it is left adjoint).}
\end{definition}

\begin{definition}[Geometric implication]
\label{def:geometric_implication}
A {\em geometric implication} is a $\mathcal{L}_{\mathrm{FO}}$-formula 
$\forall \overline x(A \Rightarrow B)$,
such that $A$ and $B$ are geometric formulas.
\end{definition}

\begin{lemma}[Section 2.3 of \cite{simpson1994}]
\label{lemma:simpson_basic_geometric_sequent}
Any $\mathcal{L}_\mathrm{FO}$-geometric implication is intuitionistically equivalent to some formula
\[
\forall \overline x 
\left( 
\bigmetaand_{i=1}^n A_i \Rightarrow 
\bigmetaor_{j=1}^m \exists \overline y_j B_j 
\right)
\]
such that the $A_i$ are atoms, and the $B_j$ are conjunctions of atoms such that the variables in $\overline y_j$ do not occur free in the antecedent.
\end{lemma}

\begin{definition}[Generalized geometric implication \cite{negri2016sysrul}]
\label{def:generalized:geometric}
Let $\gengeomlevel{0}$ be the collection of geometric implications, and for every $n \in \mathbb{N}$, let $\gengeomlevel{n+1}$ be the collection of $\mathcal{L}_\mathrm{FO}$ formulas with shape
\[
\forall \overline x 
\left( 
\bigmetaand_{i=1}^n A_i \Rightarrow 
\bigmetaor_{j=1}^m \exists \overline y_j \bigmetaand_{k=1}^{o_j}\varphi_{j,k}^{h_{j,k}}
\right),
\]
where each $\varphi_{j,k}^{h_{j,k}}$ is a formula in $\gengeomlevel{h_{j,k}}$, with $h_{j,k} \leq n$.
\end{definition}

\begin{example}
\label{eg:geomarithmetic}
In the first order language of preorders containing a predicate symbol $\leq$, the formula
\[
\forall x \forall y (x \leq y \Rightarrow \exists z(y \leq z) \metaor \exists w(w \leq x \metaand w \leq y))
\]
is a geometric implication. The following formula is generalized geometric in $\gengeomlevel{1}$,
\[
\forall x \forall y (x \leq y \Rightarrow \exists z(y \leq z \wedge \forall t(t \leq x \Rightarrow \exists u(u \leq y))) \metaor \exists w(w \leq x \metaand w \leq y)).
\]
\end{example}

\begin{example}
\label{eg:krachtasgeometric}
In the first order language of Kripke frames containing a binary predicate $R$, Kracht formulas \cite{krachtphdthesis,blackburn2002modal} are generalized geometric implications. For instance, the Kracht formula
\[
\forall x_1\forall x_2(Rx_0x_1 \metaand Rx_1x_2 \Rightarrow \exists y_1(Rx_0y_1 \metaand (Rx_2x_1 \metaor Ry_1x_0)))
\]
is a geometric implication. The Kracht formula
\[
\forall x_1(Rx_0x_1 \Rightarrow \exists y_1 Rx_0y_1 \metaand \forall y_2(Rx_1y_2 \Rightarrow \exists y_3 (Ry_1y_3 \metaand \forall y_4(Rx_0y_4 \Rightarrow Rx_1y_4 \metaor (Ry_3x_0 \metaand Rx_0x_1)))))
\]
is a generalized geometric implication in $\gengeomlevel{2}$.
\end{example}

\begin{example}
The following $\mathcal{L}_{\mathrm{DML}}^\mathrm{FO}$-formula
\[
\forall \nomj\forall \nomi\forall \cnomm(\nomj \leq \Diamond\rhd \nomi \vee \cnomm \Rightarrow \exists\cnomn (\Box \cnomm \wedge \lhd \nomi \leq \cnomn \metaand \forall \nomh(\nomh \leq \nomi \rightarrow (\cnomm \vee \Box \cnomn))))
\]
is a generalized geometric implication in $\gengeomlevel{1}$.
\end{example}
\section{ALBA outputs as generalized geometric implications}
\label{ssec:ALBAcorr}

In the present section, we fix an arbitrary LE-language $\langbase = \langbase(\mathcal{F},\mathcal{G})$, and, throughout the section,
we will omit to specify that formulas and terms pertain to $\langbase$; moreover, 
when referring to perfect LEs, we will omit to state every time that their signature is compatible with $\langbase$. 
We revisit the algorithm ALBA for correspondence in Section \ref{ssec:albadirect}, and show that the first order correspondents of (refined) inductive inequalities are generalized geometric implications in $\langmeta$ in Section \ref{ssec:flattening}.

\subsection{ALBA correspondence}
\label{ssec:albadirect}

The algorithm ALBA for LE-logics has been introduced in \cite{CoPa:non-dist}. In the present subsection, we present the algorithm from a different perspective which is useful to show the new developments contained in the next section.

\begin{notation}
\label{notation:termfunction}
For any lattice expansion $\mathbb{A} = (L, \mathcal{F}^\mathbb{A}, \mathcal{G}^\mathbb{A})$, and any formula $\alpha(!x_1,\ldots,!x_n)$, let $\alpha^\mathbb{A}:\mathbb{A}^n \to \mathbb{A}$ the term function associated to $\alpha$ mapping any $a_1,\ldots,a_n \in \mathbb{A}$ as follows: for any $x_i$, $x_i^\mathbb{A} = a_i$; for any formulas $\alpha_1$ and $\alpha_2$, $(\alpha_1 \wedge \alpha_2)^\mathbb{A} = \alpha_1^\mathbb{A}\wedge^\mathbb{A} \alpha_2^\mathbb{A}$ and $(\alpha_1 \vee \alpha_2)^\mathbb{A} = \alpha_1^\mathbb{A}\vee^\mathbb{A} \alpha_2^\mathbb{A}$; for any $m$-ary connective $h \in \mathcal{F} \cup \mathcal{G}$ and any formulas $\alpha_1, \ldots, \alpha_m$, $h^{\mathbb{A}}(\alpha_1,\ldots,\alpha_{n}) = h^{\mathbb{A}}(\alpha_1^{\mathbb{A}},\ldots,\alpha_{n}^{\mathbb{A}})$. It is useful to note that the function $\alpha^\mathbb{A}$ is monotone (resp.\ antitone) with respect to the $i$-th coordinate whenever the leaf containing $x_i$ has sign $+$ (resp.\ $-$) in the positive signed generation tree of $\alpha$. When it is clear from the context, we abuse the notation and write $\alpha$ instead of $\alpha^\mathbb{A}$.
\end{notation}

The following lemma is a restatement of \cite[Lemma 6.2]{CoPa:non-dist}.

\begin{lemma}
\label{lemma:piaadjoint}
For any perfect LE $\mathbb{A} = (L, \mathcal{F}^\mathbb{A}, \mathcal{G}^\mathbb{A})$, and any $\langbase$-formula $\alpha(!\overline x)$, consider the positive generation tree of $\alpha$. For all $i \in \{1, \ldots, n\}$, let $s$ be the sign of the leaf containing $x_i$ in $+\alpha$, then:
\begin{itemize}
    \item[(1)] if the path from the root to the leaf containing $x_i$ contains only $+f$ and $-g$ nodes, then  
    \smallskip

    {{\centering
    $\alpha^{\mathbb{A}}(x_1,\ldots, x_{i-1}, \bigvee^{s} B, \ldots, x_n) = 
    \bigvee_{b \in B}\alpha^{\mathbb{A}}(x_1, \ldots, x_{i-1}, b, \ldots, x_n)$
    \par}}
    \smallskip
    
    \noindent for any $B \subseteq \mathbb{A}$, 
    i.e.\ $\alpha^{\mathbb{A}}$ has a right residual w.r.t.\ $x_i$.
    \item[(2)] if the path from the root to the leaf containing $x_i$ contains only $+g$ and $-f$ nodes, then 
    \smallskip

    {{\centering
    $\alpha^{\mathbb{A}}(x_1,\ldots, x_{i-1}, \bigwedge^{s} B, \ldots, x_n) = 
    \bigwedge_{b \in B}\alpha^{\mathbb{A}}(x_1, \ldots, x_{i-1}, b, \ldots, x_n)$
    \par}}
    \smallskip
    
    \noindent for any $B \subseteq \mathbb{A}$, 
    i.e.\ $\alpha^{\mathbb{A}}$ has a left residual w.r.t.\ $x_i$.
\end{itemize}
\end{lemma}
\begin{proof}
We prove the two items by simultaneous induction on the length of the path from the root to $z$. The two base cases are readily obtained from the fact that $\mathbb{A}$ is perfect (see Definition \ref{def:perfect LE}). Now let $\alpha = h_1(\beta_1, \ldots, \beta_m)$ and $\pi = (+h_1, s_2 h_2 \ldots, s_k h_k, s x_i)$ be a path from the root to the leaf containing $z$ with $k \geq 1$, where $s_1,\ldots,s_{k+1} \in \{+,-\}$, and $h_1, \ldots, h_k \in \mathcal{F} \cup \mathcal{G}$. 

Consider the case in which $h_1 \in \mathcal{F}$. If $h_2 \in \mathcal{F}$, then $\varepsilon_j(h_1) = 1$, where $j$ is the coordinate at which the subtree rooted in $s_2h_2$ occurs in $h_1$ (i.e., $s_2 = +$). By inductive hypothesis, $\beta_j^\mathbb{A}$ has a right residual w.r.t.\ $x_i$, i.e., $\beta_j^\mathbb{A}$ preserves joins (resp.\ reverses meets) if $s=+$ (resp.\ $s=-$); hence, since $\mathbb{A}$ is perfect, also preserves joins (resp.\ reverses meets) if $s=+$ (resp.\ $s=-$).

The case in $h_1 \in \mathcal{G}$ is similarly proved.
\end{proof}

The following corollary straightforwardly follows from the lemma above, as definite skeleton formulas contain only $+f$ and $-g$ nodes (see Notation \ref{notation:compactinductive}); thus all the paths from the root to a leaf contain only $+f$ and $-g$ nodes (and hence, the positive signed generation tree of a negative skeleton formula contains only $+g$ and $-f$ nodes).

\begin{corollary}
\label{cor:skeletonadjoint}
For any positive (resp.\ negative) scattered definite skeleton formula $\varphi(!\overline x,!\overline y)$, its associated term function $\varphi^\mathbb{A}$ has right (resp.\ left) residuals with respect to each coordinate. Equivalently, $\varphi^\mathbb{A}$ preserves joins (resp.\ meets) in its monotone coordinates, and reverses meets (resp.\ joins) in its antitone coordinates.
\end{corollary}

The result above allows to prove the following key lemma.

\begin{lemma}[First approximation, \cite{CoPa:non-dist}]
\label{lemma:firstapprox}
The following are semantically equivalent on perfect LEs for any positive definite skeleton formula $\varphi(!\overline x,!\overline y)$ and negative definite skeleton formula $\psi(!\overline z, !\overline w)$ such that the leaves containing the variables in $\overline x$ and $\overline z$ (resp.\ $\overline y$ and $\overline w$) occur positively (resp.\ negatively) in $+\varphi$ and $-\psi$,
\[
\begin{array}{rl}
1. &\forall \overline x, \overline y, \overline z, \overline w \ \ \varphi(\overline x, \overline y) \leq \psi(\overline z, \overline w);\\
2.& \forall \overline x, \overline y, \overline z, \overline w, \overline\nomj_x, \overline \cnomm_y, \overline \nomj_z, \overline \cnomm_w 
\big(
\bigmetaand_{i=1}^{n_x}\nomj_{x_i}\leq x_i \metaand
\bigmetaand_{i=1}^{n_y} y_i\leq \cnomm_{y_i} \metaand
\bigmetaand_{i=1}^{n_z} \nomj_{z_i} \leq z_i \metaand
\bigmetaand_{i=1}^{n_w} w_i \leq \cnomm_{w_i} \Rightarrow \\[1mm]
&\hfill
(\varphi \leq \psi)[\overline{\nomj_x/!x}, \overline{\cnomm_y/!y}, \overline{\nomj_z/!z},\overline{\cnomm_w/!w}]
\big),
\end{array}
\]
where $n_x$, $n_y$, $n_z$, and $n_w$ are the number of variables in $\overline x$, $\overline y$, $\overline z$, and $\overline w$, respectively.
\end{lemma} 
\begin{proof}
Let $\mathbb{A}$ be a perfect LE. Since it is perfect, it is join generated by $J^{\infty}(\mathbb{A})$ and meet generated by $M^\infty(\mathbb{A})$; hence for every $x_i$ and every $z_i$ (resp.\ $y_i$ and $w_i$), $x_i = \bigvee \{ j_{x_i} \in J^\infty(\mathbb{A}) \mid j_{x_i} \leq x_i \}$ and $z_i = \bigvee \{ j_{z_i} \in J^\infty(\mathbb{A}) \mid j_{z_i} \leq z_i \}$ (resp.\ $\bigwedge \{ m_{y_i} \in M^\infty(\mathbb{A}) \mid y_i \leq m_{y_i} \}$ and $w_i = \bigwedge \{ m_{w_i} \in M^\infty(\mathbb{A}) \mid  w_i \leq m_{w_i} \}$. It follows that $\varphi(\overline x, \overline y) \leq \psi(\overline z, \overline w)$ is equivalent to
\[
\varphi\Big(\bigvee_{j_{x_i} \leq x_i} j_{x_i}, \bigwedge_{y_i\leq m_{y_i}} m_{y_i}\Big)
\leq
\psi\Big(\bigvee_{j_{z_i} \leq z_i} j_{z_i}, \bigwedge_{w_i\leq m_{w_i}} m_{w_i}\Big)
\]
which, by Corollary \ref{cor:skeletonadjoint}, is equivalent to
\[
\bigvee_{\substack{j_{x_i}\leq x_i\\ 
y_i\leq m_{y_i}}} \varphi(\overline{j_{x_i}}, \overline{m_{y_i}})
\leq
\bigwedge_{\substack{j_{z_i}\leq z_i \\ w_i\leq m_{w_i}}} \varphi(\overline{j_{z_i}}, \overline{m_{w_i}}).
\]
\end{proof}


\begin{remark}
\label{remark:definiteskeleton}
In the lemma above, the additional assumption of $\varphi$ and $\psi$ being definite is needed since in any LE $\mathbb{A}$, the meet (resp.\ join) is not completely meet (resp.\ join) preserving, negative (resp.\ positive) definite skeleton nodes distribute  over the meet (resp.\ join) in their positive coordinates, and reverse joins (resp.\ meets) in their negative coordinates. 
However, any inequality $\varphi \leq \psi$ such that $\varphi$ (resp.\ $\psi$) is a positive (resp.\ negative) skeleton formula is semantically equivalent in perfect LEs to a conjunction of inequalities $\varphi_i \leq \psi_i$ where each $\varphi_i$ (resp.\ $\psi_i$) is a positive (resp.\ negative) {\em definite} skeleton formula (cf.\ Lemma \ref{lemma:piaadjoint}) on which the lemma above applies.

Moreover, notice how the polarity of the variables in $\overline x$, $\overline y$, $\overline z$, and $\overline w$ is the same in both the equivalent rewritings.
\end{remark} 

We remind the reader of Notation \ref{notation:compactinductive} for denoting inductive inequalities, e.g., in the following corollary. 
\begin{corollary}
\label{cor:firstapprox_inductive}
The following are semantically equivalent on perfect LEs:
\[
\begin{array}{rl}
1. & (\varphi\leq \psi)[\overline{\eta_a}/!\overline{a}, \overline{\eta_b} / !\overline{b}][\overline{\alpha}/!\overline{x}, \overline{\beta}/!\overline{y},\overline{\gamma}/!\overline{z}, \overline{\delta}/!\overline{w}]\\
2. & \forall \overline p, \overline\nomj_x, \overline \cnomm_y, \overline \nomj_z, \overline \cnomm_w 
\big(
\bigmetaand_{i=1}^{n_x}\nomj_{x_i}\leq \alpha_i \metaand
\bigmetaand_{i=1}^{n_y} \beta_i\leq \cnomm_{y_i} \metaand
\bigmetaand_{i=1}^{n_z}\nomj_{z_i}\leq \gamma_i \metaand 
\bigmetaand_{i=1}^{n_w} \delta_i\leq \cnomm_{w_i} \Rightarrow \\[1mm]
&\hfill
(\varphi \leq \psi)[\overline{\nomj_a/!a}, \overline{\cnomm_b/!b}]
\big),
\end{array}
\]
where each $\nomj_a$ (resp.\ $\cnomm_b$) is associated to a variable in $\overline a$ (resp.\ $\overline b$), each $\nomj_x$ and $\nomj_z$ (resp. $\cnomm_y$ and $\cnomm_w$) is $\nomj_a$ (resp.\ $\cnomm_b$) such that $\eta_a$ (resp.\ $\eta_b$) is the unique formula in $\overline{\eta_a}$ (resp.\ $\overline{\eta_b}$) that contains $x$ or $z$ (resp.\ $y$ or $w$).
\end{corollary}
\begin{proof}
By Lemma \ref{lemma:firstapprox}, $\varphi\leq\psi[\overline{\eta_a/!a},\overline{\eta_b/!b}][\overline{\alpha/!x},\overline{\beta/!y},\overline{\gamma/!z},\overline{\delta/!w}]$ is equivalent to
\[
\begin{array}{l}
\forall \overline p,
 \overline {\nomj_a}, \overline{\cnomm_b}
\big(
\bigmetaand_{i=1}^{n_a} (\nomj_a \leq \eta_a) [\overline{\alpha/x},\overline{\beta/y},\overline{\gamma/z},\overline{\delta/w}]
\metaand
\bigmetaand_{i=1}^{n_b} (\eta_b \leq \cnomm_b) [\overline{\alpha/x},\overline{\beta/y},\overline{\gamma/z},\overline{\delta/w}] \Rightarrow \\[1mm]
\hfill
(\varphi \leq \psi)[\overline{\eta_a/!a},\overline{\eta_b/!b}],
\big)
\end{array}
\]
since each $\eta_a$ (resp.\ $\eta_b$) is a finite meet (resp.\ join) of variables $\bigwedge V_{\eta_a}$ (resp.\ $\bigvee V_{\eta_b}$), each inequality $\nomj_a \leq \eta_a = \bigwedge V_{\eta_a}$ (resp.\ $\bigvee V_{\eta_b} = \eta_b \leq \cnomm_b$) is equivalent to $\bigmetaand_{v \in V_{\eta_a}} \nomj_a \leq v$ (resp.\ $\bigmetaand_{v \in V_{\eta_b}} v \leq \cnomm_b$).
\end{proof}

\begin{example}
\label{eg:firstapproxonK}
Given a fixed perfect LE $\mathbb{A}$, consider the inequality K in classical modal logic $\Box(p \rightarrow q) \leq \Box p \rightarrow \Box q$. On the righthand side, the skeleton is $\psi(a,b) = b \rightarrow \Box a$. Since $\mathbb{A}$ is join generated by $J^\infty(\mathbb{A})$ and meet generated by $M^\infty(\mathbb{A})$, $\Box p = \bigvee\{ j \in J^\infty(\mathbb{A}) : j \leq \Box p \}$, $q = \bigwedge\{m \in M^\infty(\mathbb{A}) : q \leq m \}$, and $\Box(p\rightarrow q) = \bigvee\{i \in J^\infty(\mathbb{A}) : i \leq \Box(p \rightarrow q)\}$. Hence, the inequality K can be equivalently rewritten as
\[
\bigvee\{i \in J^\infty(\mathbb{A}) : i \leq \Box(p \rightarrow q)\} \leq \bigvee\{ j \in J^\infty(\mathbb{A}) : j \leq \Box p \} \rightarrow \Box \bigwedge\{m \in M^\infty(\mathbb{A}) : q \leq m \}.
\]
As $\Box$ is completely meet preserving, it is equivalent to
\[
\bigvee\{i \in J^\infty(\mathbb{A}) : i \leq \Box(p \rightarrow q)\} \leq \bigvee\{ j \in J^\infty(\mathbb{A}) : j \leq \Box p \} \rightarrow \bigwedge\{\Box m : m \in M^\infty(\mathbb{A}), q \leq m \},
\]
and since $\rightarrow$ is completely join reversing with respect to the first coordinate, and completely meet preserving with respect to the second, it can be rewritten as
\[
\bigvee\{i \in J^\infty(\mathbb{A}) : i \leq \Box(p \rightarrow q)\} \leq \bigwedge\{ j \rightarrow \Box m : j \in J^\infty(\mathbb{A}), j \leq \Box p, m \in M^\infty(\mathbb{A}), q \leq m \},
\]
which holds if and only if
\[
(\forall i \in J^\infty(\mathbb{A}))(\forall j \in J^\infty(\mathbb{A}))(\forall m \in M^\infty(\mathbb{A}))(i \leq \Box(p \rightarrow q) \metaand j \leq \Box p \metaand q \leq m \Rightarrow i \leq j \rightarrow \Box m).
\]
In ALBA's language,
\[
\forall\nomi,\nomj,\cnomm(\nomi \leq \Box(p \rightarrow q) \metaand \nomj \leq \Box p \metaand q \leq \cnomm \Rightarrow \nomi \leq \nomj \rightarrow \Box \cnomm).
\]
\end{example}

Lemma \ref{lemma:piaadjoint} states some sufficient condition for the interpretation of a formula in a perfect LE to have a residual with respect to some variable. The following definition shows how to construct such a residual.

\begin{definition}[$\mathsf{LA}$ and $\mathsf{RA}$ \cite{CoPa:non-dist}]
\label{def:RA_and_LA}
For any sequence $\overline\phi = (\phi_1,\ldots,\phi_n)$, let $\overline{\phi_{-j}} = (\phi_1,\ldots, \phi_{j-1}, \phi_{j+1}, \ldots, \phi_n)$. 
Let $\phi(!x, \overline z)$ and $\psi(!x, \overline z)$ be any formulas, such that the path from the root to the leaf containing $x$ in $+\phi$ contains only $+f$ and $-g$ inner nodes, and the path from the root to the leaf containing $x$ in $+\psi$ contains only $+g$ and $-f$ inner nodes. 
Let $\mathsf{LA}(\phi)(u, \overline z)$ and $\mathsf{RA}(\psi)(u, \overline z)$ be defined by simultaneous recursion as follows:
\begin{center}
\begin{tabular}{r c l}
\multicolumn{3}{c}{$\mathsf{LA}(x)$ \ =\  $u$; \quad\quad $\mathsf{RA}(x)$ \ =\  $u$;}\\
$\mathsf{LA}(g(\overline{\phi_{-j}(\overline z)},\phi_j(x,\overline z), \overline{\psi(\overline z)}))$ &= &$\mathsf{LA}(\phi_j)(g^{\flat}_{j}(\overline{\phi_{-j}(\overline z)},u, \overline{\psi(\overline z)} ), \overline z)$;\\
$\mathsf{LA}(g(\overline{\phi(\overline z)}, \overline{\psi_{-j}(\overline z)},\psi_j(x,\overline z)))$ &= &$\mathsf{RA}(\psi_j)(g^{\flat}_{j}(\overline{\phi(\overline z)}, \overline{\psi_{-j}(\overline z)},u), \overline z)$;\\
$\mathsf{RA}(f(\overline{\psi_{-j}(\overline z)},\psi_j(x,\overline z), \overline{\phi(\overline z)}))$ &= &$\mathsf{RA}(\psi_j)(f^{\sharp}_{j}(\overline{\psi_{-j}(\overline z)},u, \overline{\phi(\overline z)} ), \overline z)$;\\
$\mathsf{RA}(f(\overline{\psi(\overline z)}, \overline{\phi_{-j}(\overline z)},\phi_j(x,\overline z)))$ &= &$\mathsf{LA}(\phi_j)(f^{\sharp}_{j}(\overline{\psi(\overline z)}, \overline{\phi_{-j}(\overline z)},u), \overline z)$.\\
\end{tabular}
\end{center}
\end{definition}

We refer the reader to Example \ref{eg:albaoutput_goranko_LE} to see how the definition above is concretely applied.

In the equivalent rewriting of an $(\Omega,\varepsilon)$-inductive inequality $(\varphi \leq \psi)[\eta_a/!a,\eta_b/!b][\alpha/!x,\beta/!y,\gamma/!z,\delta/!w]$ described in Corollary \ref{cor:firstapprox_inductive}, each formula $\alpha_i$ in $\overline\alpha$ (resp.\ $\beta_i$ in $\overline\beta$) of contains (cf.\ Definition \ref{Inducive:Ineq:Def}) a unique path of PIA nodes ($-f$ and $+g$) to a leaf labelled with $\crit{\alpha_i} := v_i$ (resp.\ $\crit{\beta_i} := v_i$) and with sign $s_i$, such that $\varepsilon(v_i) = s_i$; hence $\alpha_i$ (resp.\ $\beta_i$) has a right adjoint with respect to $v_i$ (see Lemma \ref{lemma:piaadjoint}), and, by adjunction, 
\[
\begin{array}{l}
\nomj_{x_i} \leq \alpha_i(v_i, \overline z) \quad \mbox{iff} \quad \mathsf{LA}(\alpha_i)(\nomj_{x_i}, \overline z) \leq^{s_i} v_i, \quad\quad and \quad\quad
\beta_i(v_i, \overline z) \leq \cnomm_{y_i} \quad \mbox{iff} \quad  v_i \leq^{s_i} \mathsf{RA}(\beta_i)(\cnomm_{y_i}, \overline z).
\end{array}
\]

Therefore, if $P$ (resp.\ $Q$) is the set of variables $p$ (resp.\ $q$) such that $\varepsilon(p)=1$ (resp.\ $\varepsilon(q) = \partial$), the antecedent of the implication
\[
\bigmetaand_{i=1}^{n_x}\nomj_{x_i}\leq \alpha_i \metaand
\bigmetaand_{i=1}^{n_y} \beta_i\leq \cnomm_{y_i} \metaand
\bigmetaand_{i=1}^{n_z}\nomj_{z_i}\leq \gamma_i \metaand 
\bigmetaand_{i=1}^{n_w} \delta_i\leq \cnomm_{w_i}
\]
is equivalent to
\[
\bigmetaand_{p \in P} \bigvee \mathsf{Mv}(p)\leq p
\metaand
\bigmetaand_{q \in Q} q \leq \bigwedge \mathsf{Mv}(q),
\]
where for each $v_i \in P \cup Q$, 
\[
\mathsf{Mv}(v) = \{ \mathsf{LA}(\alpha_i)(\nomj_{x_i}, \overline z) \mid \alpha_i \mbox{ in } \overline\alpha, p = \crit{\alpha_i}  \} \cup \{\mathsf{RA}(\beta_i)(\cnomm_{y_i}, \overline z) \mid \beta_i \mbox{ in } \overline \beta, v_i = \crit{\beta_i}\}.
\]

After applying the transformations described above, the inductive shape of the input inequality allows to apply the following lemma to eliminate all the non-pure variables.

\begin{lemma}[Ackermann Lemma \cite{CoPa:non-dist}]
\label{lemma:ackermann}
Let $\alpha, \beta(p), \gamma(p), \delta(p), \varepsilon(p)$ be $\langalba$-formulas, and let $p \in \atprop$ such that $p$ does not occur in $\alpha$; if $\beta$ and $\varepsilon$ are negative in $p$ and $\gamma$ and $\delta$ are positive in $p$, then
 the following equivalences hold for every LE  $\mathbb{A}$: 
 \smallskip

 {{\centering
 \begin{tabular}{rlcl}
1. & $\mathbb{A}\models [(p\leq \alpha\ \&\ \beta(p)\leq\gamma(p))\ \Rightarrow\ \delta(p)\leq \varepsilon(p)]$ &
iff & $ \mathbb{A}\models (\beta(\alpha/p)\leq \gamma(\alpha/p)\ \Rightarrow\ \delta(\alpha/p)\leq \varepsilon(\alpha/p))$; \\
2. & $\mathbb{A}\models [(\alpha \leq p \ \&\ \gamma(p)\leq\beta(p))\ \Rightarrow\ \varepsilon(p)\leq \delta(p)]$ &
iff & $ \mathbb{A}\models (\gamma(\alpha/p)\leq \delta(\alpha/p)\ \Rightarrow\ \varepsilon(\alpha/p)\leq \delta(\alpha/p))$; \\
\end{tabular}
\par}}
\end{lemma}

The following corollary is an immediate consequence of Corollary \ref{cor:firstapprox_inductive} and Lemma \ref{lemma:ackermann}.

\begin{corollary}
\label{cor:albaoutput}
Any inductive inequality $(\varphi \leq \psi)[\eta_a/!a,\eta_b/!b][\alpha/!x,\beta/!y,\gamma/!z,\delta/!w]$ is equivalent to
\end{corollary}

\begin{equation}
\label{eq:albaoutput}
\forall \overline\nomj_x, \overline \cnomm_y, \overline \nomj_z, \overline \cnomm_w 
\left(
\left(\bigmetaand_{i=1}^{n_z}\nomj_{z_i}\leq \gamma_i \metaand 
\bigmetaand_{i=1}^{n_w} \delta_i\leq \cnomm_{w_i}\right)\left[\overline{\bigvee \mathsf{Mv}(p)/p}, \overline{\bigwedge \mathsf{Mv}(q)/q}\right] \Rightarrow
\hfill
(\varphi \leq \psi)[\overline{\nomj_a/!a}, \overline{\cnomm_b/!b}]
\right).
\end{equation}

\begin{example}
\label{eg:albaoutput_goranko_LE}
Consider the inequality $p\wedge \Box(\Diamond p\rightarrow \Box q)\leq \Diamond\Box\Box q$ from Example \ref{example:inductive} in the language of non distributive modal logic. It is inductive for $\varepsilon(p) = 1$ and $\varepsilon(q) = 1$. Corollary \ref{cor:firstapprox_inductive} yields that it is equivalent to
\begin{equation}
\label{eq:goranko_albaout_example_first_approx}
\forall \nomj,\cnomm(\nomj \leq p \metaand 
\nomj \leq \Box(\Diamond p\rightarrow \Box q)
\metaand 
\Diamond\Box\Box q \leq \cnomm
\Rightarrow \nomj \leq \cnomm
).
\end{equation}
In order to apply Lemma \ref{lemma:ackermann} and Corollary \ref{cor:albaoutput}, we need to compute $\mathsf{Mv}(q)$, and hence, we need to compute $\mathsf{LA}(\Box(\Diamond p\rightarrow \Box q))(\nomj, p)$ with respect to $q$ (the $x$ in Definition \ref{def:RA_and_LA} is $q$ here).
\smallskip

{{\centering
\begin{tabular}{rcl}
$\mathsf{LA}(\Box(\Diamond p\rightarrow \Box q))(\nomj, p)$ & $=$ & $\mathsf{LA}(\Diamond p\rightarrow \Box q)(\Diamondblack \nomj, p)$\\
& $=$ & $\mathsf{LA}(\Box q)(\Diamond p \wedge \Diamondblack \nomj, p)$ \\
& $=$ & $\mathsf{LA}(q)(\Diamondblack(\Diamond p \wedge \Diamondblack \nomj), p)$ \\
& $=$ & $\Diamondblack(\Diamond p \wedge \Diamondblack \nomj)$. \\
\end{tabular}
\par}}
\smallskip

\noindent The inequality $\nomj \leq \Box(\Diamond p\rightarrow \Box q)$ is indeed equivalent to $\Diamondblack(\Diamond p \wedge \Diamondblack \nomj) \leq q$. By Lemma \ref{lemma:ackermann}, \eqref{eq:goranko_albaout_example_first_approx} is equivalent to
\[
\forall \nomj, \cnomm(
\Diamond\Box\Box\Diamondblack(\Diamond \nomj \wedge \Diamondblack \nomj) \leq \cnomm \Rightarrow \nomj \leq \cnomm
).
\]
\end{example}

\begin{example}
\label{eg:alba_fregeLE}
Consider the LE version of Frege axiom for propositional logic, namely the inequality $p \rightharpoonup (q \rightharpoonup r) \leq (p \rightharpoonup q) \rightharpoonup (p \rightharpoonup r)$, where $\rightharpoonup \in \mathcal{G}$, its order type is $(\partial, 1)$, and $\rightharpoondown$ and $\bullet$ are its residuals w.r.t.\ the first and second coordinate, respectively. The inequality is not Sahlqvist, but is $(\Omega, \varepsilon)$-inductive for $\varepsilon(p) = \varepsilon(q) = \varepsilon(r) = 1$ and $\Omega$ such that $p <_\Omega q <_\Omega r$. Corollary \ref{cor:firstapprox_inductive} and Lemma \ref{lemma:ackermann} yield
\[
\begin{tabular}{rl}
& $\forall \nomj_1, \nomj_2, \nomj_3, \cnomm \big(
\nomj_1 \leq p \rightharpoonup (q \rightharpoonup r)
\metaand
\nomj_2 \leq p \rightharpoonup q
\metaand
\nomj_3 \leq p
\metaand
r \leq \cnomm
\Rightarrow
\nomj_1 \leq \nomj_2 \rightharpoonup (\nomj_3 \rightharpoonup \cnomm)
\big)$ \\
iff & $\forall \nomj_1, \nomj_2, \nomj_3, \cnomm \big(
q \bullet (p \bullet \nomj_1) \leq r
\metaand
p \bullet \nomj_2 \leq q
\metaand
\nomj_3 \leq p
\metaand
r \leq \cnomm
\Rightarrow
\nomj_1 \leq \nomj_2 \rightharpoonup (\nomj_3 \rightharpoonup \cnomm)
\big)$ \\
iff & $\forall \nomj_1, \nomj_2, \nomj_3, \cnomm \big(
(\nomj_3 \bullet \nomj_2) \bullet (\nomj_3 \bullet \nomj_1) \leq \cnomm
\Rightarrow
\nomj_1 \leq \nomj_2 \rightharpoonup (\nomj_3 \rightharpoonup \cnomm)
\big)$. \\
\end{tabular}
\]
\end{example}

\subsection{ALBA outputs flattened to generalized geometric implications}
\label{ssec:flattening}

In the present section, we show that the ALBA output of any inductive $\langbase$-inequality is semantically equivalent on perfect LEs to a generalized geometric $\langmeta$-implication (see Definitions \ref{def:albametalang} and \ref{def:generalized:geometric}). As shown in the previous subsection, the antecedents of such outputs have the following shape (carrying all the notational conventions from there):
\begin{equation}
\label{eq:albaoutputantecedent}
\left(\bigmetaand_{i=1}^{n_z}\nomj_{z_i}\leq \gamma_i 
\metaand 
\bigmetaand_{i=1}^{n_w} \delta_i\leq \cnomm_{w_i}\right)
\left[
\overline{\bigvee \mathsf{Mv}(p)/p}, 
\overline{\bigwedge \mathsf{Mv}(q)/q}
\right].
\end{equation}

In principle, all the $\gamma_i\left[\overline{\bigvee \mathsf{Mv}(p)/p},\overline{\bigwedge \mathsf{Mv}(q)/q}\right]$ and $\delta_i \left[\overline{\bigvee \mathsf{Mv}(p)/p},\overline{\bigwedge \mathsf{Mv}(q)/q}\right]$ can contain any alternation of skeleton and PIA connectives, with the only restriction that the connectives in $\langbase^*$ are found lower in their generation trees, since they can occur only in the minimal valuations $\mathsf{Mv}(p)$ and $\mathsf{Mv}(q)$. 

\begin{lemma}
\label{lemma:flattifypartial}
For every inductive inequality, representing the antecedent of its ALBA output as in \eqref{eq:albaoutputantecedent}, every inequality in $\nomj_{z_i} \leq \gamma_i$ (resp.\ $\delta_i \leq \cnomm_{w_i}$) is semantically equivalent on perfect LEs to a formula in $\langmeta$.
\end{lemma}
\begin{proof}
Each formula in $\overline\gamma$ and $\overline\delta$ can be recursively broken down by grouping the connectives of the same type (skeleton or PIA), i.e., each $\gamma$ (and $\delta$) is either a pure variable, or a formula\footnote{Notice that for a formula whose positive (resp.\ negative) signed generation tree contains only skeleton nodes, its negative (resp.\ positive) signed generation tree contains only PIA nodes and vice versa.}
\begin{equation}
\label{eq:gammadelta_inpartialparts}
\gamma \coloneqq \psi(\zeta_1,\ldots,\zeta_n,\xi_1,\ldots,\xi_m) \quad\quad \delta \coloneqq \varphi(\xi_1,\ldots,\xi_n,\zeta_1,\ldots,\zeta_m)
\end{equation}
where $\psi$ (resp.\ $\varphi$) is a positive (resp.\ negative) PIA formula, hence a negative (resp.\ positive) skeleton formula, such that (without loss of generality) the first $n$ coordinates are positive, and the last $m$ are negative; each $\zeta_i$ (resp.\ $\xi_i$) is a formula of a syntactic shape analogous to $\delta$ (resp.\ $\gamma$). Hence, each  inequality $\nomj_{z_i} \leq \gamma_i$ (resp.\ $\delta_i \leq \cnomm_{w_i}$) is equivalent to a finite conjunction of inequalities $\nomj_{z_i} \leq \gamma_{ij}$ (resp.\ $\delta_{ij} \leq \cnomm_{w_i}$) where, by Remark \ref{remark:definiteskeleton}, each $\gamma_{ij}$ (resp.\ $\delta_{ij}$) has shape such as in \eqref{eq:gammadelta_inpartialparts}, with $\psi$ (resp.\ $\varphi$) being a negative (resp.\ positive) definite skeleton formula. In the case in which $\psi$ (resp.\ $\varphi$) is not an atom, by Lemma \ref{lemma:firstapprox}, each non-empty $\gamma_{ij}$ (resp.\ $\delta_{ij}$) is equivalent to
\[
\forall \overline\cnomo,\overline\nomh
( 
\bigmetaand_{k=1}^n\zeta_{k}\leq\cnomo_i \metaand \bigmetaand_{k=1}^m\nomh_i\leq\xi_{k}
\Rightarrow
\nomj_{z_i} \leq \psi(\overline \cnomo, \overline \nomh)
)
\ \ (\mbox{resp.} \ \forall\overline\nomh,\overline\cnomo(\bigmetaand_{k=1}^m\zeta_{i}\leq\cnomo_i \metaand \bigmetaand_{k=1}^n\nomh_i\leq\xi_k
\Rightarrow
\varphi(\overline\nomh,\overline\cnomo) \leq \cnomm_{w_i}
)).
\]
Since each $\zeta_i$ and (resp.\ $\xi_i$) has the same shape of formulas in $\overline\delta$ (resp.\ $\overline\gamma$), and it has the same position relative to some conominal (resp.\ nominal), it is possible to recursively apply Corollary \ref{cor:skeletonadjoint} again until atomic inequalities are reached. 

In the case in which one of such $\gamma_{ij}$ (resp.\ $\delta_{ij}$ is empty, it might happen to have an inequality $\nomj_{z_i} \leq t$ (resp.\ $t \leq \cnomm_{w_i}$) such that $t$ is a positive (resp.\ negative) skeleton; thus making it impossible to apply Lemma \ref{lemma:firstapprox}. For instance, consider the case of the inequality $\nomj \leq \Box p \wedge \Diamond q$. By Remark \ref{remark:definiteskeleton}, the inequality with definite skeleton formula $\Box \cdot \wedge \ \cdot $ is indeed equivalent to the conjunction of two inequalities with definite skeletons $\nomj \leq \Box p$ and $\nomj \leq \Diamond q$. The skeleton of the first one is non-empty; therefore it is possible to proceed as described above, but the second one contains a positive skeleton node at the root of the right side. To resolve this inconvenience, it is sufficient to introduce a new conominal that approximates the right hand side, i.e. it is sufficient to rewrite $\nomj \leq \Diamond q$ as $\forall \cnomn(\Diamond q \leq \cnomn \Rightarrow \nomj \leq \cnomn)$. More in general, if  $t$ is any positive (resp.\ negative) skeleton formula, then $\nomj \leq t$ (resp.\ $t \leq \cnomm$) is equivalent to 
\[
\forall \cnomn(t \leq \cnomn \Rightarrow \nomj \leq \cnomn) \quad\quad (\mbox{resp.\ } \forall \nomi(\nomi \leq t \Rightarrow \nomi \leq \cnomm)),
\]
for some fresh conominal (resp.\ nominal) $\cnomm$ (resp.\ $\nomi$). Then, one can proceed recursively on $t \leq \cnomn$ (resp.\ $\nomi \leq t$).\footnote{Notice that the two cases presented above can be unified in just the first one, as the approximation done in the second one is the same approximation in Lemma \ref{lemma:firstapprox} when the skeleton formula is a propositional variable. However, we have decided to split the two cases for clarity of presentation.}
\end{proof}

\begin{notation}
\label{notation:flat}
Let us denote the procedure described in the proof of the Lemma above as $\mathsf{Flat}$, and its output on some inequality $\nomj_{z_i} \leq \gamma_i$ (resp.\ $\delta_i \leq \cnomm_{w_i}$) as $\flattify{\nomj_{z_i} \leq \gamma_i}$ (resp.\ $\flattify{\delta_i \leq \cnomm_{w_i}}$).
\end{notation}

\begin{remark}
\label{remark:negatedskeleton}
Let $\varphi(!\overline x,!\overline y)$ and $\psi(!\overline z,!\overline w)$ be some positive and negative skeleton formulas, respectively, which are positive in $\overline x$ (resp.\ $\overline z$), and negative in $\overline y$ (resp.\ $\overline w$). The negated inequality $\varphi \nleq \psi$ is semantically equivalent on any perfect LE $\mathbb{A}$ to $\bigvee\{\nomj \in J^\infty(\mathbb{A}) \mid \nomj \leq \varphi\} \nleq \bigwedge \{\cnomm \in M^\infty(\mathbb{A}) \mid \psi \leq \cnomm \}$, as $\mathbb{A}$ is join-generated and meet-generated by $J^\infty(\mathbb{A})$ and $M^\infty(\mathbb{A})$, respectively. Hence, $\varphi \nleq \psi$ is equivalent to
\begin{equation}
\label{eq:negatedskeleton}
\exists \nomj,\cnomm(\nomj \leq \varphi \metaand \psi \leq \cnomm \metaand \nomj \nleq \cnomm),
\end{equation}
and therefore each such negated inequality can be written as an existentially quantified conjunction in $\langmeta$. Thus, we will informally regard these negated inequalities as abbreviations for \eqref{eq:negatedskeleton} (where $\nomj$ and $\cnomm$ are fresh pure variables), and hence pertaining to $\langmeta$. When such inequalities occur in the antecedent of some meta-implication $\varphi \nleq \psi \metaand \eta \Rightarrow \xi$, they can be equivalently rewritten as $\forall \nomj,\cnomm(\nomj \leq \varphi \metaand \psi \leq \cnomm \metaand \nomj \nleq \cnomm \metaand \eta \Rightarrow \xi)$.
\end{remark}

\begin{example}
\label{eg:flattification1}
The inequality $(\Box\Box p \circ q) \circback \Box \Diamond {\rhd} p \leq (p \circ q) \circback q$ in the language of the Lambek-Grishin calculus enriched with unary modalities is $(\Omega,\varepsilon)$-inductive for
$\varepsilon(p) = 1$ and $\varepsilon(q) = 1$, and any $\Omega$. Corollary \ref{cor:firstapprox_inductive} yields that it is equivalent to
\[
\forall \nomj_1,\nomj_2,\nomj_3,\cnomm_2 
\left[
\nomj_1 \leq (\Box\Box p \circ q) \circback \Box \Diamond {\rhd} p 
\metaand
\nomj_2 \leq p
\metaand
\nomj_3 \leq q
\metaand
q \leq \cnomm_1
\Rightarrow
\nomj_1 \leq (\nomj_2 \circ \nomj_3) \circback \cnomm_1
\right],
\]
and by Lemma \ref{lemma:ackermann}, its ALBA output is
\begin{equation}
\label{eq:example_flattening_alba_output_1}
\forall \nomj_1,\nomj_2,\nomj_3,\cnomm_2 
\left[
\nomj_1 \leq \underbrace{(\Box\Box \nomj_2 \circ \nomj_3) \circback \Box \Diamond {\rhd} \nomj_2}_{\gamma} 
\metaand
\underbrace{\nomj_3}_\delta \leq \cnomm_1
\Rightarrow
\nomj_1 \leq (\nomj_2 \circ \nomj_3) \circback \cnomm_1
\right].
\end{equation}
By the discussion above, the inequality $\nomj_1 \leq \gamma$ is equivalent to
\[
\forall \cnomo_1, \nomh_1, \nomh_2
(
\nomh_1 \leq \Box\Box \nomj_2 
\metaand
\nomh_2 \leq \nomj_3
\metaand
\Diamond{\rhd} \nomj_2 \leq \cnomo_1
\Rightarrow
\nomj_1 \leq (\nomh_1 \circ \nomh_2) \circback \Box\cnomo_1
).
\]
The inequality $\Diamond{\rhd} \nomj_2 \leq \cnomo_1$ can be further flattened, indeed it is equivalent to
\[
\forall \nomh_3(\nomh_3 \leq {\rhd} \nomj_2 \Rightarrow \Diamond\nomh_3 \leq \cnomo_1).
\]
Since \eqref{eq:example_flattening_alba_output_1} is classically equivalent (in the meta-language) to
\[
\forall \nomj_1,\nomj_2,\nomj_3,\cnomm_2 
\left[
\underbrace{\nomj_3}_\delta \leq \cnomm_1
\Rightarrow
\nomj_1 \leq (\nomj_2 \circ \nomj_3) \circback \cnomm_1
\metaor
\nomj_1 \nleq \underbrace{(\Box\Box \nomj_2 \circ \nomj_3) \circback \Box \Diamond {\rhd} \nomj_2}_{\gamma} 
\right],
\]
and $\nomj_1 \nleq (\Box\Box \nomj_2 \circ \nomj_3) \circback \Box \Diamond {\rhd} \nomj_2$ is equivalent to
\[
\exists \cnomo_1, \nomh_1, \nomh_2
(
\nomh_1 \leq \Box\Box \nomj_2 
\metaand
\nomh_2 \leq \nomj_3
\metaand
\forall \nomh_3(\nomh_3 \leq {\rhd} \nomj_2 \Rightarrow \Diamond\nomh_3 \leq \cnomo_1)
\metaand
\nomj_1 \nleq (\nomh_1 \circ \nomh_2) \circback \Box\cnomo_1
),
\]
by Remark \ref{remark:negatedskeleton}, quasi-inequality \eqref{eq:example_flattening_alba_output_1} is equivalent to some generalized geometric implication in $\langmeta$ with two alternations of quantifiers (against three alternations of skeleton and PIA connectives in the starting inequality).
\end{example}

The example above shows how to effectively rewrite an ALBA output as a generalized geometric implication in $\langmeta$, and can be straightforwardly generalized to the following lemma.

\begin{lemma}
\label{lemma:flattification}
The ALBA output of any inductive inequality $(\varphi\leq \psi)[\overline{\eta_a}/!\overline{a}, \overline{\eta_b} / !\overline{b}][\overline{\alpha}/!\overline{x}, \overline{\beta}/!\overline{y},\overline{\gamma}/!\overline{z}, \overline{\delta}/!\overline{w}]$ is semantically equivalent on perfect LEs to some generalized geometric implication in $\langmeta$ (cf.\ Definition \ref{def:generalized:geometric}). Furthermore, the number of quantifier alternations in such implication is $m-1$, where $m$ is the number of alternations of skeleton and PIA connectives in the starting inequality, excluding $\wedge$ and $\vee$ nodes.
\end{lemma}
\begin{proof}
Let us consider the ALBA output of  $(\varphi\leq \psi)[\overline{\eta_a}/!\overline{a}, \overline{\eta_b} / !\overline{b}][\overline{\alpha}/!\overline{x}, \overline{\beta}/!\overline{y},\overline{\gamma}/!\overline{z}, \overline{\delta}/!\overline{w}]$, i.e.,
\[
\forall \overline\nomj_x, \overline \cnomm_y, \overline \nomj_z, \overline \cnomm_w 
\left(
\left(\bigmetaand_{i=1}^{n_z}\nomj_{z_i}\leq \gamma_i \metaand 
\bigmetaand_{i=1}^{n_w} \delta_i\leq \cnomm_{w_i}\right)\left[\overline{\bigvee \mathsf{Mv}(p)/p}, \overline{\bigwedge \mathsf{Mv}(q)/q}\right] \Rightarrow
\hfill
(\varphi \leq \psi)[\overline{\nomj_a/!a}, \overline{\cnomm_b/!b}]
\right),
\]
and let us write $\gamma_i^{\mathrm{mv}}$ (resp.\ $\delta_i^{\mathrm{mv}}$) to denote $\gamma_i\left[\overline{\bigvee \mathsf{Mv}(p)/p}, \overline{\bigwedge \mathsf{Mv}(q)/q}\right]$ (resp.\ $\delta_i\left[\overline{\bigvee \mathsf{Mv}(p)/p}, \overline{\bigwedge \mathsf{Mv}(q)/q}\right]$). By applying Remark \ref{remark:negatedskeleton} to the inequality $(\varphi \leq \psi)[\overline{\nomj_a/!a}, \overline{\cnomm_b/!b}]$, the ALBA output can be equivalently rewritten as follows
\[
\begin{array}{l}
\forall \overline\nomj_x, \overline \cnomm_y, \overline \nomj_z, \overline \cnomm_w, \nomj_\varphi, \cnomm_\psi 
\big(
\left(\bigmetaand_{i=1}^{n_z}\nomj_{z_i}\leq \gamma_i^{\mathrm{mv}} \metaand 
\bigmetaand_{i=1}^{n_w} \delta_i^{\mathrm{mv}} \leq \cnomm_{w_i}\right)
\metaand
\left(
\nomj_\varphi \leq \varphi
\metaand
\psi \leq \cnomm_\psi
\right)[\overline{\nomj_a/!a}, \overline{\cnomm_b/!b}]
\Rightarrow \\
\hfill
\nomj_\varphi \leq \cnomm_\psi
\big),
\end{array}
\]
which, by contraposition, is equivalent to
\[
\begin{array}{l}
\forall \overline\nomj_x, \overline \cnomm_y, \overline \nomj_z, \overline \cnomm_w, \nomj_\varphi, \cnomm_\psi 
\big(
\left(
\nomj_\varphi \leq \varphi
\metaand
\psi \leq \cnomm_\psi
\right)[\overline{\nomj_a/!a}, \overline{\cnomm_b/!b}]
\Rightarrow \\
\hspace{7cm}
\nomj_\varphi \leq \cnomm_\psi 
\metaor
\bigmetaor_{i=1}^{n_z}\nomj_{z_i}\nleq \gamma_i^{\mathrm{mv}} \metaor
\bigmetaor_{i=1}^{n_w} \delta_i^{\mathrm{mv}} \nleq \cnomm_{w_i}
\big).
\end{array}
\]
By Lemma \ref{lemma:flattifypartial} the consequent can be rewritten as a formula in $\langmeta$ using the $\mathsf{Flat}$ procedure (Notation \ref{notation:flat}):
\[
\begin{array}{l}
\forall \overline\nomj_x, \overline \cnomm_y, \overline \nomj_z, \overline \cnomm_w, \nomj_\varphi, \cnomm_\psi 
\big(
\left(
\nomj_\varphi \leq \varphi
\metaand
\psi \leq \cnomm_\psi
\right)[\overline{\nomj_a/!a}, \overline{\cnomm_b/!b}]
\Rightarrow \\
\hspace{6cm}
\nomj_\varphi \leq \cnomm_\psi 
\metaor
\bigmetaor_{i=1}^{n_z}\flattify{\nomj_{z_i}\nleq \gamma_i^{\mathrm{mv}}} \metaor
\bigmetaor_{i=1}^{n_w} \flattify{\delta_i^{\mathrm{mv}} \nleq \cnomm_{w_i}}
\big).
\end{array}
\]
By Remark \ref{remark:negatedskeleton}, it is easy to show by induction that each $\flattify{\nomj_{z_i}\nleq \gamma_i^{\mathrm{mv}}}$ (resp.\ $\flattify{\delta_i^{\mathrm{mv}} \nleq \cnomm_{w_i}}$) is equivalent to some formula in $\mathrm{GA}_{\lceil m/2 \rceil }$, where $m$ is the number of alternations of skeleton and PIA connectives in $-\gamma^{\mathrm{mv}}$ (resp.\ $+\delta^\mathrm{mv}$) excluding $\wedge$ and $\vee$.
\end{proof}

\begin{example}
\label{eg:flatteningmanyalternations}
The inequality $\Box\Diamond\Box\Diamond\Box p \leq p$ is $(\Omega,\varepsilon)$-inductive for $\varepsilon(p) = \partial$, and any dependency order $\Omega$. Its ALBA output is
\[
\forall \nomj, \cnomm
\left(
\nomj \leq \Box\Diamond\Box\Diamond\Box \cnomm \Rightarrow \nomj \leq \cnomm
\right).
\]
The $\mathsf{Flat}$ procedure on $\nomj \leq \Box\Diamond\Box\Diamond\Box\cnomm$ yields
\smallskip

{{\centering
\begin{tabular}{rl}
&  $\nomj \leq \Box\Diamond\Box\Diamond\Box\cnomm$\\
iff & $\forall \cnomn_1 (\Diamond\Box\Diamond\Box\cnomm \leq \cnomn_1 \Rightarrow \nomj \leq \Box\cnomn_1)$ \\
iff & $\forall \cnomn_1 (
\forall \nomi_1(\nomi_1 \leq \Box\Diamond\Box\cnomm \Rightarrow
\Diamond\nomi_1 \leq \cnomn_1
) \Rightarrow \nomj \leq \Box\cnomn_1)$ \\
iff & $\forall \cnomn_1 (
\forall \nomi_1(
\forall \cnomn_2(\Diamond\Box\cnomm\leq\cnomn_2\Rightarrow
\nomi_1 \leq \Box\cnomn_2
) \Rightarrow
\Diamond\nomi_1 \leq \cnomn_1
) \Rightarrow \nomj \leq \Box\cnomn_1)$ \\
iff & $\forall \cnomn_1 (
\forall \nomi_1(
\forall \cnomn_2(
\forall\nomi_2(\nomi_2 \leq \Box\cnomm \Rightarrow \Diamond\nomi_2\leq\cnomn_2)\Rightarrow
\nomi_1 \leq \Box\cnomn_2
) \Rightarrow
\Diamond\nomi_1 \leq \cnomn_1
) \Rightarrow \nomj \leq \Box\cnomn_1)$. \\
\end{tabular}
\par}}
\smallskip

The formula $\mathsf{Flat}(\nomj \nleq \Box\Diamond\Box\Diamond\Box\cnomm)$ can be rewritten by means of basic classical first order logic equivalences as:
\smallskip

{{\centering
\begin{tabular}{rl}
& $\exists \cnomn_1 (
\forall \nomi_1(
\forall \cnomn_2(
\forall\nomi_2(\nomi_2 \leq \Box\cnomm \Rightarrow \Diamond\nomi_2\leq\cnomn_2)\Rightarrow
\nomi_1 \leq \Box\cnomn_2
) \Rightarrow
\Diamond\nomi_1 \leq \cnomn_1
) \metaand \nomj \nleq \Box\cnomn_1)$ \\
iff & $\exists \cnomn_1 (
\forall \nomi_1(
\forall \cnomn_2(
\forall\nomi_2(\nomi_2 \leq \Box\cnomm \Rightarrow \Diamond\nomi_2\leq\cnomn_2)\Rightarrow
\nomi_1 \leq \Box\cnomn_2
) \Rightarrow
\Diamond\nomi_1 \leq \cnomn_1
) \metaand \nomj \nleq \Box\cnomn_1)$ \\
iff & $\exists \cnomn_1 (
\forall \nomi_1(
\Diamond\nomi_1 \nleq \cnomn_1
 \Rightarrow
\exists \cnomn_2(
\forall\nomi_2(\nomi_2 \leq \Box\cnomm \Rightarrow \Diamond\nomi_2\leq\cnomn_2)\metaand
\nomi_1 \nleq \Box\cnomn_2
)
) \metaand \nomj \nleq \Box\cnomn_1)$. \\
\end{tabular}
\par}}
\smallskip

\noindent The last line above is a generalized geometric implication in $\mathrm{GA}_2$.
\end{example}

\begin{example}
The $\mathcal{L}_{\mathrm{DML}}$-inequality $\Box(\Box q \rightarrow \Diamond ((\Box q \wedge q) \rightarrow p) ) \leq \Diamond p \vee {\lhd} q$ from Example \ref{example:inductive} is refined $(\Omega,\varepsilon)$-inductive for $\varepsilon(p) = \partial$, $\varepsilon(q) = \partial$, and any $\Omega$. By Corollary \ref{cor:firstapprox_inductive}, it is equivalent to
\[
\forall \nomj, \cnomm(
\nomj \leq \Box(\Box q \rightarrow \Diamond ((\Box q \wedge q) \rightarrow p) ) 
\metaand
\Diamond p \leq \cnomm
\metaand
{\lhd} q \leq \cnomm
\Rightarrow 
\nomj \leq \cnomm
),
\]
and by Lemma \ref{lemma:ackermann}, its ALBA output is
\[
\forall \nomj, \cnomm(
\nomj \leq \Box(\Box {\blacktriangleleft}\cnomm \rightarrow \Diamond ((\Box{\blacktriangleleft}\cnomm \ \wedge {\blacktriangleleft}\cnomm) \rightarrow \blacksquare\cnomm) ) 
\Rightarrow 
\nomj \leq \cnomm
).
\]
The $\mathsf{Flat}$ procedure on $\nomj \leq \Box(\Box {\blacktriangleleft}\cnomm \rightarrow \Diamond ((\Box{\blacktriangleleft}\cnomm \ \wedge {\blacktriangleleft}\cnomm) \rightarrow \blacksquare\cnomm) ) $ yields
\smallskip

{{\centering
\begin{tabular}{rl}
& $\nomj \leq \Box(\Box {\blacktriangleleft}\cnomm \rightarrow \Diamond ((\Box{\blacktriangleleft}\cnomm \ \wedge {\blacktriangleleft}\cnomm) \rightarrow \blacksquare\cnomm) ) $ \\
iff & $\forall \nomi_1,\cnomn_1(
\nomi_1 \leq \Box{\blacktriangleleft}\cnomm
\metaand
\Diamond ((\Box{\blacktriangleleft}\cnomm \ \wedge {\blacktriangleleft}\cnomm) \rightarrow \blacksquare\cnomm) \leq \cnomn_1
\Rightarrow
\nomj \leq \Box(\nomi_1 \rightarrow \cnomn_1)
)$ \\
iff & $\forall \nomi_1,\cnomn_1(
\nomi_1 \leq \Box{\blacktriangleleft}\cnomm
\metaand
\forall \nomi_2( \nomi_2 \leq (\Box{\blacktriangleleft}\cnomm \ \wedge {\blacktriangleleft}\cnomm) \rightarrow \blacksquare\cnomm \Rightarrow \Diamond\nomi_2 \leq \cnomn_1)
\Rightarrow
\nomj \leq \Box(\nomi_1 \rightarrow \cnomn_1)
)$\\
iff & $\forall \nomi_1,\cnomn_1(
\nomi_1 \leq \Box{\blacktriangleleft}\cnomm
\metaand
\forall \nomi_2( 
\forall \nomi_3,\cnomn_2( \nomi_3 \leq \Box{\blacktriangleleft}\cnomm \ \wedge {\blacktriangleleft}\cnomm
$\\ &\hfill $
\Rightarrow
\nomi_2 \leq \nomi_3 \rightarrow \blacksquare\cnomn_2)
\Rightarrow\Diamond\nomi_2 \leq \cnomn_1)
\Rightarrow
\nomj \leq \Box(\nomi_1 \rightarrow \cnomn_1)
)$ \\
iff & $\forall \nomi_1,\cnomn_1(
\forall\cnomn_3({\blacktriangleleft}\cnomm\leq\cnomn_3 \Rightarrow \nomi_1 \leq \Box\cnomn_3)
\metaand
\forall \nomi_2( 
\forall \nomi_3,\cnomn_2( \forall\cnomn_4({\blacktriangleleft}\cnomm \leq \cnomn_4 \Rightarrow \nomi_3 \leq \Box\cnomn_4) \metaand \nomi_3 \leq {\blacktriangleleft}\cnomm
$\\ &\hfill $
\Rightarrow
\nomi_2 \leq \nomi_3 \rightarrow \blacksquare\cnomn_2)
\Rightarrow\Diamond\nomi_2 \leq \cnomn_1)
\Rightarrow
\nomj \leq \Box(\nomi_1 \rightarrow \cnomn_1)
)$
\end{tabular}
\par}}
\smallskip

\noindent
where $\nomi_2 \leq (\Box{\blacktriangleleft}\cnomm \ \wedge {\blacktriangleleft}\cnomm) \rightarrow \blacksquare\cnomm$ gets flattened to
$\forall \nomi_3,\cnomn_2( \nomi_3 \leq \Box{\blacktriangleleft}\cnomm \ \wedge {\blacktriangleleft}\cnomm
\Rightarrow
\nomi_2 \leq \nomi_3 \rightarrow \blacksquare\cnomn_2)$.
Hence, $\nomj \nleq \Box(\Box {\blacktriangleleft}\cnomm \rightarrow \Diamond ((\Box{\blacktriangleleft}\cnomm \ \wedge {\blacktriangleleft}\cnomm) \rightarrow \blacksquare\cnomm) ) $ is equivalent to
\[
\begin{array}{l}
\exists \nomi_1,\cnomn_1\Big(
\forall\cnomn_3({\blacktriangleleft}\cnomm \leq \cnomn_3 \Rightarrow \nomi_1 \leq \Box\cnomn_3)
\metaand
\forall \nomi_2\big(\Diamond\nomi_2 \nleq \cnomn_1
\Rightarrow\\
\phantom{\exists \nomi_1,\cnomn_1\Big(}\ \ \ \
\exists \nomi_3,\cnomn_2( \forall\cnomn_4({\blacktriangleleft}\cnomm\leq \cnomn_4 \Rightarrow \nomi_3 \leq \Box\cnomn_4) \ \metaand \nomi_3 \leq{\blacktriangleleft}\cnomm
\metaand
\nomi_2 \nleq \nomi_3 \rightarrow \blacksquare\cnomn_2)
\big)
\metaand
\nomj \nleq \Box(\nomi_1 \rightarrow \cnomn_1)
\Big),
\end{array}
\]
which is a generalized geometric implication in $\mathrm{GA}_2$.
\end{example}
\section{Inverse ALBA for LEs}
\label{sec:inverse}

In the present section, we characterize the syntactic shape of the generalized geometric $\langmeta$-implications, referred to as {\cryptoinductiveinversecorrespondent}, which are semantically equivalent, via ALBA rules, to the ALBA outputs of inductive $\langbase$-inequalities (cf.\ Definition \ref{def:crypto_inductive_inverse_corr}). 
Our strategy proceeds in two steps, which are treated in the following subsections. Specifically, in Section \ref{ssec:inductiveinversecorrespondents} we characterize the syntactic shape of those generalized geometric implications which correspond to very simple Sahlqvist $\langbase^\ast$-inequalities (for some arbitrary and fixed LE-language $\langbase$).
In Section \ref{sec:cryptoinductive}, this characterization is refined to the proper subclass of generalized geometric implications which correspond to inductive $\langbase$-inequalities.

\subsection{\Inductiveinversecorrespondent}
\label{ssec:inductiveinversecorrespondents}

Towards the goal of the present section, we will need the following auxiliary definition:

\begin{definition}
\label{def:inversedisjunct}
A {\em positive} (resp.\ {\em negative}) {\em inverse disjunct} is a $\langmeta$-formula $\theta^+(\purev)$ (resp.\ $\theta^-(\purev)$) defined inductively together with a pure variable $\purev$ as follows:
\begin{description}
    \item[Base case:] $\theta^+(\purev) \coloneqq \nomj \leq \psi$ (resp.\ $\theta^-(\purev) \coloneqq \nomj \nleq \psi$) with $\purev \coloneqq \nomj$, or $\theta^+(\purev) \coloneqq \varphi \leq \cnomm$ (resp.\ $\theta^-(\purev) \coloneqq \varphi \nleq \cnomm$) with $\purev \coloneqq \cnomm$, where $\psi$ and $\varphi$ are a negative and a positive skeleton formula, respectively.
    \item[Conjunction/Disjunction:] $\theta^+(\purev) \coloneqq \bigmetaand_i \theta_i^+(\purev)$ (resp.\ $\theta^-(\purev) \coloneqq \bigmetaor_i \theta_i^-(\purev)$); 
    \item[Quantification 1:] $\theta^+(\purev) \coloneqq\forall\overline\pureu[\nomj \nleq \psi(\overline\pureu) \Rightarrow \bigmetaor_i \theta_i^-(\pureu_i)]$ (resp.\ $\theta^-(\purev) \coloneqq\exists \overline\pureu[\bigmetaand_i \theta_i^+(\pureu_i) \metaand \nomj \nleq \psi(\overline\pureu)]$), where $\nomj \coloneqq \purev$, $\psi$ is a negative definite skeleton formula, and each $\theta_i^-$ (resp.\ $\theta^+_i$) is a negative (resp.\ positive) inverse disjunct where $\nomj$ does not occur. The formula $\psi$ can be just a variable.
    \item[Quantification 2:] $\theta^+(\purev) \coloneqq \forall\overline\pureu[\varphi(\overline\pureu) \nleq \cnomm \Rightarrow \bigmetaor_i \theta_i^-(\pureu_i)]$ (resp.\ $\theta^-(\purev) \coloneqq\exists \overline\pureu[\bigmetaand_i \theta_i^+(\pureu_i) \metaand \varphi(\overline\pureu) \nleq \cnomm]$), where $\cnomm \coloneqq \purev$, $\varphi$ is a positive definite skeleton formula, and each $\theta_i^-$ (resp.\ $\theta^+_i$) is a negative (resp.\ positive) inverse disjunct where $\cnomm$ does not occur. The formula $\varphi$ can be just a variable.
\end{description}
\end{definition}

We will always assume that the quantifiers contained in inverse disjuncts introduce fresh variables.

\begin{remark}
\label{remark:restrictedquantifiersnothere}
The cases in Definition \ref{def:inversedisjunct} which contain quantifiers also contain some negated inequality, e.g.\ $\nomj \nleq \psi(\overline\pureu)$ where the negative skeleton contains all (and only) the variables in $\overline\pureu$ that were introduced by the quantifiers of that case. Such inequalities in the classical setting encode relational atoms that {\em restrict} the quantification. For instance, a formula such as $\forall \cnomn(\nomj \nleq \Box \cnomn \Rightarrow \xi)$ in the classical setting would be interpreted as $\forall y(x \nleq \Box y^c \Rightarrow \xi)$ when $\cnomn$ is interpreted as the cosingleton $\{y\}^c$, and $\nomj$ as the singleton $\{x\}$. Hence, the above mentioned formula is equivalent to $\forall y(\Box \{y\}^c \leq \{x\}^c \Rightarrow \xi)$, which, as discussed in Remark \ref{remark:flatinequalities}, is equivalent to $\forall y(Rxy \Rightarrow \xi)$, i.e.\ it would be one of Kracht's {\em restricted quantifiers}, namely a quantifier over the $R$-successors of $x$, often denoted $(\forall y \rhd x)\xi$ \cite{krachtphdthesis,blackburn2002modal}. Kracht's restricted quantifiers have been generalized in \cite{CoPan22} and \cite{inversedle} using inequalities in the language of ALBA. In the classical and distributive setting, such inequalities correspond to the {\em restricting inequalities} of Kracht (DLE) disjuncts in \cite{CoPan22,inversedle}. Due to the fact that relational atoms are not expressed uniformly across the various relational semantic settings for LE-logics\footnote{For instance, restricted inequalities would represent the right type of relational atom in {\em graph based frames}, but their negations in {\em polarity based frames} (see \cite{conradie2020non} and Section \ref{sec:instances} for more information on these semantics).}, we have decided to drop the {\em restricted quantifier} terminology as it does not apply in the present, purely algebraic setting. 
\end{remark}

\begin{lemma}
\label{lemma:compaction}
Every positive (resp.\ negative) inverse disjunct $\theta^+(\purev)$ (resp.\ $\theta^-(\purev)$) is equivalent to some inequality $\nomj \leq \xi$ (resp.\ $\nomj \nleq \xi$) if $\purev$ is some nominal $\nomj$, and to some inequality $\zeta \leq \cnomm$ (resp.\ $\zeta \nleq \cnomm$) if $\purev$ is some conominal $\cnomm$. 
\end{lemma}
\begin{proof}
We proceed by simultaneous induction on the complexity of $\theta^+$ and $\theta^-$ (we do not discuss the cases of $\theta^-$ as they are analogous to those of $\theta^+$). If $\theta^+$ is an inequality $\nomj \leq \varphi$ (resp.\ $\varphi \leq \cnomm)$ such that $\varphi$ is a negative (resp.\ positive) skeleton, we are done. 

If $\theta^+$ is a conjunction of positive inverse disjuncts sharing the same pure variable $\bigmetaand_i^n\theta^+_i(\purev)$, by applying the inductive hypothesis on each $\theta^+_i$, then $\theta^+$ is equivalent to $\bigmetaand_i^n \nomj \leq \xi_i$ (resp.\ $\bigmetaand_i^n \zeta_i \leq \cnomm)$ if $\purev$ is some nominal $\nomj$ (resp.\ conominal $\cnomm$); hence $\theta^+$ is equivalent to $\nomj \leq \bigwedge_i^n \xi_i$ (resp.\ $\bigvee_i^n \zeta_i \leq \cnomm$). 

Suppose that $\purev$ is some nominal $\nomj$ (resp.\ conominal $\cnomm$) and $\theta^+$ has shape $\forall\overline\pureu(\nomj \nleq \varphi(\overline\pureu) \Rightarrow \bigmetaor^n_i\theta^-_i(\pureu_i))$ (resp.\ $\forall\overline\pureu(\varphi(\overline\pureu) \nleq \cnomm \Rightarrow \bigmetaor^n_i\theta^-_i(\pureu_i))$) such that $\varphi$ is some negative (resp.\ positive) skeleton formula. By taking the contrapositive, $\theta^+$ is equivalent to
\begin{equation}
\label{eq:contrapositive_of_kracht_disjunct}
\forall\overline\pureu(\bigmetaand_i^n\metanot\theta^-_i(\pureu_i) \Rightarrow \nomj \leq \varphi(\overline\pureu)) \quad\quad 
(\mbox{resp.} \ 
\forall\overline\pureu(\bigmetaand_i^n\metanot\theta^-_i(\pureu_i) \Rightarrow \varphi(\overline\pureu) \leq \cnomm),
)
\end{equation}
and by inductive hypothesis, each $\theta_i^-(\pureu_i)$ is equivalent to some inequality $\pureu_i \nleq \xi_i$ if $\pureu_i$ is a nominal, and to some inequality $\zeta_i \nleq \pureu_i$ if it is a conominal; hence $\metanot\theta_i^-(\pureu_i)$ is equivalent to some inequality $\pureu_i \leq \xi_i$ or $\zeta_i \leq \pureu_i$, depending on whether $\pureu_i$ is a nominal or conominal. Let us assume without loss of generality that the first $m$ coordinates of $\varphi$ are positive, and the last $n - m$ are negative. Then \eqref{eq:contrapositive_of_kracht_disjunct} is equivalent to
\begin{equation}
\label{eq:contrapositive_before_inverse_first_approx_eg1}
\forall\overline\pureu(\bigmetaand_{i=1}^m \pureu_i \leq \xi_i
\metaand
\bigmetaand_{i = m+1}^n \zeta_i \leq \pureu_i
\Rightarrow \nomj \leq \varphi(\overline\pureu)) \quad\quad 
(\mbox{resp.} \ 
\forall\overline\pureu(\bigmetaand_{i=1}^m \pureu_i \leq \xi_i
\metaand
\bigmetaand_{i = m+1}^n \zeta_i \leq \pureu_i\Rightarrow \varphi(\overline\pureu) \leq \cnomm),
\end{equation}
Then, by Lemma \ref{lemma:firstapprox}, \eqref{eq:contrapositive_before_inverse_first_approx_eg1} is equivalent to
$\nomj \leq \varphi(\xi_1,\ldots,\xi_m,\zeta_{m+1},\ldots,\zeta_n)$ (resp.\ $\varphi(\xi_1,\ldots,\xi_m,\zeta_{m+1},\ldots,\zeta_n) \leq \cnomm$).
\end{proof}

\begin{example}
\label{eg:compaction1}
The negated outputs of the $\mathsf{Flat}$ procedure (cf.\ Notation \ref{notation:flat}) on the inequalities in the antecedent of ALBA output are straightforwardly equivalent to negative inverse disjuncts. Consider the formula 
\[
\exists \cnomo_1, \nomh_1, \nomh_2
(
\nomh_1 \leq \Box\Box \nomj_2 
\metaand
\nomh_2 \leq \nomj_3
\metaand
\forall \nomh_3(\Diamond\nomh_3 \nleq \cnomo_1 \Rightarrow \nomh_3 \nleq {\rhd} \nomj_2)
\metaand
\nomj_1 \nleq (\nomh_1 \circ \nomh_2) \circback \cnomo_1
),
\]
from Example \ref{eg:flattification1} (in that example, the subformula $\forall \nomh_3(\Diamond\nomh_3 \nleq \cnomo_1 \Rightarrow \nomh_3 \nleq {\rhd} \nomj_2)$ occurs contrapositively, namely $\forall \nomh_3(\nomh_3 \leq {\rhd} \nomj_2 \Rightarrow \Diamond\nomh_3 \leq \cnomo_1)$). By Lemma \ref{lemma:compaction}, the positive inverse disjunct $\theta^+(\cnomo_1)\coloneqq \forall \nomh_3(\Diamond\nomh_3 \nleq \cnomo_1 \Rightarrow \nomh_3 \nleq {\rhd} \nomj_2)$ is equivalent to the inequality $\Diamond{\rhd}\nomj_2 \leq \cnomo_1$. Hence, again by the procedure described in Lemma \ref{lemma:compaction}, 
\[
\exists \cnomo_1, \nomh_1, \nomh_2
(
\nomh_1 \leq 
\metaand
 \leq \nomj_3
\metaand
\Diamond{\rhd}\nomj_2 \leq \cnomo_1
\metaand
\nomj_1 \nleq (\nomh_1 \circ \nomh_2) \circback \cnomo_1
),
\]
is equivalent to the negated inequality $\nomj_1 \nleq (\Box\Box \nomj_2  \circ \nomj_3) \circback \Diamond{\rhd}\nomj_2$.
\end{example}

The following definition characterizes the syntactic shape of those $\langmeta$-formulas which result from the procedure $\mathsf{Flat}$ (cf.\ Notation \ref{notation:flat}) applied to the ALBA outputs of inductive $\langbase$-inequalities. 
This syntactic shape includes those geometric implications in $\langmeta$ which can be effectively transformed via ALBA rules into ALBA outputs of inductive LE inequalities (the \cryptoinductiveinversecorrespondent, cf.~Definition \ref{def:crypto_inductive_inverse_corr}). 
In Section \ref{ssec:kripke}, we will show that these formulas project over the class of generalized Kracht formulas \cite{kikot}, and hence Kracht formulas \cite{krachtphdthesis,blackburn2002modal}, in the setting of Kripke frames for classical modal logic.

\begin{definition}
\label{def:inductiveinvcorr}
A $\langmeta$-formula is {\em \aninductiveinversecorrespondent}  if it is closed with shape
\[
\forall \overline \nomj, \overline\cnomm, \overline\nomi, \overline\cnomn \left( 
\bigmetaor_{\nomj_i \mbox{ in } \overline\nomj}
\theta_{\nomj_i}^-(\nomj_i)
\metaor
\bigmetaor_{\cnomm_i \mbox{ in } \overline\cnomm}
\theta_{\cnomm_i}^-(\cnomm_i)
\metaor
\varphi \leq \psi
\right),
\]
where each $\theta_{\nomj_i}^-$ and $\theta_{\cnomm_i}^-$ is a negative inverse disjunct (cf.\ Definition \ref{def:inversedisjunct}), and
\begin{enumerate}
    \item the variables in $\overline\nomi$ (resp.\ $\overline\cnomn$) occur negatively (resp.\ positively) in their (possibly negated) inequalities;\footnote{We remind that given an inequality $\alpha\leq\beta$, to understand whether a variable is positive or negative, it is sufficient to compute the signed generation trees $+\alpha$ and $-\beta$ and check the sign of the occurrences of the variable. Given a negated inequality $\alpha\nleq\beta$, the signed generation trees $-\alpha$ and $+\beta$ shall be used.}
    \item $\varphi$ (resp.\ $\psi$) is a scattered positive (resp.\ negative) skeleton formula; 
    \item each variable in $\overline\nomj$ and $\overline\nomi$ (resp.\ $\overline\cnomm$ and $\overline\cnomn$) occurs positively (resp.\ negatively) in exactly one of $\varphi$ and $\psi$. 
\end{enumerate}
Pure variables in $\overline\nomj$ and $\overline\nomi$ (resp.\ $\overline\cnomm$ and $\overline\cnomn$) may also coincide.\footnote{The shape presented above is not properly in $\langmeta$, but it can be equivalently rewritten as 
\[
\forall \nomj_\varphi, \cnomm_\psi, \overline \nomj, \overline\cnomm, \overline\nomi, \overline\cnomn \left( 
\nomj_\varphi \leq \varphi \metaand \psi \leq \cnomm_\psi
\Rightarrow
\bigmetaor_{\nomj_i \mbox{ in } \overline\nomj}
\theta_{\nomj_i}^-(\nomj_i)
\metaor
\bigmetaor_{\cnomm_i \mbox{ in } \overline\cnomm}
\theta_{\cnomm_i}^-(\cnomm_i)
\metaor
\nomj_\varphi \leq \cnomm_\psi
\right),
\]
which is a generalized geometric implication in $\langmeta$.}
\end{definition}

Since the negated outputs of the $\mathsf{Flat}$ procedure on the inequalities in the antecedent of an ALBA output are straightforwardly equivalent to negative inverse disjuncts, the following Proposition readily follows.

\begin{prop}
\label{prop:inductive_inductiveinverse}
Every inductive $\langbase$-inequality is equivalent to a $\langmeta$-{\inductiveinversecorrespondent}.
\end{prop}

The following property of scattered very simple Sahlqvist inequalities is essential for proving the main result of this section.
\begin{lemma}
\label{lemma:vss_substitute_nomcnom_var}
Any  scattered very simple $\varepsilon$-Sahlqvist inequality $(\varphi \leq \psi)[!\overline p / !\overline x, !\overline q / !\overline y, \overline\gamma / !\overline z, \delta / !\overline w]$\footnote{We remind the reader that scattered very simple $\varepsilon$-Sahlqvist inequalities are  Sahlqvist inequalities (cf.\ Definition \ref{def:sahlqvist_and_vss}) such that there is at most one  $\varepsilon$-critical  occurrence for each variable, and such occurrence, if any, is the leaf of a Skeleton branch.} is equivalent to the pure inequality
\[
(\varphi \leq \psi)[!\overline p / !\overline x, !\overline q / !\overline y, \overline\gamma / !\overline z, \overline\delta / !\overline w][\overline{\nomj_p} / \overline p, \overline{\cnomm_q} / \overline q].
\]
\end{lemma}
\begin{proof}
By Corollary \ref{cor:albaoutput}, the validity of  $(\varphi \leq \psi)[!\overline p / !\overline x, !\overline q / !\overline y, \overline\gamma / !\overline z, \delta / !\overline w]$ is equivalent to 
\[
\forall \overline\nomj_p, \overline \cnomm_q, \overline \nomj_z, \overline \cnomm_w 
\left(
\left(\bigmetaand_{i=1}^{n_z}\nomj_{z_i}\leq \gamma_i \metaand 
\bigmetaand_{i=1}^{n_w} \delta_i\leq \cnomm_{w_i}\right)\left[\overline{\nomj_p}/\overline{p}, \overline{\cnomm_q}/\overline{q}\right] \Rightarrow
\hfill
(\varphi \leq \psi)[!\overline{\nomj_p} / !\overline x, !\overline{\cnomm_q} / !\overline y, \overline{\nomj_z} / !\overline z, \overline{\cnomm_w} / !\overline w]
\right),
\]
because the sets $\mathsf{Mv}(p)$ and $\mathsf{Mv}(q)$ only contain $\nomj_p$ and $\cnomm_q$ for every $p$ and $q$. Notice that, since the input inequality is scattered, it is possible to rename nominals $\overline {\nomj_x}$ (resp.\ conominals $\overline {\cnomm_y}$) to $\overline{\nomj_p}$ (resp.\ $\overline{\cnomm_q}$). 
By Lemma \ref{lemma:firstapprox}, the pure quasi-inequality above is  equivalent to
\[
\begin{array}{rl}
&\forall \overline\nomj_p, \overline \cnomm_q, \overline \nomj_z, \overline \cnomm_w 
\left(
(\varphi \leq \psi)[!\overline{\nomj_p} / !\overline x, !\overline{\cnomm_q} / !\overline y, \overline{\gamma}[\overline{\nomj_p}/\overline{p}, \overline{\cnomm_q}/\overline{q}] / !\overline z, \overline\delta [\overline{\nomj_p}/\overline{p}, \overline{\cnomm_q}/\overline{q}], \overline{\cnomm_w} / !\overline w]
\right), \\
\mbox{i.e.} \quad & (\varphi \leq \psi)[!\overline p / !\overline x, !\overline q / !\overline y, \overline\gamma / !\overline z, \overline\delta / !\overline w][\overline{\nomj_p} / \overline p, \overline{\cnomm_q} / \overline q].
\end{array}
\]
\end{proof}

\begin{theorem}
\label{thm:invcorr}
Any $\langmeta$-{\inductiveinversecorrespondent} $\forall \overline \nomj, \overline\cnomm, \overline\nomi, \overline\cnomn \left( 
\bigmetaor_{i}^{n_j}
\theta_{\nomj_i}^-(\nomj_i)
\metaor
\bigmetaor_{i}^{n_m}
\theta_{\cnomm_i}^-(\cnomm_i)
\metaor
\varphi \leq \psi
\right)$ is equivalent to some $\langbase^*$-very simple Sahlqvist inequality. Moreover, such very simple Sahlqvist inequality can be effectively computed.
\end{theorem}
\begin{proof}
Since the metalanguage of $\langmeta$ is classical, any inductive inverse correspondent can be equivalently rewritten as follows:
\begin{equation}
\label{eq:inversecorr_contrap}
\forall \overline \nomj, \overline\cnomm, \overline\nomi, \overline\cnomn \left( 
\bigmetaand_{i}^{n_j}\metanot
\theta_{\nomj_i}^-(\nomj_i)
\metaand
\bigmetaand_{i}^{n_m}\metanot
\theta_{\cnomm_i}^-(\cnomm_i)
\Rightarrow
\varphi \leq \psi
\right).
\end{equation}
By Lemma \ref{lemma:compaction}, each $\metanot\theta^-_{\nomj_i}$ (resp.\ $\metanot\theta^-_{\cnomm_i}$) is equivalent to some inequality $\nomj_i \leq \xi_i$ (resp.\ $\zeta_i \leq \cnomm_i$)
; therefore \eqref{eq:inversecorr_contrap} can be rewritten as
\begin{equation}
\label{eq:inversecorr_contrap_substsanteced}
\forall \overline \nomj, \overline\cnomm, \overline\nomi, \overline\cnomn \left( 
\bigmetaand_{i}^{n_j}\nomj_i\leq\xi_i
\metaand
\bigmetaand_{i}^{n_m}\zeta_i\leq\cnomm_i
\Rightarrow
\varphi \leq \psi
\right).
\end{equation}
The quasi inequality above can be recognized as the ALBA output of a very simple Sahlqvist inequality to which Lemma \ref{lemma:vss_substitute_nomcnom_var} applies. Indeed,
by introducing for each $\nomi$ in $\overline\nomi$ (resp.\ $\cnomn$ in $\overline\cnomn$) a propositional variable $p_\nomi$ (resp.\ $q_\cnomn$), consider now the following formula
\begin{equation}
\label{eq:inversecorr_introducedvars}
\begin{split}
\forall \overline{p_\nomi}, \overline{q_\cnomn}, \overline\nomj, \overline\cnomm, \overline\nomi, \overline\cnomn
\Big(
\bigmetaand_{i}^{n_j}\nomj_i\leq\xi_i\left[ \overline{p_\nomi}/\overline\nomi, \overline{q_\cnomn}/\overline\cnomn \right]
\metaand
\bigmetaand_{i}^{n_m}\zeta_i\left[ \overline{p_\nomi}/\overline\nomi, \overline{q_\cnomn}/\overline\cnomn \right]\leq\cnomm_i \metaand 
\bigmetaand_{\nomi \mbox{ in } \overline\nomi} \nomi \leq p_\nomi 
\metaand 
\bigmetaand_{\cnomn \mbox{ in } \overline\cnomn} q_\cnomn \leq \cnomn
\Rightarrow \\
\varphi \leq \psi
\Big).
\end{split}
\end{equation}
Thanks to the hypothesis on the polarity of the occurrences of each $\nomi$ and $\cnomn$ in each $\xi_i$ and $\zeta_i$ (cf.\ Definition \ref{def:inductiveinvcorr}), by Lemma \ref{lemma:ackermann} and Remark \ref{remark:definiteskeleton}, \eqref{eq:inversecorr_introducedvars} and \eqref{eq:inversecorr_contrap_substsanteced} are equivalent. In particular, since the first approximation (cf.\ Lemma \ref{lemma:firstapprox}) preserves the polarities of the variables in the inequality, also the procedure in Lemma \ref{lemma:compaction} does so; hence the hypothesis on the polarities of the pure variables in Definition \ref{def:inductiveinvcorr} are preserved. Thus, the conditions to apply Lemma \ref{lemma:ackermann} are satisfied.

Since some of the variables in $\overline\nomi$ (resp.\ $\overline\cnomn$) might also occur in $\overline\nomj$ (resp.\ $\overline\cnomm$), let us denote by $\overline\nomi'$ (resp.\ $\overline\cnomn'$) the sequences of (co)nominals in $\overline\nomi$ (resp.\ $\overline\cnomn$) which do not also occur in $\overline\nomj$ (resp.\ $\overline\cnomm$). For the pure variables $\nomj_i$ (resp.\ $\cnomm_i$) that occur also in $\overline\nomi$ (resp.\ $\overline\cnomn$), let $\xi_i'$ (resp.\ $\zeta_i'$) denote $\xi_i \wedge p_{\nomj_i}$ (resp.\ $\zeta_i \vee q_{\cnomm_i}$), and for the remaining $\nomj_j$ (resp.\ $\cnomm_j$) let $\xi_j' \coloneqq \xi_j$ (resp.\ $\zeta_j' \coloneqq \zeta_j$). Then, \eqref{eq:inversecorr_introducedvars} is equivalent to
\begin{equation}
\label{eq:inversecorr_introducedvars_purged}
\forall \overline{p_\nomi}, \overline{q_\cnomn}, \overline\nomj, \overline\cnomm, \overline\nomi', \overline\cnomn'
\left( \!
\left( \!\bigmetaand_{i}^{n_j}\nomj_i\leq\xi_i'
\metaand
\bigmetaand_{i}^{n_m}\zeta_i'\leq\cnomm_i\right)\!\left[ \overline{p_\nomi}/\overline\nomi', \overline{q_\cnomn}/\overline\cnomn' \right] \!\metaand 
\bigmetaand_{\nomi \mbox{ in } \overline\nomi'} \nomi \leq p_\nomi 
\metaand 
\bigmetaand_{\cnomn \mbox{ in } \overline\cnomn'} q_\cnomn \leq \cnomn
\Rightarrow
\varphi \leq \psi
\right).
\end{equation}
By Corollary \ref{cor:firstapprox_inductive}, \eqref{eq:inversecorr_introducedvars_purged} is equivalent to
\begin{equation}
\label{eq:inversecorr_contrap_noj_nom}
\forall \overline{p_\nomi}, \overline{q_\cnomn} \left(
\varphi \leq \psi
\right)
\left[\overline{
\xi_i'
/\nomj_i
},\overline{
\zeta_i' 
/\cnomm_i}
\right]
\left[
\overline{p_\nomi}/\overline\nomi, \overline{q_\cnomn}/\overline\cnomn
\right],
\end{equation}
which is a very simple Sahlqvist inequality for the order type $\varepsilon$ such that $\varepsilon(p_\nomi) = 1$ and $\varepsilon(q_\cnomn) = \partial$ for every $p_\nomi$ in $\overline{p_\nomi}$ and $q_\cnomn$ in $\overline{q_\cnomn}$.
\end{proof}

The proposition and the theorem above yield the following Corollary.
\begin{corollary}
\label{cor:inductive_to_vss}
Every inductive $\langbase$-inequality is equivalent to a very simple Sahlqvist $\langbase^\ast$-inequality.
\end{corollary}

\begin{example}
\label{eg:goranko_inv_to_vss}
The following $\langmeta$ formula
\[
\forall \nomj, \cnomm(
\exists \nomi_1 \left(
\Diamond\nomi_1 \nleq \cnomm
\metaand
    \forall \cnomn_1 \left(
    \nomi_1 \nleq \Box\Box\cnomn_1
    \Rightarrow
        \exists \nomi_2 \left(
        \Diamondblack(\Diamond\nomi_2 \circ \Diamondblack \nomj) \nleq \cnomn_1
        \metaand
        \nomi_2 \leq \Box\nomj
        \right)
    \right)
\right) 
\metaor \nomj \leq \cnomm
)
\]
satisfies all the conditions in Definition \ref{def:inductiveinvcorr}.\footnote{To be precise, the subformula
\[
\exists \nomi_2 \left(
\Diamondblack(\Diamond\nomi_2 \circ \Diamondblack \nomj) \nleq \cnomn_1
\metaand
\nomi_2 \leq \Box\nomj
\right)
\]
should be rewritten as the (trivially) equivalent formula
\[
\exists \nomi_2,\nomi_3 \left(
\Diamondblack(\Diamond\nomi_2 \circ \Diamondblack \nomi_3) \nleq \cnomn_1
\metaand
\nomi_2 \leq \Box\nomj
\metaand
\nomi_3 \leq \nomj
\right)
\]
to be in the shape described in Definition \ref{def:inductiveinvcorr}.
}
The procedure in Lemma \ref{lemma:compaction} computes the following three equivalent rewritings
\smallskip

{{\centering
\begin{tabular}{rl}
& $\forall \nomj, \cnomm(
\exists \nomi_1 \left(
\Diamond\nomi_1 \nleq \cnomm
\metaand
    \forall \cnomn_1 \left(
    \nomi_1 \nleq \Box\Box\cnomn_1
    \Rightarrow
    \Diamondblack(\Diamond\Box \nomj \circ \Diamondblack \nomj) \nleq \cnomn_1
    \right)
\right) 
\metaor \nomj \leq \cnomm
)$\\
iff & $\forall \nomj, \cnomm(
\exists \nomi_1 \left(
\Diamond\nomi_1 \nleq \cnomm
\metaand
\nomi_1 \leq \Box\Box\Diamondblack(\Diamond\Box \nomj \circ \Diamondblack \nomj)
\right) 
\metaor \nomj \leq \cnomm
)$\\
iff & $\forall \nomj, \cnomm(
\Diamond\Box\Box\Diamondblack(\Diamond\Box \nomj \circ \Diamondblack \nomj) \nleq \cnomm \metaor \nomj \leq \cnomm
)$.
\end{tabular}
\par}}
\smallskip

\noindent The procedure described in Theorem \ref{thm:invcorr} introduces a new variable $p_\nomj$, as $\nomj$ does not occur as main variable in any inverse disjunct. By Lemma \ref{lemma:ackermann}, the formula above is equivalent to
\[
\forall p_\nomj \forall \nomj, \cnomm(
\Diamond\Box\Box\Diamondblack(\Diamond\Box p_\nomj \circ \Diamondblack p_\nomj) \leq \cnomm 
\metaand
\nomj \leq p_\nomj
\Rightarrow \nomj \leq \cnomm
),
\]
and, by Lemma \ref{lemma:ackermann}, it is equivalent to the following very simple Sahlqvist $\langbase^*$-inequality
\[
p_\nomj \leq \Diamond\Box\Box\Diamondblack(\Diamond\Box p_\nomj \circ \Diamondblack p_\nomj).
\]
\end{example}

\subsection{Crypto-inductive inequalities}
\label{sec:cryptoinductive}

In Section \ref{ssec:inductiveinversecorrespondents}, we have shown how {\inductiveinversecorrespondent} (cf.\  Definition \ref{def:inductiveinvcorr}) are equivalent to very simple Sahlqvist $\langbase^*$-inequalities (see Definition \ref{def:fully_residuated_language}). However, it is well known (see for instance \cite{CoPan22}) that very simple Sahlqvist $\langbase^*$-inequalities  are more expressive than inductive $\langbase$-inequalities, in the sense that, via semantic equivalence, inductive $\langbase$-inequalities correspond to a proper subclass of  very simple Sahlqvist $\langbase^\ast$-inequalities. In the present section, we generalize the characterization of crypto-inductive $\mathcal{L}_{\mathrm{DLE}}$-inequalities  (cf.~\cite[Section 3]{CoPan22}) to the setting of LE-logics, characterizing a class of very simple Sahlqvist $\langbase^\ast$-inequalities which can be effectively shown, via application of ALBA-rules, to be semantically equivalent to inductive $\langbase$-inequalities.

\begin{definition}
\label{def:adj:res:conservative}
A node labelled with $+g$ (resp.\ $-f$)
 is \emph{residuation-conservative in the $i$-th coordinate} for $1 \leq i \leq n_g$ ($1 \leq i \leq n_f$) if the residual $g_i^{\flat}$ of $g$ (resp.\ $f_i^{\sharp}$ of $f$) belongs to $\langbase$.
\end{definition}

\begin{definition}
\label{def:mvtree}
For any order type $\varepsilon$, any strict order on variables $\Omega$, and for any leaf $l$ labelled with a variable $p$ of a signed generation tree $\pm s$, let $\mathsf{MVTree}^\varepsilon_\Omega(l)$ denote the largest subtree $\pm t$ containing $l$ such that:
\begin{enumerate}
    \item the path to $l$ contains only negative (resp.\ positive) nodes labelled by connectives in $\mathcal{F}^\ast$ (resp.\ $g\in\mathcal{G}^\ast$) which are residuation conservative with respect to the coordinate that contains $l$,
    \item no variable $q >_\Omega p$ and occurs in $\pm t$, as well as any $\varepsilon$-critical leaf,
    \item the connective labelling the topmost node in the tree is in $(\mathcal{F}^*\cup\mathcal{G}^*) \setminus (\mathcal{F}\cup\mathcal{G})$,
    \item the topmost node in $\mathsf{MVTree}^\varepsilon_\Omega(l)$ does not lie in the path to the leaf $l'$ in any other $\mathsf{MVTree}^\varepsilon_\Omega(l')$ with $l'$ occurring more on the left than $l$.
\end{enumerate}
\end{definition}

Informally, the definition above intends to identify the (sub)formulas $\mathsf{LA}(\alpha_i)(j_{x_i}, \overline z)$ (resp.\ $\mathsf{RA}(\beta_i)(m_{y_i}, \overline z)$) in the minimal valuations $\mathsf{Mv}(p)$ or $\mathsf{Mv}(q)$ which are substituted by the Ackermann Lemma in an ALBA run. 
For instance, consider the very simple Sahlqvist inequality equivalent to the non-distributive version of Frege's inequality from Example \ref{eg:alba_fregeLE}, namely $
p_{\nomj_1} \leq p_{\nomj_2} \rightharpoonup (p_{\nomj_3} \rightharpoonup ((p_{\nomj_3} \bullet p_{\nomj_2}) \bullet (p_{\nomj_3} \bullet p_{\nomj_1})))$. Given the order $p_{\nomj_3} <_\Omega p_{\nomj_2} <_\Omega p_{\nomj_1}$ and order type $\varepsilon(p) = \varepsilon(q) = \varepsilon(r) = 1$, the $\mathsf{MVTree}_\Omega^\varepsilon$ of both non-critical occurrences of $p_{\nomj_3}$ contains just $p_{\nomj_3}$, the $\mathsf{MVTree}_\Omega^\varepsilon$ of the non-critical occurrence of $p_{\nomj_2}$ is $p_{\nomj_3} \bullet p_{\nomj_2}$, and the $\mathsf{MVTree}_\Omega^\varepsilon$ of the non-critical occurrence of $p_{\nomj_1}$ is $(p_{\nomj_3} \bullet p_{\nomj_2}) \bullet (p_{\nomj_3} \bullet p_{\nomj_1})$. These trees correspond to the minimal valuations of the variables in the original inductive inequality in Example \ref{eg:alba_fregeLE}, namely $\mathsf{Mv}(p)=\{\nomj_3\}, \mathsf{Mv}(q) = \{ \nomj_3 \bullet \nomj_2 \}$, and $\mathsf{Mv}(r) = \{ ({\nomj_3} \bullet {\nomj_2}) \bullet ({\nomj_3} \bullet {\nomj_1}) \}$.

More in general, if an inequality is obtained via ALBA by transforming an inductive $\langbase$-inequality into a very simple Sahlqvist $\langbase^\ast$-inequality {\inductiveinversecorrespondent} (cf.~Corollary \ref{cor:inductive_to_vss}), and 
 $l$ is a non-critical leaf, $\mathsf{MVTree}_\Omega^\varepsilon(l)$ detects a  subformula which possibly corresponds to one of the (parts of the) minimal valuations $\mathsf{Mv}(p)$ or $\mathsf{Mv}(q)$ which is substituted by the Ackermann Lemma in an ALBA run. Indeed, the minimal valuation for such a positive (resp.\ negative) variable $p$ (resp.\ $q$) is the join (resp.\ meet) of adjoints of definite PIAs (see the definition of $\mathsf{Mv}$ right above Lemma \ref{lemma:ackermann}). Since every (co)nominal is eventually substituted by a variable (cf.~\eqref{eq:inversecorr_introducedvars} in Theorem \ref{thm:invcorr}) the $\mathsf{MVTree}_\Omega^\varepsilon$ of a leaf $l$ labelled by a variable $r$ ideally corresponds to the formula $\mathsf{LA}(\alpha_i)(j_{x_i}, \overline z)$ (resp.\ $\mathsf{RA}(\beta_i)(m_{y_i}, \overline z)$) where $j_{x_i}$ (resp.\ $m_{y_i}$) is the (co)nominal which is substituted by $r$.

Since the syntax tree of the adjoint of a definite PIA (cf.~Definition \ref{def:RA_and_LA}) contains a path of nodes which are residuals of nodes in the PIA formula, each $\mathsf{MVTree}_\Omega^\varepsilon(l)$ must contain a path of nodes which are residuation conservative, thus motivating Condition 1 of the definition above. Furthermore, since every definite PIA subformula of an inductive inequality contains a unique critical variable, which is $\Omega$-larger than all the (non-critical) variables occurring in the given PIA subformula, the minimal valuation of such critical variable cannot receive substitutions of minimal valuations of variables which are $\Omega$-larger than the given variable by means of the Ackermann Lemma, which justifies Condition 2. Condition 3 guarantees that the nodes which already belong to $\langbase$ are included in $\mathsf{MVTree}_\Omega^\varepsilon(l)$ only if they are in the scope of a node which does not belong to $\langbase$. Notice that Condition 2 ensures that, if $l$ and $l'$ are leaves labelled by  different variables, the path from the root of $\mathsf{MVTree}_\Omega^\varepsilon(l)$ to $l$ does not intersect the path from the root of $\mathsf{MVTree}_\Omega^\varepsilon(l')$ to $l'$; indeed, if the two paths intersected, both trees would contain both leaves, contradicting the fact that, in every very simple Sahlqvist $\langbase^\ast$-inequality which is ALBA-equivalent to some inductive one, the variable labelling one of the two leaves is $\Omega$-larger than the other.
However, Condition 2  does not guarantee  this non-intersection property to hold for leaves labelled by the same variable. This is guaranteed by Condition 4 which establishes an arbitrary order between the leaves labelled by the same variables. Nonetheless, in principle any order (other than the ordering from left-to-right of the leaves) can work, and one has been picked just to have a precise definition.

\begin{definition}\label{def:Crypto:inductive}
An $\langbase^{\ast}$-inequality $\phi \leq \psi$ is  \emph{crypto $\langbase$-inductive}\footnote{The reference to the base language $\langbase$ is necessary, since different base languages can have the same fully residuated language $\langbase^{\ast}$, and moreover, it is easy to find examples of inequalities that are crypto inductive w.r.t.~one base language but not w.r.t.~another (cf.~Remark \ref{rmk:crypto_depends_on_base_language}).} if it is a  very simple $\varepsilon$-Sahlqvist $\langbase^{\ast}$-inequality and, in the signed generation trees $+\phi$ and $- \psi$:
\begin{enumerate}
    \item All $\varepsilon$-critical branches contain only signed connectives from $\langbase$,
    \item There exists a strict partial order $\Omega$ on the propositional variables occurring in $\phi \leq \psi$, such that every occurrence of a connective in $(\mathcal{F}^*\cup\mathcal{G}^*) \setminus (\mathcal{F}\cup\mathcal{G})$ lies in the path to a non-critical leaf $l$ from the root of $\mathsf{MVTree}_\Omega^\varepsilon(l)$, and it contains a $\Omega$-largest variable in its scope.
\end{enumerate}
\end{definition}

The order $\Omega$ on the variables of a crypto $\langbase$-inductive inequality induces a rather natural strict order on the non-critical leaves which is defined as the smallest strict order $\prec$ satisfying the following conditions: if two leaves $l_1$ and $l_2$ are labelled by two (different) variables $v_1$ and $v_2$ such that $v_1 <_\Omega v_2$, then $l_1 \prec l_2$; if two different leaves $l_1$ and $l_2$ are labelled by the same variable and $l_1$ occurs in $\mathsf{MVTree_\Omega^\varepsilon}(l_2)$, then $l_1 \prec l_2$.

\begin{remark}
\label{rmk:crypto_depends_on_base_language}
Given two different LE languages $\mathcal{L}_1$ and $\mathcal{L}_2$ such that $\mathcal{L}_1^\ast = \mathcal{L}_2^\ast$, an inequality might be crypto $\mathcal{L}_1$-inductive, but not crypto $\mathcal{L}_2$-inductive. Consider for instance $\mathcal{L}_1$ with connectives $\mathcal{F}_1 = \{ \circ \}$, $\mathcal{G}_1 = \{ \star \}$ and $\mathcal{L}_2$ with connectives $\mathcal{F}_1 = \{ \starback \}$, $\mathcal{G}_1 = \{ \circback \}$.\footnote{In $\mathcal{L}_1^\ast$, $\circfor$ and $\circback$ are the adjoints of $\circ$, and $\starfor$ and $\starback$ are the adjoints of $\star$. In $\mathcal{L}_2^\ast$, $\circfor$ and $\circ$ are the adjoints of $\circback$, and $\starfor$ and $\star$ are the adjoints of $\starback$.}
The inequality $p \starback q \leq p \circback q$ is plainly crypto $\mathcal{L}_2$-inductive, but it is not crypto $\mathcal{L}_1$-inductive.
\end{remark}

\begin{prop}
\label{prop:cryptotoinductive}
	Every crypto $\langbase$-inductive inequality is equivalent to an inductive $\langbase$-inequality.
\end{prop}
\begin{proof}
Let $\phi \leq \psi$ be a crypto $\langbase$-inductive inequality and let $\varepsilon$ and $\Omega$ be an order type and a strict partial order satisfying Definition \ref{def:Crypto:inductive}. 

Consider now the equivalence relation $\equiv$ on the non-critical leaves of the inequality defined as follows: $l \equiv l'$ iff $\mathsf{MVTree_\Omega^\varepsilon}(l)=\mathsf{MVTree_\Omega^\varepsilon}(l')$. Item 2 of Definition \ref{def:Crypto:inductive} guarantees that $\equiv$ is well defined w.r.t.~$\prec$, in the sense that $l_1\prec l_2$ whenever $l_1' \prec l_2'$ for some $l_1'\equiv l_1$ and  $l_2'\equiv l_2$. This also guarantees that  $\equiv$-equivalent leaves  
are labelled by the same variable.\footnote{Indeed, the variable labelling $l_1$ in $\mathsf{MVTree}_\Omega^\varepsilon(l_1)$ is the $\Omega$-largest variable in $\mathsf{MVTree}_\Omega^\varepsilon(l_1)$ whose existence is guaranteed by item 2 of Definition \ref{def:Crypto:inductive}, since, otherwise, item 2 of Definition \ref{def:mvtree} would be violated. The same argument can be applied to $l_2$ and $\mathsf{MVTree}_\Omega^\varepsilon(l_2)$; hence, the variables labelling $l_1$ and $l_2$ must coincide whenever $l_1 \equiv l_2$.}
 
Let us now consider the set $\{[l_1],\ldots,[l_n]\}$ of $\equiv$-classes, and, without loss of generality, let us assume that $i < j$ implies $l_i \prec l_j$. Following this order, for each set of leaves $[l_i]$ apply the Ackermann Lemma (cf.\ Lemma \ref{lemma:ackermann}) to extract all the occurrences of $\mathsf{MVTree}_\Omega^\varepsilon(l_i)$ from the inequality. More specifically, at each step the inequality $(\varphi'\leq\psi')[{\mathsf{MVTree}_\Omega^\varepsilon(l_i)}/!\overline{z}]$, where $!\overline z$ is a sequence of unique placeholder variables for the occurrences of $\mathsf{MVTree}_\Omega^\varepsilon(l_i)$, can be equivalently rewritten as $p_i \leq \mathsf{MVTree}_\Omega^\varepsilon(l_i) \Rightarrow (\varphi'\leq\psi')[p_i/!\overline{z}]$ or $\mathsf{MVTree}_\Omega^\varepsilon(l_i) \leq q_i \Rightarrow (\varphi'\leq\psi')[q_i/!\overline{z}]$, depending on whether the topmost symbol of $\mathsf{MVTree}_\Omega^\varepsilon(l_i)$ has a positive or negative sign in the signed generation tree.\footnote{Note that, by item 2 of Definition \ref{def:Crypto:inductive}, all occurrences of $\mathsf{MVTree}_\Omega^\varepsilon(l_i)$ have the same sign in the given crypto inequality.} 
The application of the steps discussed above equivalently transforms  the input inequality into a quasi inequality of the following form:
\begin{equation}
\label{eqn:crypto_extraction}
\bigmetaand_j p_{i_j} \leq \alpha_j(v_{i_j}) \metaand \bigmetaand_j \beta_j(v_{k_j}) \leq q_{k_j} \Rightarrow (\phi'\leq \psi')[\overline{p}/!\overline{x}, \overline{q}/!\overline{y}],
\end{equation}
where $v_{h}$ is the variable labelling every node in $[l_{h}]$, and
each $\alpha_j(v_{i_j})$ (resp.\ $\beta_j(v_{k_j})$) is a positive (resp.\ negative) definite PIA (cf.\ Notation \ref{notation:compactinductive}), by item 1 of Definition \ref{def:mvtree}, which can contain connectives in $(\mathcal{F}^*\cup\mathcal{G}^*) \setminus (\mathcal{F}\cup\mathcal{G})$ only in the path of PIA nodes from the root to the leaf $l_{i_j}$ (resp.\ $l_{k_j}$), and nowhere else by items 2 and 4 of Definition \ref{def:mvtree}. By applying adjunction to each $\alpha_j$ and $\beta_j$ (cf.~Definition \ref{def:RA_and_LA}), the quasi-inequality above is equivalent to the following\footnote{In \eqref{eqn:crypto_ackermann_shape}, we adopt the convention that $a \leq^1 b$ iff $a \leq b$, and $a \leq^\partial b$ iff $b \leq a$, for all terms $a$ and $b$. Note that, by item 1 of Definition \ref{def:Crypto:inductive}, each $v_{i_j}$ (resp.\ $v_{k_j}$) occurs in $\alpha_j$ (resp.\ $\beta_j$) with sign $\varepsilon^\partial(v_{i_j})$ (resp.\ $\varepsilon^\partial(v_{k_j})$). To justify the superscripts in the  inequalities of \eqref{eqn:crypto_ackermann_shape}, consider the following case (the other three cases being similar). If $\varepsilon(v_{i_j}) = 1$, then $v_{i_j}$ occurs negatively in $\alpha_j$. Since the root of $\alpha_j$ is positive in the original inequality, $\alpha_j$ is antitone in $v_{i_j}$; hence the adjunction is a Galois connection which reverses the order.}
\begin{equation}
\label{eqn:crypto_ackermann_shape}
\bigmetaand_j \mathsf{LA}(\alpha_j)(p_{i_j}) \leq^{\varepsilon^\partial(v_{i_j})} v_{i_j} \metaand \bigmetaand_j v_{k_j} \leq^{\varepsilon(v_{k_j})} \mathsf{RA}(\beta_j)(q_{k_j}) \Rightarrow (\phi'\leq \psi')[\overline{p}/!\overline{x}, \overline{q}/!\overline{y}].
\end{equation}
Note that every $\mathsf{LA}(\alpha_j)$ and every $\mathsf{RA}(\beta_j)$ contains only nodes in $\langbase$.

Similarly to the case in which  ALBA is used to eliminate proposition variables in favour of nominal and conominal variables, we proceed to eliminate the original variables $v$ in favour of the new variables $p$ and $q$ by computing a minimal valuation set $M_h$ for each $v_h$ depending on $\varepsilon(v_h)$. More precisely, for each $v_h$, $M_h := \{ \mathsf{LA}(\alpha_j)(p_{i_j}) : v_{i_j} = v_h \}$ if $\varepsilon(v_h) = 1$, and $M_h := \{ \mathsf{RA}(\beta_j)(q_{k_j}) : v_{k_j} = v_h \}$ otherwise. Hence, \eqref{eqn:crypto_ackermann_shape} can  be  equivalently rewritten as
$\bigmetaand_i \bigvee^{\varepsilon(v_i)} M_i \leq^{\varepsilon(v_i)} v_i \Rightarrow (\phi'\leq \psi')[\overline{p}/!\overline{x}, \overline{q}/!\overline{y}]$. By Ackermann Lemma (cf.~Lemma \ref{lemma:ackermann}), the latter inequality can be  equivalently rewritten as $(\phi'\leq \psi')[\overline{p}/!\overline{x}, \overline{q}/!\overline{y},\overline{\bigvee^{\varepsilon(v)} M}/\overline v]$, 
which is an $(\Omega', \varepsilon')$-inductive $\langbase$-inequality, where $\varepsilon'(p_{i_j}) = \varepsilon^\partial(v_{i_j})$, $\varepsilon'(q_{k_j}) = \varepsilon(v_{k_j})$, and $\Omega'$ is defined as follows:  
$r <_{\Omega'} r'$ iff $l \prec l'$ with $l$ (resp.\ $l'$) being the leaf such that $\mathsf{MVTree}_\Omega^\varepsilon(l)$ (resp.\ $\mathsf{MVTree}_\Omega^\varepsilon(l')$) is extracted in \eqref{eqn:crypto_extraction} using $r$ (resp.\ $r'$). The  definition of $\varepsilon'$ given above ensures that the critical occurrences of each variable $p_{i_j}$ and $q_{k_j}$  are those occurring in  $\mathsf{LA}(\alpha_j)(p_{i_j})$ and $\mathsf{RA}(\beta_j)(q_{k_j})$ which are terms belonging to some minimal valuation set, and, therefore, they occur in the final inequality. Hence, item 1 of Definition \ref{Inducive:Ineq:Def} is satisfied, as each $\mathsf{LA}(\alpha_j)(p_{i_j})$ and $\mathsf{RA}(\beta_j)(q_{k_j})$ replaces a variable whose ancestors are only skeleton nodes, and in each $\mathsf{LA}(\alpha_j)(p_{i_j})$ and, $\mathsf{RA}(\beta_j)(q_{k_j})$, the path from the root to $p_{i_j}$ and $q_{k_j}$ contains only definite PIA nodes. The fact that the order $\prec$ is strict guarantees that  item 2 of Definition \ref{Inducive:Ineq:Def} is also satisfied.
\end{proof}

\begin{example}
\label{eg:vssaxiomk_to_inductive}
The inequality $\Box(\Diamondblack v_1 \backslash v_2) \leq v_1 \backslash \Box v_2$ is $(\Omega, \varepsilon)$-crypto inductive for the language $\langbase$ of the full Lambek calculus extended with a unary operator $\Box$ having left residual $\Diamondblack$, for $\varepsilon(v_1) = 1$ and $\varepsilon(v_2) = \partial$ and any $\Omega$. The $\mathsf{MVTree}_\Omega^\varepsilon$ of the occurrence of $p$ on the left side is $\Diamondblack p$, while that of the occurrence of $q$ on the left side is just $q$. Therefore, the procedure in Proposition \ref{prop:cryptotoinductive} yields
\smallskip

{{\centering
\begin{tabular}{rl}
& $\Box(\Diamondblack v_1 \backslash v_2) \leq v_1 \backslash \Box v_2$ \\
iff & $\Diamondblack v_1 \leq p \Rightarrow \Box(p \backslash v_2) \leq v_1 \backslash \Box v_2$\\
iff & $\Diamondblack v_1 \leq p \metaand q \leq v_2 \Rightarrow \Box(p \backslash q) \leq v_1 \backslash \Box v_2$ \\ 
iff & $v_1 \leq \Box p \metaand q \leq v_2 \Rightarrow \Box(p \backslash q) \leq v_1 \backslash \Box v_2$\\
iff & $\Box(p \backslash q) \leq \Box p \backslash \Box q$,
\end{tabular}
\par}}
\smallskip

\noindent i.e., the non-distributive version of axiom K for classical normal modal logic.
\end{example}

\begin{example}
\label{eg:crypto_frege}
The inequality $v_3 \leq v_2 \rightharpoonup \big[ v_1 \rightharpoonup \big( (v_1 \bullet v_2) \bullet (v_1 \bullet v_3) \big) \big]$ in the language of Example \ref{eg:alba_fregeLE} is $(\Omega,\varepsilon)$-crypto inductive for $\varepsilon(v_1)=\varepsilon(v_2)=\varepsilon(v_3)=1$, and $v_1 <_\Omega v_2 <_\Omega v_3$. The signed generation trees $+v_1$, $+(v_1 \bullet v_2)$, and $+((v_1 \bullet v_2) \bullet (v_1 \bullet v_3))$, are the trees $\mathsf{MVTree}_\Omega^\varepsilon(l)$ of the non critical occurrences of $v_1$, $v_2$, and $v_3$, respectively. The procedure in Proposition \ref{prop:cryptotoinductive} computes
\smallskip

{{\centering 
\begin{tabular}{rl}
& $v_3 \leq v_2 \rightharpoonup \big[ v_1 \rightharpoonup \big( (v_1 \bullet v_2) \bullet (v_1 \bullet v_3) \big) \big]$ \\
iff & $v_1 \leq p \Rightarrow v_3 \leq v_2 \rightharpoonup \big[ v_1 \rightharpoonup \big( (p \bullet v_2) \bullet (p \bullet v_3) \big) \big]$\\ 
iff & $v_1 \leq p \metaand p \bullet v_2 \leq q \Rightarrow v_3 \leq v_2 \rightharpoonup \big[ v_1 \rightharpoonup \big( q \bullet (p \bullet v_3) \big) \big]$ \\
iff & $v_1 \leq p \metaand p \bullet v_2 \leq q \metaand q \bullet (p \bullet v_3) \leq r \Rightarrow v_3 \leq v_2 \rightharpoonup (v_1 \rightharpoonup r)$ \\
iff & $v_1 \leq p \metaand v_2 \leq p \rightharpoonup q \metaand v_3 \leq p \rightharpoonup (q \rightharpoonup r) \Rightarrow v_3 \leq v_2 \rightharpoonup ( v_1 \rightharpoonup r )$ \\
iff & $ p \rightharpoonup (q \rightharpoonup r) \leq (p \rightharpoonup q) \rightharpoonup (p \rightharpoonup r)$, \\
\end{tabular}
\par}}
\smallskip

\noindent i.e., the non-distributive version of Frege's axiom for propositional logic.
\end{example}

\begin{lemma}
\label{lemma:inductivetocrypto}
Every $\langbase$-inductive inequality is equivalent to some crypto $\langbase$-inductive inequality.
\end{lemma}
\begin{proof}
By Corollary \ref{cor:inductive_to_vss}, any refined $(\Omega,\varepsilon)$-inductive inequality $(\varphi \leq \psi)[\eta_a/!a,\eta_b/!b][\alpha/!x,\beta/!y,\gamma/!z,\delta/!w]$ is equivalent to some very simple Sahlqvist $\langbase^\ast$-inequality. In particular, given the ALBA output 
\begin{equation}
\label{eq:cryptoalbaoutput}
\forall \overline\nomj_x, \overline \cnomm_y, \overline \nomj_z, \overline \cnomm_w 
\left(
\left(\bigmetaand_{i=1}^{n_z}\nomj_{z_i}\leq \gamma_i \metaand 
\bigmetaand_{i=1}^{n_w} \delta_i\leq \cnomm_{w_i}\right)\left[\overline{\bigvee \mathsf{Mv}(p)/p}, \overline{\bigwedge \mathsf{Mv}(q)/q}\right] \Rightarrow
\hfill
(\varphi \leq \psi)[\overline{\nomj_a/!a}, \overline{\cnomm_b/!b}]
\right),
\end{equation}
the proof of Theorem \ref{thm:invcorr} can be replicated by instantiating $\xi_i \coloneqq \gamma_i \left[\overline{\bigvee \mathsf{Mv}(p)/p}, \overline{\bigwedge \mathsf{Mv}(q)/q}\right]$ and $\zeta_i \coloneqq \delta_i \left[\overline{\bigvee \mathsf{Mv}(p)/p}, \overline{\bigwedge \mathsf{Mv}(q)/q}\right]$ for each $\gamma_i$ and $\delta_i$, and (still following the proof of Theorem \ref{thm:invcorr}), the initial inequality can be rewritten as
\begin{equation}
\label{eq:proving_crypto_things}
\forall \overline{p_{\nomj_{x_i}}}, \overline{q_{\cnomm_{y_i}}} \left(
\varphi \leq \psi
\right)
\left[\overline{
\xi_i'
/\nomj_{z_i}
},\overline{
\zeta_i' 
/\cnomm_{w_i}}
\right]
\left[
\overline{p_{\nomj_{x_i}}}/\overline\nomj_{x_i}, \overline{q_{\cnomm_{y_i}}}/\overline{\cnomm_{y_i}}
\right],
\end{equation}
which is a very simple Sahlqvist inequality for $\varepsilon'(p_{\nomj_{x_i}}) = 1$ and $\varepsilon'(q_{\cnomm_{y_i}}) = \partial$ (for any index $i$).
Of course, all the $\varepsilon'$-critical branches in \eqref{eq:proving_crypto_things} are branches in the $\langbase$-formulas $\varphi$ and $\psi$, and hence  contain only nodes from $\langbase$. 
Let us exhibit an order on the variables in $\overline{p_{\nomj_{x_i}}}$ and $\overline{q_{\cnomm_{y_i}}}$ that satisfies item 2 of Definition \ref{def:Crypto:inductive}. 

In the computation of the ALBA output \eqref{eq:cryptoalbaoutput}, each $\nomj_x$ and $\cnomm_y$ has been introduced in relation to some definite positive (resp.\ negative) PIA $\alpha$ (resp.\ $\beta$) during first approximation. Each such PIA formula has a critical variable occurrence $v$ on which Lemma \ref{lemma:ackermann} is applied after computing the left (resp.\ right) residual  of $\alpha$ (resp.\ $\beta$) with respect to $v$. For every $\nomj_x$ and every $\cnomm_y$, let $\tau(\nomj_x)$ and $\tau(\cnomm_y)$ be the critical variables of the PIA formulas to which they are related, and let $\rho(\nomj_x)$ and $\rho(\cnomm_y)$ be the variables $p_{\nomj_x}$ and $q_{\cnomm_y}$, respectively. Consider the  order $<_{\Omega'}$ on the variables of \eqref{eq:proving_crypto_things} defined as follows: for any $\pureu,\purev \in \{ \nomj_{x_i} : 1 \leq i \leq n_x \} \cup \{ \cnomm_{y_i} : 1 \leq i \leq n_y \}$,
\[
\rho(\pureu) <_{\Omega'} \rho(\purev) \quad\quad \mbox{iff} \quad\quad 
\tau(\pureu) <_\Omega \tau(\purev).
\]
By construction and by item 2 of Definition \ref{Inducive:Ineq:Def}, item 2 of Definition \ref{def:Crypto:inductive} is satisfied.
\end{proof}

\begin{example}
\label{eg:cryptogoranko}
Consider the LE language with $\mathcal{F} = \{\Box,\Diamond\}$ and $\mathcal{G} = \{ \circfor, \circback \}$.
For $\varepsilon(p_j) = 1$ and empty $\Omega$, the following inequality\footnote{We remind the reader that in this example $\circ,\Diamondblack \in \mathcal{F}^* \setminus \mathcal{F}$, and $\circfor, \circback \in \mathcal{G}$ are the residuals of $\circ$. }
\[
p_\nomj \leq \Diamond\Box\Box\Diamondblack(\Diamond\Box p_\nomj \circ \Diamondblack p_\nomj)
\]
from Example \ref{eg:goranko_inv_to_vss} is crypto $(\Omega,\varepsilon)$-inductive, as the two occurrences of $p_\nomj$ do not share the same $\mathsf{MVTree}_\Omega$ for any $\Omega$. The term encoded by the $\mathsf{MVTree}_\Omega^\varepsilon$ of the left non-critical occurrence of $p_\nomj$ is just $p_\nomj$, and the one of the right occurrence is $\Diamondblack(\Diamond\Box p_\nomj \circ \Diamondblack p_\nomj)$. Therefore, the procedure in Proposition \ref{prop:cryptotoinductive} yields
\smallskip

{{\centering 
\begin{tabular}{rl}
& $p_\nomj \leq \Diamond\Box\Box\Diamondblack(\Diamond\Box p_\nomj \circ \Diamondblack p_\nomj)$ \\
iff & $p_\nomj \leq q_1 \Rightarrow p_\nomj \leq \Diamond\Box\Box\Diamondblack(\Diamond\Box q_1 \circ \Diamondblack p_\nomj)$\\
iff & $p_\nomj \leq q_1 \metaand \Diamondblack(\Diamond\Box q_1 \circ \Diamondblack p_\nomj) \leq q_2 \Rightarrow p_\nomj \leq \Diamond\Box\Box q_2$ \\
iff & $p_\nomj \leq q_1 \metaand p_\nomj \leq \Box(\Diamond\Box q_1 \circback \Box q_2) \Rightarrow p_\nomj \leq \Diamond\Box\Box q_2$ \\
iff & $p_\nomj \leq q_1 \wedge \Box(\Diamond\Box q_1 \circback \Box q_2) \Rightarrow p_\nomj \leq \Diamond\Box\Box q_2$ \\
iff & $q_1 \wedge \Box(\Diamond\Box q_1 \circback \Box q_2) \leq \Diamond\Box\Box q_2$
\end{tabular}
\par}}
\end{example}
In the previous example, due to the fact that the parent node of the  leftmost noncritical occurrence of $p_\nomj$ is labelled by $\Box$ (a non-residuation conservative connective), its $\mathsf{MVTree}_\Omega^\varepsilon$ can never intersect the path from the root to the other non-critical occurrence in its $\mathsf{MVTree}_\Omega^\varepsilon$. Hence, the choice of the order on the leaves labelled by the same variable does not influence the definition of $\mathsf{MVTree}$. In the following example, we exhibit an inequality in which such choice matters, in the precise sense that the procedure in Proposition \ref{prop:cryptotoinductive} yields a different inductive inequality.
\begin{example}
Consider the following variation on the previous example:
\[p_\nomj \leq \Diamond\Box\Box\Diamondblack(\Diamondblack\Diamondblack p_\nomj \circ \Diamondblack p_\nomj),
\]
which is still crypto inductive for $\varepsilon(p_j) = 1$. Since by item 4 of Definition \ref{def:mvtree}, priority is given by ordering leaves (labelled by the same variable) from left to right, the $\mathsf{MVTree}_\Omega^\varepsilon$ of the first non-critical occurrence of $p_j$ is $\Diamondblack(\Diamondblack\Diamondblack p_\nomj \circ \Diamondblack p_\nomj)$, and the one of the second occurrence is just $\Diamondblack p_\nomj$. The procedure in Proposition \ref{prop:cryptotoinductive} yields
\smallskip

{{\centering 
\begin{tabular}{rl}
& $p_\nomj \leq \Diamond\Box\Box\Diamondblack(\Diamondblack\Diamondblack p_\nomj \circ \Diamondblack p_\nomj)$ \\
iff & $\Diamondblack p_\nomj \leq q_1 \Rightarrow p_\nomj \leq \Diamond\Box\Box\Diamondblack(\Diamondblack\Diamondblack p_\nomj \circ q_1)$\\
iff & $\Diamondblack p_\nomj \leq q_1 \metaand \Diamondblack(\Diamondblack\Diamondblack p_\nomj \circ q_1) \leq q_2 \Rightarrow p_\nomj \leq \Diamond\Box\Box q_2$ \\
iff & $p_\nomj \leq \Box q_1 \metaand p_\nomj \leq \Box\Box(\Box q_2 \circfor q_1) \Rightarrow p_\nomj \leq \Diamond\Box\Box q_2$ \\
iff & $p_\nomj \leq \Box q_1 \wedge \Box\Box(\Box q_2 \circfor q_1) \Rightarrow p_\nomj \leq \Diamond\Box\Box q_2$ \\
iff & $\Box q_1 \wedge \Box\Box(\Box q_2 \circfor q_1) \leq \Diamond\Box\Box q_2$
\end{tabular}
\par}}
\smallskip

\noindent As mentioned above, it is possible to use any order on the leaves labelled by the same variable. Any choice will lead to different outcomes. Let us check what happens when priority is given to rightmost non-critical leaf. The $\mathsf{MVTree}_\Omega^\varepsilon$ of the first and second non-critical occurrence of $p_j$ become $\Diamondblack\Diamondblack p_j$ and  $\Diamondblack(\Diamondblack\Diamondblack p_\nomj \circ \Diamondblack p_\nomj)$, respectively. The procedure in Proposition \ref{prop:cryptotoinductive} yields
\smallskip

{{\centering 
\begin{tabular}{rl}
& $p_\nomj \leq \Diamond\Box\Box\Diamondblack(\Diamondblack\Diamondblack p_\nomj \circ \Diamondblack p_\nomj)$ \\
iff & $\Diamondblack\Diamondblack p_\nomj \leq q_1 \Rightarrow p_\nomj \leq \Diamond\Box\Box\Diamondblack(q_1 \circ \Diamondblack p_\nomj)$\\
iff & $\Diamondblack\Diamondblack p_\nomj \leq q_1 \metaand \Diamondblack(q_1 \circ \Diamondblack p_\nomj) \leq q_2 \Rightarrow p_\nomj \leq \Diamond\Box\Box q_2$ \\
iff & $p_\nomj \leq \Box\Box q_1 \metaand p_\nomj \leq \Box(q_1 \circback \Box q_2) \Rightarrow p_\nomj \leq \Diamond\Box\Box q_2$ \\
iff & $p_\nomj \leq \Box\Box q_1 \wedge \Box(q_1 \circback \Box q_2) \Rightarrow p_\nomj \leq \Diamond\Box\Box q_2$ \\
iff & $\Box\Box q_1 \wedge \Box(q_1 \circback \Box q_2) \leq \Diamond\Box\Box q_2$
\end{tabular}
\par}}
\end{example}

Having shown that crypto $\langbase$-inductive inequalities can be equivalently rewritten as inductive $\langbase$ inequalities, Definition \ref{def:inductiveinvcorr} can be restricted so as to target such inequalities:
\begin{definition} 
\label{def:crypto_inductive_inverse_corr}
An {\em $\langbase$-\cryptoinductiveinversecorrespondent} is a $\langbase$-{\inductiveinversecorrespondent}
\[
\forall \overline \nomj, \overline\cnomm, \overline\nomi, \overline\cnomn \left( 
\bigmetaor_{\nomj_i \mbox{ in } \overline\nomj}
\theta_{\nomj_i}^-(\nomj_i)
\metaor
\bigmetaor_{\cnomm_i \mbox{ in } \overline\cnomm}
\theta_{\cnomm_i}^-(\cnomm_i)
\metaor
\varphi \leq \psi
\right),
\]
whose equivalent very simple Sahlqvist $\langbase^\ast$-inequality obtained using the procedure in Theorem \ref{thm:invcorr} is crypto $\langbase$-inductive.

We consider to be inverse correspondents also formulas with the (trivially equivalent) shape
\[
\forall \overline \nomj, \overline\cnomm, \overline\nomi, \overline\cnomn \left( 
\bigmetaand_{\nomj_i \mbox{ in } \overline\nomj}
\metanot\theta_{\nomj_i}^-(\nomj_i)
\metaor
\bigmetaand_{\cnomm_i \mbox{ in } \overline\cnomm}
\metanot\theta_{\cnomm_i}^-(\cnomm_i)
\Rightarrow
\varphi \leq \psi
\right).
\]
\end{definition}

A straightforward consequence of the discussion above follows from Theorem \ref{thm:invcorr}: 

\begin{theorem}
\label{thm:cryptoinvcorr}
Any $\langbase$-{\cryptoinductiveinversecorrespondent} is equivalent to some $\langbase$-crypto inductive inequality. Hence, by Proposition \ref{prop:cryptotoinductive}, every $\langbase$-{\cryptoinductiveinversecorrespondent} is equivalent to some $\langbase$-inductive inequality.
\end{theorem}

The reader is referred to Section \ref{sec:moreexamples} for more examples.
\section{Instantiation to specific semantic settings}
\label{sec:instances}

In the present section we discuss how the argument in Section \ref{sec:inverse} can be instantiated to a few semantics for LE-logics. More specifically, we describe the shape in Definition \ref{def:inductiveinvcorr} in the frame correspondence languages of such semantics. Examples of (inverse) correspondence in these settings can be then found in Section \ref{sec:moreexamples}. Other semantics for LE-logics which are not taken into account for space reasons do exists (e.g., see graph-based frames in \cite{conradie2020non}).

\subsection{Kripke semantics}
\label{ssec:kripke}

In the usual Kripke semantics, formulas are interpreted using structures $(X, (R_h)_{h \in \mathcal{F} \cup \mathcal{G}})$ to each $n$-ary connective $h$ in a fixed LE-language $\langbase$ is associated a relation $R_h \subseteq X \times X^{n}$. Valuations map variables into the powerset of $X$, $\bot$ to $\varnothing$, $\top$ to $X$, and are extended to formulas as follows:
\smallskip

{{\centering
\begin{tabular}{rclcrcl}
$V(\varphi \wedge \psi)$ & $=$ & $V(\varphi)\cap V(\psi)$ &
\quad\quad &
$V(g(\varphi_1,\ldots,\varphi_n))$ & $=$ & $(R_g^{-1}[V(\varphi_1)^{\varepsilon_g^\partial(1)} \times \cdots \times V(\varphi_1)^{\varepsilon_g^\partial(n)}])^{c}$
\\
$V(\varphi \vee \psi)$ & $=$& $V(\varphi)\cup V(\psi)$ &
\quad\quad &
$V(f(\varphi_1,\ldots,\varphi_n))$ & $=$ & $R_f^{-1}[V(\varphi_1)^{\varepsilon_f(1)} \times \cdots \times V(\varphi_1)^{\varepsilon_f(n)}]$,
\end{tabular}
\par}}
\smallskip

\noindent for every $f \in \mathcal{F}, g\in\mathcal{G}$, and where for a set $S \subseteq X$, $S^1 := S$ and $S^\partial := S^{c}:= X \setminus S$.

As discussed in Remark \ref{remark:flatinequalities}, in the classical Kripke setting, both nominals and conominals can be  translated as individual variables ranging in the domain of  Kripke frames. As already mentioned (cf.~Remark \ref{remark:negatedskeleton}), an inequality $\nomj \nleq \cnomm$ translates as $x = y$ if $\nomj$ is translated as $x$ and $\cnomm$ as $y$. To understand what the other inequalities in $\langmeta$ become in the standard first order correspondence language, let us see what happens with the simplest forms of inequalities that may occur. For any assignment that maps $\nomj$ to $\{x\}$, $\nomi$ to $\{y\}$, $\cnomm$ to $\{y\}^{c}$, and $\cnomn$ to $\{x\}^{c}$,
\smallskip

{{\centering 
\begin{tabular}{rl}
&$\nomj \leq \Box\cnomm$ \\
iff & $x \in R_\Box^{-1}[(\{y\}^{c})^{c}]^{c}$ \\
iff & $x \notin R_\Box^{-1}[y]$ \\
iff & $\metanot R_\Box xy$.
\end{tabular}
\begin{tabular}{rl}
&$\nomj \leq \rhd\nomi$ \\
iff & $x \in R_\rhd^{-1}[y]^{c}$ \\
iff & $x \notin R_\rhd^{-1}[y]$ \\
iff & $\metanot R_\rhd xy$.
\end{tabular}
\begin{tabular}{rl}
&$\Diamond \nomj \leq \cnomm$ \\
iff & $R_\Diamond^{-1}[x] \subseteq \{y\}^{c}$ \\
iff & $y \notin R_\Diamond^{-1}[x]$ \\
iff & $\metanot R_\Diamond yx$.
\end{tabular}
\begin{tabular}{rl}
&$\lhd \cnomn \leq \cnomm$ \\
iff & $R_\lhd^{-1}[(\{x\}^{c})^{c}] \subseteq \{y\}^{c}$ \\
iff & $y \notin R_\lhd^{-1}[x]$ \\
iff & $\metanot R_\lhd yx$.
\end{tabular}
\par}}
\smallskip

\noindent Such inequalities can be straightforwardly generalized to operators with arbitrary arity and tonicity in each coordinate. Let us consider any assignment that maps $\nomj$ to $\{x\}$, $\cnomm$ to $\{y\}^{c}$, and each pure variable $\pureu_i$ in the vector $\overline\pureu$ to $\{z_i\}$ if $\pureu_i$ is a nominal and to $\{z_i\}^{c}$ otherwise. For any $f \in \mathcal{F}$ (resp.\ $g\in\mathcal{G}$) such that $\pureu_i$ is a nominal (resp.\ conominal) for every positive coordinate of $f$ (resp.\ negative coordinate of $g$), and a conominal for every negative one (resp.\ positive one),
\smallskip 

{{\centering
\begin{tabular}{rl}
    & $\nomj \leq g(\overline\pureu)$ \\
iff & $x \in R_g^{-1}[(z_1,\ldots,z_n)]^{c}$ \\
iff & $x \notin R_g^{-1}[(z_1,\ldots,z_n)]$ \\
iff & $\metanot R_g (x,z_1,\ldots,z_n)$
\end{tabular}
\begin{tabular}{rl}
    & $f(\overline\pureu) \leq \cnomm$ \\
iff & $R_f^{-1}[(z_1,\ldots,z_n)] \subseteq \{y\}^{c}$ \\
iff & $y \notin R_f^{-1}[(z_1,\ldots,z_n)]$ \\
iff & $\metanot R_f (y,z_1,\ldots,z_n)$.
\end{tabular}
\par}}
\smallskip

Inequalities $\varphi \leq \cnomm$ (resp.\ $\nomj \leq \psi$) involving arbitrary positive (resp.\ negative) skeleton formulas $\varphi$ (resp.\ $\psi$) can be interpreted in two ways: either the term function of $\varphi$ (resp.\ $\vee$) is thought of as an $f$-operator (resp.\ $g$-operator) as above, or the inequality is further flattified to a formula containing only inequalities as above. This further flattification process is not possible in general lattice expansions (see Appendix \ref{sssec:flatteningskeleton}), but in distributive settings (such as the one given by Kripke semantics) it can be carried out \cite{inversedle}.
Indeed, for any positive skeleton $\varphi$ (resp.\ negative skeleton $\psi$), a negated inequality $\nomj \nleq \psi$ (resp. $\varphi \nleq \cnomm$) is equivalent to $\psi \leq \neg\nomj$ (resp.\ $\neg\cnomm \leq \varphi$); hence any $\langmeta$-atom can be rewritten as $\psi \leq \cnomm'$ or $\nomj' \leq \varphi$, by identifying $\cnomm'$ with $\neg\nomj$ and $\nomj'$ with $\neg\cnomm$.
Following \cite{inversedle}, if $\overline\varphi$ and $\overline \psi$ are vectors of positive and negative skeleton formulas respectively, then
 any inequality $\nomj \leq f(\overline \varphi, \overline\psi)$ (resp.\ $g(\overline \psi, \overline \varphi) \leq \cnomm$), assuming w.l.o.g.~that $f \in \mathcal{F}$ (resp.~$g \in \mathcal{G}$) be positive (resp.\ negative) in $\overline \varphi$ and negative (resp.\ positive) in $\overline \psi$, is equivalent to
\footnote{Note that this equivalent rewriting works only in distributive settings because it uses the fact that completely join/meet irreducible elements are completely join/meet prime. }
\smallskip

{{\centering
\begin{tabular}{rl}
& $\exists \overline\nomi \exists \overline\cnomn ( \nomj \leq f(\overline\nomi, \overline\cnomn) \metaand \bigmetaand_h \nomi_h \leq \varphi_h \metaand \bigmetaand_k \psi_k \leq \cnomn_k ) $ \\
resp.\ & $\exists \overline\cnomn \exists \overline\nomi (g(\overline\cnomn, \overline\nomi) \leq \cnomm \metaand \bigmetaand_h \nomi_h \leq \varphi_h \metaand \bigmetaand_k \psi_k \leq \cnomn_k ) $,
\end{tabular}
\par}}
\smallskip

\noindent and it is possible to proceed recursively on the inequalities $\nomi_h \leq \varphi_h$ and $\psi_k \leq \cnomn_k$ until all the inequalities above contain at most one connective. 

Hence, compared to the general lattice setting, the distributive version of the inverse correspondence algorithm allows for a finer-grained  manipulation of the base case of the inverse-disjunct definition, which also allows to work only with $\langmeta$-atoms containing at most one connective, which better reflect the classical and distributive usual frame correspondence languages.

The definition of positive (resp.\ negative) inverse disjunct is specialized as follows:
\begin{description}
\item[Base case:] $\theta^+(x) := \metanot R_\varphi x\overline y$ (resp.\ $\theta^-(x) := R_\varphi x\overline y$) for a positive skeleton $\varphi$, or $\theta^+(x) := \metanot R_\psi x\overline y$ (resp.\ $\theta^-(x) := R_\psi x\overline y$) for a negative skeleton $\psi$. As a consequence of the discussion above, $\theta^+(x)$ (resp.\ $\theta^-(x)$) is equivalent to a universally quantified disjunction (resp.\ existentially quantified conjunction) of relational atoms with certain properties on the position in which the quantified variables occur in the relational atoms (we refer the reader to \cite[Section 4]{inversedle} and \cite[Section 3.7]{blackburn2002modal}).
\item[Conjunction/Disjunction:] $\theta^+(x) \coloneqq \bigmetaand_i \theta_i^+(x)$ (resp.\ $\theta^-(x) \coloneqq \bigmetaor_i \theta_i^-(x)$); 
\item[Quantification:] $\theta^+(x) \coloneqq\forall\overline y[R_\psi x\overline y \Rightarrow \bigmetaor_i \theta_i^-(y_i)]$ (resp.\ $\theta^-(x) \coloneqq\exists \overline y[\bigmetaand_i \theta_i^+(y_i) \metaand R_\psi x\overline y]$), where $\psi$ is a definite skeleton formula, and each $\theta_i^-$ (resp.\ $\theta^+_i$) is a negative (resp.\ positive) inverse disjunct where $x$ does not occur. The formula $\psi$ can be just a variable.
\end{description}
\subsection{Polarity based semantics}
\label{ssec:polarity_based_semantics}

Polarity based semantics were introduced in \cite{conradie2020non}. In the present section we recall the basic definitions, and we instantiate the whole framework to this specific setting. The preliminaries in this Section are largerly drawn from \cite{dairapaper}.

For all sets $A, B$ and any relation $S \subseteq A \times B$, we let, for any $A' \subseteq A$ and $B' \subseteq B$,
$$S^{(1)}[A'] := \{b \in B\mid  \forall a(a \in A' \Rightarrow a Sb ) \} \quad \mathrm{and}\quad S^{(0)}[B'] := \{a \in A \mid \forall b(b \in B' \Rightarrow a S b)  \}.$$
For all sets $A, B_1,\ldots B_n,$ and any relation $S \subseteq A \times B_1\times \cdots\times B_n$, for any $\overline{C}: = (C_1,\ldots, C_n)$ where $C_j\subseteq B_j$ for all $1 \leq j \leq n$, we let, for any $A'\subseteq A$ and $1\leq i\leq n$,
\begin{equation*}\label{eq:notation bari}
\overline{C}^{\,i}:  = (C_1,\ldots,C_{i-1}, C_{i+1},\ldots, C_n)
\quad \mbox{ and } \quad
\overline{C}^{\,i}_{A'}: = (C_1\ldots,C_{i-1}, A', C_{i+1},\ldots, C_n).
\end{equation*}
When $B_i\supseteq C_i: = \{c_i\}$ and $A\supseteq A': =\{a'\} $, we write $\overline{c}$ for $\overline{\{c\}}$,  and $\overline{c}^{\,i}$ for $\overline{\{c\}}^{\,i}$, and $\overline{c}^{\,i}_{a'}$ for $\overline{\{c\}}^{\,i}_{\{a'\}}$.
We also let:
$S_i \subseteq B_i \times B_1 \times \cdots \times B_{i-1} \times A \times B_{i + 1} \times \cdots\times B_n$ be defined by
$(b_i, \overline{c}_{a}^{\, i})\in S_i \ \mbox{ iff }\ (a,\overline{c})\in S$; $S^{(0)}[\overline{C}] := \{a \in A\mid  \forall \overline{b}(\overline{b}\in \overline{C} \Rightarrow aS \overline{b} ) \}$, and $S^{(i)}[A', \overline{C}^{\,i}] := S_i^{(0)}[\overline{C}^{\, i}_{A'}]$.

A {\em formal context} or {\em polarity} \cite{ganter2012formal} is a structure $\mathbb{P} = (A, X, I)$ s.t.~$A$ and $X$ are sets and $I\subseteq A\times X$ is a binary relation. 
For every polarity $\mathbb{P}$, maps $(\cdot)^\uparrow: \mathcal{P}(A)\to \mathcal{P}(X)$ and $(\cdot)^\downarrow: \mathcal{P}(X)\to \mathcal{P}(A)$ can be defined as follows:
$B^\uparrow: = I^{(1)}[B] =  \{x\in X \mid (\forall a \in B) aIx\}$ and $Y^\downarrow: = I^{(0)}[Y] = \{a\in A\mid (\forall x \in Y) aIx\}$. The maps $(\cdot)^\uparrow$ and $(\cdot)^\downarrow$ form a \textit{Galois connection} between $(\mathcal{P}(A), \subseteq)$ and $(\mathcal{P}(X), \subseteq)$, i.e. $Y \subseteq B^\uparrow$ iff $B\subseteq Y^\downarrow$
for all $B \in \mathcal{P}(A)$ and $Y\in \mathcal{P}(X)$. 
A {\em formal concept} of $\mathbb{P}$ is a pair 
$c = (\val{c}, \descr{c})$ s.t.~$\val{c}\subseteq A$ and $\descr{c}\subseteq X$, and  $\val{c}^{\uparrow} = \descr{c}$ and $\descr{c}^{\downarrow} = \val{c}$. The set $\val{c}$ is the {\em extension} of $c$, while $\descr{c}$ is its {\em intension}. It  immediately follows from this definition that if $(\val{c}, \descr{c})$ is a formal concept, then $\val{c}^{\uparrow\downarrow} = \val{c}$ and $\descr{c}^{\downarrow\uparrow} = \descr{c}$. That is, $\val{c}$ and $\descr{c}$ are  \textit{Galois-stable}.  The set $\mathbb{L}(\mathbb{P})$  of the formal concepts of $\mathbb{P}$ can be partially ordered as follows: for any $c, d\in \mathbb{L}(\mathbb{P})$, \[c\leq d\quad \mbox{ iff }\quad \val{c}\subseteq \val{d} \quad \mbox{ iff }\quad \descr{d}\subseteq \descr{c}.\]
With this order, $\mathbb{L}(\mathbb{P})$ forms a complete lattice.

A {\em polarity-based $\mathcal{L}_{\text{LE}}$-frame} (see \cite{dairapaper}) is a tuple $\mathbb{F} = (\mathbb{P}, \mathcal{R}_{\mathcal{F}}, \mathcal{R}_{\mathcal{G}})$, where  $\mathbb{P} = (A, X,  I)$ is a polarity, $\mathcal{R}_{\mathcal{F}} = \{R_f\mid f\in \mathcal{F}\}$, and $\mathcal{R}_{\mathcal{G}} = \{R_g\mid g\in \mathcal{G}\}$, such that  for each $f\in \mathcal{F}$ and $g\in \mathcal{G}$, the symbols $R_f$ and  $R_g$ respectively denote $(n_f+1)$-ary and $(n_g+1)$-ary relations 
$R_f \subseteq X \times A^{\varepsilon_{f}}$ and $R_g \subseteq A \times X^{\varepsilon_{g}}$,
where $A^{\varepsilon_{f}}$ denotes the $n_f$-fold cartesian product of $A$ and $X$ such that, for each $1\leq i\leq n_f$, the $i$th projection of $A^{\varepsilon_{f}}$ is $A$ if $\varepsilon_f(i ) = 1$ and is $X$ if $\varepsilon_f(i ) = \partial$, and $X^{\varepsilon_g}$ denotes the $n_g$-fold cartesian product of $A$ and $X$ such that for each $1\leq i\leq n_g$ the $i$th projection of $X^{\varepsilon_{g}}$ is $X$ if $\varepsilon_g(i ) = 1$ and is $A$ if $\varepsilon_g(i ) = \partial$.
In addition, all relations $R_f$ and $R_g$ are required to be {\em $I$-compatible}, i.e.\ the following sets are assumed to be  Galois-stable for all $a \in A$, $x \in X $, $\overline{a} \in A^{\varepsilon_f}$, and $\overline{x} \in X^{\varepsilon_g}$:
\[
R_f^{(0)}[\overline{a}]\text{ and }R_f^{(i)}[x, \overline{a}^{\, i}] \quad\quad R_g^{(0)}[\overline{x}]\text{ and }R_g^{(i)}[a, \overline{x}^{\,i}].
\]

For any polarity-based frame $\mathbb{F}=(\mathbb{P}, \mathcal{R}_\mathcal{F}, \mathcal{R}_\mathcal{G})$ with $\mathbb{P} = (A, X, I)$, a {\em valuation} on $\mathbb{F}$ is a map $V:\atprop\to \mathbb{P}^+$. For every  $p\in \atprop$, we let  $\val{p}: = \val{V(p)}$ (resp.~$\descr{p}: = \descr{V(p)}$) denote the extension (resp.~the intension) of the interpretation of $p$ under $V$.  The elements (objects) of $\val{p}$ are the {\em members} of concept $p$ under  $V$; the elements (features) of $\descr{p}$ {\em describe}  concept $p$ under $V$. Any valuation $V$ on $\mathbb{F}$ extends homomorphically to a unique interpretation map of $\mathcal{L}$-formulas, which, abusing notation, we also denote $V$, and which is defined as follows:\footnote{\label{footn:abbreviations for val and descr of formulas}In what follows, we will drop reference to the valuation $V$ when this does not generate ambiguous readings, and write $\val{\varphi}$ for $\val{V(\phi)}$ and  $\descr{\varphi}$ for $\descr{V(\phi)}$.}.
\smallskip

{{\centering
\begin{tabular}{rcl c rcl}
$V(p)$ & $ = $ & $(\val{p}, \descr{p})$\\
$V(\top)$ & $ = $ & $(A, A^{\uparrow})$ & \quad\: &
 $V(\bot)$ & $ = $ & $(X^{\downarrow}, X)$\\
$V(\phi\wedge\psi)$ & $ = $ & $(\val{\phi}\cap \val{\psi}, (\val{\phi}\cap \val{\psi})^{\uparrow})$ &&
$V(\phi\vee\psi)$ & $ = $ & $((\descr{\phi}\cap \descr{\psi})^{\downarrow}, \descr{\phi}\cap \descr{\psi})$\\
$V(f(\overline{\phi}))$ & $ = $ & $\left(\left(R_f^{(0)}[\overline{\val{\phi}}^{\varepsilon_f}]\right)^{\downarrow}, R_f^{(0)}[\overline{\val{\phi}}^{\varepsilon_f}]\right)$ &&
$V(g(\overline{\phi}))$ & $ = $ & $\left(R_g^{(0)}[\overline{\descr{\phi}}^{\varepsilon_g}], \left(R_g^{(0)}[\overline{\descr{\phi}}^{\varepsilon_g}]\right)^{\uparrow}\right)$
\end{tabular}
\par}}

Here, for every $\overline{\phi}\in \mathcal{L}^{n_f}$  (resp.~$\overline{\phi}\in \mathcal{L}^{n_g}$), the tuple $\overline{\val{\phi}}^{\varepsilon_f}$ is such that for
each $1 \leq i \leq n_f$,  the $i$-th coordinate of $\overline{\val{\phi}}^{\varepsilon_f}$ is $\val{\phi_i}$ if $\varepsilon_f(i) = 1$, and is $\descr{\phi_i}$ if $\varepsilon_f(i) = \partial$, and $\overline{\descr{\phi}}^{\varepsilon_g}$ is such that for
each $1 \leq i \leq n_g$  the $i$-th coordinate of $\overline{\descr{\phi}}^{\varepsilon_g}$ is $\descr{\phi_i}$ if $\varepsilon_g(i) = 1$, and is  $\val{\phi_i}$ if $\varepsilon_g(i) = \partial$.

The two-sorted first order language  of polarity-based frames contains individual variables of two sorts, ranging over the objects and the attributes of the frame, and an $n$-ary relation symbol for every $n$-ary relation in the frame ($I$ included).
To see how $\langmeta$-atoms translate in the first order language of polarity based frames, let us start with the basic cases of skeleton formulas containing just one unary connective. For any polarity based frame $\mathbb{F}$ as above 
any unary $\Box, {\rhd}\in \mathcal{G}$, any $\Diamond, \lhd \in \mathcal{F}$,
and any assignment $V$ mapping (without loss of generality) $\nomj$ to the concept $(a^{\uparrow\downarrow}, a^\uparrow)$, $\nomi$ to $(b^{\uparrow\downarrow}, b^\uparrow)$, $\cnomm$ to $(x^\downarrow, x^{\downarrow\uparrow})$, and $\cnomn$ to $(y^\downarrow, y^{\downarrow\uparrow})$, the following equivalences hold:
\smallskip

{{\centering
\begin{tabular}{rl}
& $V(\nomj) \leq V(\Box\cnomm)$  \\
iff & $\val{\nomj}_V \subseteq \val{\Box\cnomm}_V$ \\
iff & $a^{\uparrow\downarrow} \subseteq R_\Box^{(0)}[x]$ \\
iff & $a \in R_\Box^{(0)}[x]$ \\
iff & $aR_\Box x$
\end{tabular}
\begin{tabular}{rl}
& $V(\nomj) \leq V({\rhd}\nomi)$  \\
iff & $\val{\nomj}_V \subseteq \val{{\rhd}\nomi}_V$ \\
iff & $a^{\uparrow\downarrow} \subseteq R_\rhd^{(0)}[b]$ \\
iff & $a \in R_\rhd^{(0)}[b]$ \\
iff & $aR_\rhd b$
\end{tabular}
\begin{tabular}{rl}
& $V(\Diamond\nomj) \leq V(\cnomm)$  \\
iff & $\descr{\cnomm}_V \subseteq \descr{\Diamond\nomj}_V$ \\
iff & $x^{\downarrow\uparrow} \subseteq R_\Diamond^{(0)}[a]$ \\
iff & $x \in R_\Diamond^{(0)}[a]$ \\
iff & $xR_\Diamond a$
\end{tabular}
\begin{tabular}{rl}
& $V({\lhd}\cnomn) \leq V(\cnomm)$  \\
iff & $\descr{\cnomm}_V \subseteq \descr{\lhd\cnomn}_V$ \\
iff & $x^{\downarrow\uparrow} \subseteq R_\lhd^{(0)}[y]^{\uparrow\downarrow}$ \\
iff & $x \in R_\lhd^{(0)}[y]$ \\
iff & $xR_\Diamond y$,
\end{tabular}
\begin{tabular}{rl}
& $V(\nomj) \leq V(\cnomm)$ \\ 
iff & $\val{\nomj}_V \subseteq \val{\cnomm}_V$ \\
iff & $a^{\uparrow\downarrow} \subseteq x^\downarrow$ \\
iff & $a \in x^\downarrow$ \\
iff & $a I x $
\end{tabular}
\par}}
\smallskip

\noindent where the third equivalence in each case follows from the fact that every relation ($I$ included) is $I$-compatible.

The equivalent rewritings above can be easily generalized to operators of any arity. For instance, for any operator $f$ in $\mathcal{F}$ (resp.\ $g \in \mathcal{G}$) whose first (resp.\ last) $n$ coordinates are positive, and whose last (resp.\ first) $m$ coordinates are negative, for any valuation mapping each $\nomi_i$ (for $1 \leq i \leq n$) to the concept generated by the object $a_i$, each $\cnomn_j$ (for $1 \leq j \leq m$) to the concept generated by the feature $x_i$, and $\cnomm$ (resp.\ $\nomj$) to the concept generated by the feature $y$ (resp.\ object $b$), it follows that
\smallskip

{{\centering
\begin{tabular}{rl}
& $V(f(\nomi_1, \ldots, \nomi_n, \cnomn_1, \ldots, \cnomn_m)) \leq V(\cnomm)$ \\
iff & $\descr{\cnomm}_V \subseteq \descr{f(\nomi_1, \ldots, \nomi_n, \cnomn_1, \ldots, \cnomn_m)}_V$ \\ 
iff & $y^{\downarrow\uparrow} \subseteq R^{(0)}_f[(a_1,\ldots,a_n,x_1,\ldots,x_m)]$ \\
iff & $y \in R^{(0)}_f[(a_1,\ldots,a_n,x_1,\ldots,x_m)]$ \\
iff & $R_f(y,a_1,\ldots,a_n,y_1,\ldots,y_m)$.
\end{tabular}
\begin{tabular}{rl}
& $V(\nomj) \leq V(g(\cnomn_1, \ldots, \cnomn_m,\nomi_1, \ldots, \nomi_n))$ \\
iff & $\val{\nomj}_V \subseteq \val{g(\cnomn_1, \ldots, \cnomn_m,\nomi_1, \ldots, \nomi_n)}_V$ \\ 
iff & $b^{\uparrow\downarrow} \subseteq R^{(0)}_g[(x_1,\ldots,x_m,a_1,\ldots,a_n)]$ \\
iff & $b \in R^{(0)}_g[(x_1,\ldots,x_m,a_1,\ldots,a_n)]$ \\
iff & $R_g(y,x_1,\ldots,x_m,a_1,\ldots,a_n)$.
\end{tabular}
\par}}
\smallskip

\noindent The same argument can easily be adapted to the interpretation of inequalities $\varphi(\nomi_1, \ldots, \nomi_n, \cnomn_1, \ldots, \cnomn_m) \leq \cnomm$ (resp.\ $\nomj \leq \psi(\cnomn_1, \ldots, \cnomn_m,\nomi_1, \ldots, \nomi_n)$), where $\varphi$ (resp.\ $\psi$) is a positive (resp.\ negative) skeleton formula which is, without loss of generality, positive in the first $n$ (resp.\ $m$) coordinates, and negative (resp.\ positive) in the last $m$ (resp.\ $n$) coordinates.
For notational convenience, let us associate a relation $R_\varphi \subseteq X \times A^n \times X^m$ (resp.\ $R_\psi \subseteq A \times X^m \times A^n$) to $\varphi$ (resp.\ $\psi$) by induction on $\varphi$ (resp.\ $\psi$) as follows.
\begin{definition}[Relation associated to skeleton formula.]
\label{def:rel_associated_with_skeleton_polarities}
Let $\varphi$ and $\psi$ be a positive and a negative skeleton formula, respectively. The relations $R_\varphi$ and $R_\psi$ are defined by simultaneous induction as follows:
\begin{description}
\item[Base case] if $\varphi$ (resp.\ $\psi$) is a variable, $R_\varphi \coloneqq J$ (resp.\ $R_\psi \coloneqq I$);
\item[Inductive case] if $\varphi = f(\varphi_1,\ldots,\varphi_h,\psi_1,\ldots,\psi_k)$ (resp.\ $\psi = g(\psi_1,\ldots,\psi_k,\varphi_1,\ldots,\varphi_h)$) such that $f \in \mathcal{F}$ (resp.\ $g\in\mathcal{G}$) is w.l.o.g.\ positive in its first $h$ (resp.\ $k$) coordinates and negative in the last $k$ (resp.\ $h$), each $\varphi_i$ is a positive skeleton formula, and each $\psi_i$ is a negative skeleton formula, then $R_\varphi \coloneqq R_f \:;_I\: (R_{\varphi_1}, \ldots, R_{\varphi_h}, R_{\psi_1}, \ldots, R_{\psi_k})$ (resp.\ $R_\psi \coloneqq R_g \:;_I\: (R_{\psi_1}, \ldots, R_{\psi_k}, R_{\varphi_1}, \ldots, R_{\varphi_h})$);
\end{description}
where the composition $;_I$ is defined as in Definition \ref{def:generalized_icomp} below.
\end{definition}

\begin{definition}[Generalized $I$-composition]
\label{def:generalized_icomp}
Let $D\subseteq X \times A^h \times X^k$ (resp.\ $B\subseteq A \times X^k \times A^h$), and let $D_i \subseteq X \times A^{h_i} \times X^{k_i}$ and $B_j \subseteq A \times X^{k_j} \times A^{h_j}$ for $1 \leq i \leq h$ and $h+1 \leq j \leq h + k$. The relation $D ;_I (D_1,\ldots,D_h,B_1,\ldots,B_k)$ is such that for all $\overline a_i \in A^{h_i}$, $\overline x_i \in X^{k_i}$, $\overline b_j \in A^{k_j}$, $\overline y_j \in X^{h_j}$, $a\in A$, $x \in X$,
\[
\begin{array}{ll}
(x,\overline a_1,\overline x_1,\ldots, \overline a_h, \overline x_h, \overline b_1, \overline y_1,\ldots, \overline b_k, \overline y_k) \in D ;_I (D_1,\ldots,D_h,B_1,\ldots,B_k)
& \mbox {if and only if}\\
x \in D^{(0)}[ (D_1^{(0)}[\overline a_1, \overline x_1]^\downarrow, \ldots, D_h^{(0)}[\overline a_h, \overline x_h]^\downarrow, B_1^{(0)}[\overline b_1, \overline y_1]^\uparrow, \ldots, B_k^{(0)}[\overline b_k, \overline y_k]^\uparrow) ],
\end{array}
\]
and, respectively,
\[
\begin{array}{ll}
(a,\overline b_1, \overline y_1,\ldots, \overline b_k, \overline y_k,\overline a_1,\overline x_1,\ldots, \overline a_h, \overline x_h) \in B ;_I (B_1,\ldots,B_k,D_1,\ldots,D_h)
& \mbox {if and only if}\\
a \in B^{(0)}[ (B_1^{(0)}[\overline b_1, \overline y_1]^\uparrow, \ldots, B_k^{(0)}[\overline b_k, \overline y_k]^\uparrow,D_1^{(0)}[\overline a_1, \overline x_1]^\downarrow, \ldots, D_h^{(0)}[\overline a_h, \overline x_h]^\downarrow) ].
\end{array}
\]
\end{definition}

We refer the reader to \cite{roughconcepts,dairapaper,conradie2020non} for discussions and motivations about the definition above which straightforwardly generalizes definitions found in these papers.

Thanks to the definitions of $R_\varphi$ and $R_\psi$, the interpretation of the inequalities $\varphi(\nomi_1, \ldots, \nomi_n, \cnomn_1, \ldots, \cnomn_m) \leq \cnomm$ (resp.\ $\nomj \leq \psi(\cnomn_1, \ldots, \cnomn_m,\nomi_1, \ldots, \nomi_n)$) where $\varphi$ is a positive (resp.\ negative) skeleton formula, can be given as follows
\smallskip

{{\centering
\begin{tabular}{rl}
& $V(\varphi(\nomi_1, \ldots, \nomi_n, \cnomn_1, \ldots, \cnomn_m)) \leq V(\cnomm)$ \\
iff & $\descr{\cnomm}_V \subseteq \descr{\varphi(\nomi_1, \ldots, \nomi_n, \cnomn_1, \ldots, \cnomn_m)}_V$ \\ 
iff & $y^{\downarrow\uparrow} \subseteq R^{(0)}_\varphi[(a_1,\ldots,a_n,x_1,\ldots,x_m)]$ \\
iff & $y \in R^{(0)}_\varphi[(a_1,\ldots,a_n,x_1,\ldots,x_m)]$ \\
iff & $R_\varphi(y,a_1,\ldots,a_n,y_1,\ldots,y_m)$,
\end{tabular}
\begin{tabular}{rl}
& $V(\nomj) \leq V(\psi(\cnomn_1, \ldots, \cnomn_m,\nomi_1, \ldots, \nomi_n))$ \\
iff & $\val{\nomj}_V \subseteq \val{\psi(\cnomn_1, \ldots, \cnomn_m,\nomi_1, \ldots, \nomi_n)}_V$ \\ 
iff & $b^{\uparrow\downarrow} \subseteq R^{(0)}_\psi[(x_1,\ldots,x_m,a_1,\ldots,a_n)]$ \\
iff & $b \in R^{(0)}_\psi[(x_1,\ldots,x_m,a_1,\ldots,a_n)]$ \\
iff & $R_\psi(y,x_1,\ldots,x_m,a_1,\ldots,a_n)$.
\end{tabular}
\par}}
\smallskip

\noindent As is clear from Definition \ref{def:rel_associated_with_skeleton_polarities}, the expressions $R_\varphi(y,a_1,\ldots,a_n,y_1,\ldots,y_m)$ and $R_\psi(y,x_1,\ldots,x_m,a_1,\ldots,a_n)$ are abbreviations for $\langpolarities$-formulas:

\begin{example}
\label{eg:generalized_icomp}
Consider the inequality $\varphi(\nomi_1,\nomi_2)\leq\cnomm$ where $\varphi(\nomi_1,\nomi_2) \coloneqq \Diamond \nomi_1 \circ \Diamond \nomi_2$. By Definition \ref{def:rel_associated_with_skeleton_polarities}, $R_\varphi \coloneqq R_\circ ;_I (R_\Diamond,R_\Diamond) \subseteq X \times A \times A$, where $R_\circ\subseteq X \times A \times A$, and $R_\Diamond \subseteq X \times A$. By Definition \ref{def:generalized_icomp}, for any $x \in X$, $a, b \in A$,
\smallskip

{{\centering
\begin{tabular}{rrl}
$(x, a, b) \in R_\circ ;_I (R_\Diamond,R_\Diamond)$ & iff & $x \in R_\circ^{(0)}[ R_\Diamond^{(0)}[a]^\downarrow, R_\Diamond^{(0)}[b]^\downarrow  ]$ \\
&iff & $x \in \{ y \in X : \forall c\forall d(c \in R_\Diamond^{(0)}[a]^\downarrow \metaand d \in R_\Diamond^{(0)}[b]^\downarrow  \Rightarrow R_\circ(y,c,d))  \}$ \\
&iff & $\forall c\forall d(c \in R_\Diamond^{(0)}[a]^\downarrow \metaand d \in R_\Diamond^{(0)}[b]^\downarrow  \Rightarrow R_\circ(x,c,d)) $ \\
&iff & $\forall c\forall d( \forall z(zR_\Diamond a \Rightarrow cIz) \metaand \forall w(w R_\Diamond b \Rightarrow dIw)  \Rightarrow R_\circ(x,c,d)) $.
\end{tabular}
\par}}
\end{example}

The formulas in Definition \ref{def:inversedisjunct} can be translated as follows for polarity based semantics.

\begin{definition}
\label{def:inversedisjunct_inpolarities}
A {\em positive} (resp.\ {\em negative}) {\em inverse disjunct} is a $\langpolarities$-formula $\theta^+(r)$ (resp.\ $\theta^-(r)$) defined inductively together with a variable $r \in \mathsf{AtObj} \cup \mathsf{AtAttr}$ as follows:
\begin{description}
    \item[Base case 1:] $\theta^+(r) \coloneqq R_\varphi x\overline s$ (resp.\ $\theta^-(r) \coloneqq \metanot R_\varphi x\overline s$), where $\varphi$ is a positive definite skeleton formula, $r = x$, and each $s$ in $\overline s$ is an object variable if it is positive in $\varphi$, a feature variable otherwise.
    \item[Base case 2:] $\theta^+(r) \coloneqq R_\psi a\overline s$ (resp.\ $\theta^-(r) \coloneqq \metanot R_\psi a\overline s$), where $\psi$ is a negative definite skeleton formula, $r = a$, and each $s$ in $\overline s$ is a feature variable if it is positive in $\psi$, an object variable otherwise.
    \item[Conjunction/Disjunction:] $\theta^+(r) \coloneqq \bigmetaand_i \theta_i^+(r)$ (resp.\ $\theta^-(r) \coloneqq \bigmetaor_i \theta_i^-(r)$); 
    \item[Quantification 1:] $\theta^+(r) \coloneqq\forall\overline s[ \metanot R_\psi a \overline s \Rightarrow \bigmetaor_i \theta_i^-(s_i)]$ (resp.\ $\theta^-(r) \coloneqq\exists \overline s[\bigmetaand_i \theta_i^+(s_i) \metaand \metanot R_\psi a\overline s]$), where $r \coloneqq a$, $\psi$ is a negative definite skeleton formula, and each $\theta_i^-$ (resp.\ $\theta^+_i$) is a negative (resp.\ positive) inverse disjunct where $a$ does not occur. The formula $\psi$ can be just a variable.
    \item[Quantification 2:] $\theta^+(r) \coloneqq \forall\overline s[\metanot R_\varphi x\overline s \Rightarrow \bigmetaor_i \theta_i^-(s_i)]$ (resp.\ $\theta^-(r) \coloneqq\exists \overline s[\bigmetaand_i \theta_i^+(s_i) \metaand \metanot R_\varphi x \overline s]$), where $r \coloneqq x$, $\varphi$ is a positive definite skeleton formula, and each $\theta_i^-$ (resp.\ $\theta^+_i$) is a negative (resp.\ positive) inverse disjunct where $x$ does not occur. The formula $\varphi$ can be just a variable.
\end{description}
\end{definition}

\begin{example}
The {\inductiveinversecorrespondent} from Example \ref{eg:goranko_inv_to_vss}
\[
\forall \nomj, \cnomm(
\exists \nomi_1 \left(
\Diamond\nomi_1 \nleq \cnomm
\metaand
    \forall \cnomn_1 \left(
    \nomi_1 \nleq \Box\Box\cnomn_1
    \Rightarrow
        \exists \nomi_2 \left(
        \Diamondblack(\Diamond\nomi_2 \circ \Diamondblack \nomj) \nleq \cnomn_1
        \metaand
        \nomi_2 \leq \Box\nomj
        \right)
    \right)
\right) 
\metaor \nomj \leq \cnomm
)
\]
can be instantiated to the following setting as follows:
\[
\forall a, x(
\exists b_1 \left(
\metanot xR_\Diamond b_1
\metaand
    \forall y_1 \left(
    \metanot b_1 (R_\Box ;_I R_\Box) y_1
    \Rightarrow
        \exists b_2 \left(
        (y_1, b_2, a) \notin 
        \Diamondblack \: ;_I (R_\circ ;_I (R_\Diamond, \Diamondblack) )
        \metaand
        b_2 R_\Box a
        \right)
    \right)
\right) 
\metaor a I x
).
\]
\end{example}

\section{More examples}
\label{sec:moreexamples}

The present section collects some examples that instantiate different signatures, and analyzes them using different semantics.

\subsection{Correspondents of generalized modal reduction principles}
\label{ssec:simpleexamples}
We discuss some examples based on various LE-languages $\langbase$, the corresponding extended language  $\langbase^\ast$ of which has parameters $\mathcal{F} = \{\circ, \Diamondblack \}$ and $\mathcal{G} = \{\Box, \backslash, \slash \}$, where the residuals of $\circ$ are $\slash$ and $\backslash$. We recall that in polarity based frames, an atom $Rxyz$ is translated as $\nomj_y \circ \nomj_z \leq \cnomm_x$ (cf.\ Section \ref{ssec:polarity_based_semantics}), while in residuated frames it is translated as $\nomj_x \leq \nomj_y \circ \nomj_z$ (cf.\ \cite{residuatedframes}).

\paragraph{Exchange, weakening, and contraction.} Let $\langbase$ be s.t.~$\mathcal{F} = \{\circ\}$ and $\mathcal{G} = \{\Box\}$.
In the (distributive) setting of the  frames in \cite{fussner2019residuation} dual to perfect residuated algebras \cite{GEHRKE20162711}, consider the following first order formulas.\footnote{Residuated algebras are bounded distributive lattices with an operation $\circ$ which has both residuals. By Priestley duality, natural dual ordered relational structures arise, where an inequality $\nomj \leq \nomi$ is interpreted as $x_i \leq x_j$, and $\nomj \leq \nomi \circ \nomh$ as $R_\circ x_j x_i x_h$.}
\smallskip

{{\centering 
\begin{tabular}{rl}
& $R_\circ xab \Rightarrow R_\circ xba$ \\
i.e.\ & $\nomj_x \leq \nomj_a \circ \nomj_b \Rightarrow \nomj_x \leq \nomj_b \circ \nomj_a$ \\
iff & $\nomj_a \circ \nomj_b \leq \nomj_b \circ \nomj_a$ \\
iff & $a \circ b \leq b \circ a$.
\end{tabular}
\begin{tabular}{rl}
& $R_\circ xab \Rightarrow a \leq x$ \\
i.e. & $\nomj_x \leq \nomj_a \circ \nomj_b \Rightarrow \nomj_x \leq \nomj_a$ \\
iff & $\nomj \circ \nomi \leq \nomj$ \\
iff & $a \circ b \leq a$.
\end{tabular}
\begin{tabular}{rl}
& $R_\circ xxx$ \\
i.e.\ & $\nomj_x \leq \nomj_x \circ \nomj_x $ \\
iff & $a \leq a \circ a$.\\
& $\quad$\\
\end{tabular}
\par}}
\smallskip

Similarly, in (non-necessarily distributive) polarity-based frames (cf.\ Section \ref{ssec:polarity_based_semantics}),
\smallskip

{{\centering 
\begin{tabular}{rl}
& $R_\circ xba \Rightarrow R_\circ xab$ \\
i.e.\ & $\nomj_b \circ \nomj_a \leq \cnomm_x \Rightarrow \nomj_a \circ \nomj_b \leq \cnomm_x$ \\
iff & $\nomj_a \circ \nomj_b \leq \nomj_b \circ \nomj_a$ \\
iff & $a \circ b \leq b \circ a$.
\end{tabular}
\begin{tabular}{rl}
& $Iax \Rightarrow R_\circ xab$ \\
i.e.\ & $\nomj_a \leq \cnomm_x \Rightarrow  \nomj_a\circ \nomj_b \leq \cnomm_x $ \\
iff & $\nomj_a \circ \nomj_b \leq \nomj_a$ \\
iff & $a \circ b \leq a $.
\end{tabular}
\begin{tabular}{rl}
& $Rxaa \Rightarrow Iax$ \\
i.e.\ & $\nomj_x \circ \nomj_a \leq \cnomm_x \Rightarrow \nomj_a \leq \cnomm_x$ \\
iff & $\nomj_a \leq \nomj_a \circ \nomj_a$\\
iff & $a \leq a \circ a$.
\end{tabular}
\par}}

\paragraph{Associativity.}
For the same language and settings as above, consider the following first order formula in the dual structures of residuated algebras:
\smallskip

{{\centering 
\begin{tabular}{rl}
& $R_\circ xad \metaand R_\circ dbc \Rightarrow \exists e(R_\circ xec \metaand R_\circ eab)$ \\
i.e.\ & $\nomj_x\leq \nomj_a \circ \nomj_d \metaand \nomj_d\leq \nomj_b\circ \nomj_c \Rightarrow \exists e(\nomj_x\leq \nomj_e\circ \nomj_c \metaand \nomj_e\leq \nomj_a\circ \nomj_b)$,
\end{tabular}
\par}}
\smallskip

\noindent which, by complete join-primeness  (cf.\ Section \ref{sec:starting_example} for a similar argument) is equivalent to
\[
\nomj_a \circ (\nomj_b\circ \nomj_c) \leq (\nomj_a\circ \nomj_b)\circ \nomj_c,
\]
that, by Lemma \ref{lemma:vss_substitute_nomcnom_var}, can be rewritten as $(a \circ b) \circ c \leq (a \circ b) \circ c$.
This axiom is an example of a {\em generalized modal reduction principle}, since it does not contain any occurrence of $\wedge$ and $\vee$, and at least some variable does not occurs uniformly in it. In \cite{dairapaper}, a connection has been established between the first order correspondents in polarity based frames and in Kripke frames of LE-inductive modal reduction principles with unary connectives. Specifically, the first order correspondent of any given inductive modal reduction principle in each type of semantic structures can be expressed as an inclusion of relations, and the one on polarity-based frames is the {\em lifted} version of the one on Kripke frames (cf.~\cite[Definition 4.17, Theorem 4.20]{dairapaper}). The first order axiom above can be rewritten as 
\begin{equation}
\label{eq:associativity_kripke}
R_\circ\, \circ ( R_\circ, \Delta) \subseteq R_\circ\, \circ (\Delta, R_\circ),
\end{equation}
where $\Delta$ is the identity relation, and, for any ternary relation $R$ and binary relations $S$ and $T$, $(R \circ (S, T)) xyz$ iff $\exists a \exists b(Rxab \metaand Say \metaand Tbz)$.
In the setting of polarity-based frames, the {\em lifting} of \eqref{eq:associativity_kripke} is the following inclusion of relations:
\begin{equation}
\label{eq:associativity_polarity}
R_\circ\, ;_I(J, R_\circ) \subseteq R_\circ\, ;_I( R_\circ, J),
\end{equation}
which can be written as a first order condition as follows:
\smallskip

{{\centering 
\begin{tabular}{rrcl}
&  $R_\circ\, ;_I(J, R_\circ)$ & $\subseteq$ & $R_\circ\, ;_I( R_\circ, J)$ \\
iff & $x \in (R_\circ\, ;_I(J, R_\circ))^{(0)}[a, b, c]$ & $\Rightarrow$ & $x \in (R_\circ\, ;_I( R_\circ, J))^{(0)}[a, b, c]$ \\
iff & $\forall d(d \in (R_\circ^{(0)}[b,c])^\downarrow \Rightarrow R_\circ xad)$ & $\Rightarrow$ & $\forall d(d \in (R_\circ^{(0)}[a,b])^\downarrow \Rightarrow R_\circ xdc)$ \\
iff & $\forall d(\forall y(R_\circ ybc \Rightarrow Idy) \Rightarrow R_\circ xad)$ & $\Rightarrow$ & $\forall d(\forall y(R_\circ yab \Rightarrow Idy)\Rightarrow R_\circ xdc)$. \\
\end{tabular}
\par}}
\smallskip

\noindent The inclusion of relations \eqref{eq:associativity_polarity} can be translated in the language of ALBA as follows: let $\varphi\coloneqq z\circ(y\circ v)$ and $\varphi'\coloneqq (z\circ y)\circ v$. Then  $R_\varphi = R_\circ\, ;_I(J, R_\circ)$ and $R_{\varphi'} = R_\circ\, ;_I( R_\circ, J)$; hence,
\smallskip

{{\centering 
\begin{tabular}{rrcl}
&  $R_\circ\, ;_I(J, R_\circ)$ & $\subseteq$ & $R_\circ\, ;_I( R_\circ, J)$ \\
iff & $x \in (R_\circ\, ;_I(J, R_\circ))^{(0)}[a, b, c]$ & $\Rightarrow$ & $x \in (R_\circ\, ;_I( R_\circ, J))^{(0)}[a, b, c]$ \\
iff & $\varphi(\nomj_a, \nomj_b, \nomj_c) \leq \cnomm_x$ & $\Rightarrow$ & $\varphi'(\nomj_a, \nomj_b, \nomj_c) \leq \cnomm_x$ \\
iff & $\nomj_a \circ (\nomj_b \circ \nomj_c) \leq \cnomm_x$ & $\Rightarrow$ & $(\nomj_a \circ \nomj_b) \circ \nomj_c \leq \cnomm_x$ \\
iff & $(\nomj_a \circ \nomj_b) \circ \nomj_c$ & $\leq $ & $\nomj_a \circ (\nomj_b \circ \nomj_c)$,
\end{tabular}
\par}}
\smallskip

\noindent which is, again, equivalent to the LE-inequality $(a \circ b) \circ c \leq (a \circ b) \circ c$.

\paragraph{Axiom K.} 
Let $\langbase$ be s.t. $\mathcal{F} = \varnothing$ and $\mathcal{G} = \{ \Box, \backslash \}$.
The inclusion of relations 
\begin{equation}
\label{eq:axiomk_incl_rel}
R_\Box \, ;_I (R_\backslash \, ;_I (R_{\Diamondblack}, I))\subseteq R_\backslash \, ;_I (J, R_\Box)
\end{equation} 
can be expressed in the first order language of polarity-based  $\langbase$-frames as follows:
\smallskip

{{\centering 
\begin{tabular}{rrcl}
& $R_\Box \, ;_I (R_\backslash \, ;_I (R_{\Diamondblack}, I))$ & $\subseteq$ & $R_\backslash \, ;_I (J, R_\Box)$ \\
iff & $a \in (R_\Box \, ;_I (R_\backslash \, ;_I (R_{\Diamondblack}, I)))^{(0)}[b, x]$ & $\Rightarrow$ & $a \in (R_\backslash \, ;_I (J, R_\Box))^{(0)}[b, x]$ \\
iff & $a \in R_\Box^{(0)}[(R_\backslash \, ;_I (R_{\Diamondblack}, I))^{(0)}[b, x]^\uparrow]$ & $\Rightarrow$ & $a \in R_\backslash^{(0)}[J^{(0)}[b]^\downarrow, R_\Box^{(0)}[x]^\uparrow]$ \\
iff & $a \in R_\Box^{(0)}[R_\backslash^{(0)}[R_{\Diamondblack}^{(0)}[b]^\downarrow, x^{\downarrow \uparrow}]^\uparrow]$ & $\Rightarrow$ & $a \in R_\backslash^{(0)}[b^{\uparrow\downarrow}, R_\Box^{(0)}[x]^\uparrow]$ \\
iff & $\forall y (y \in R_\backslash^{(0)}[R_{\Diamondblack}^{(0)}[b]^\downarrow, x^{\downarrow \uparrow}]^\uparrow \Rightarrow R_\Box ay)$ & $\Rightarrow$ & $\forall y(y \in R_\Box^{(0)}[x]^\uparrow \Rightarrow R_\backslash aby)$ \\
iff & $\forall y (\forall c ( c \in R_\backslash^{(0)}[R_{\Diamondblack}^{(0)}[b]^\downarrow, x^{\downarrow \uparrow}] \Rightarrow Icy) \Rightarrow R_\Box ay)$ & $\Rightarrow$ & $\forall y( \forall c(c \in R_\Box^{(0)}[x] \Rightarrow Icy)  \Rightarrow R_\backslash aby)$ \\
iff & $\forall y (\forall c ( 
\forall d(d \in R_{\Diamondblack}^{(0)}[b]^\downarrow \Rightarrow  R_\backslash cdx) \Rightarrow Icy) \Rightarrow R_\Box ay)$ & $\Rightarrow$ & $\forall y( \forall c(R_\Box cx \Rightarrow Icy)  \Rightarrow R_\backslash aby)$ \\
iff & $\forall y (\forall c ( 
\forall d(
\forall z (R_{\Diamondblack}zb \Rightarrow I cz)
\Rightarrow  R_\backslash cdx) \Rightarrow Icy) \Rightarrow R_\Box ay)$ & $\Rightarrow$ & $\forall y( \forall c(R_\Box cx \Rightarrow Icy)  \Rightarrow R_\backslash aby)$. \\
\end{tabular}
\par}}
\smallskip

\noindent This inclusion of relations can also be translated as a pure $\langmeta$-quasi-inequality. Indeed, let $\psi \coloneqq \Box(\Diamondblack u \backslash v)$ and $\psi' \coloneqq u \backslash \Box v$. Then, $R_\psi = R_\Box \, ;_I (R_\backslash \, ;_I (R_{\Diamondblack}, I))$, and $R_{\psi'}= R_\backslash \, ;_I (J, R_\Box)$; hence,
\smallskip

{{\centering
\begin{tabular}{rrcl}
& $R_\Box \, ;_I (R_\backslash \, ;_I (R_{\Diamondblack}, I))$ & $\subseteq$ &  $R_\backslash \, ;_I (J, R_\Box)$ \\
iff & $a \in (R_\Box \, ;_I (R_\backslash \, ;_I (R_{\Diamondblack}, I)))^{(0)}[b, x]$ & $\Rightarrow$ & $a \in (R_\backslash \, ;_I (J, R_\Box))^{(0)}[b, x]$ \\
iff & $\nomi_a \leq \Box(\Diamondblack \nomj_b \backslash \cnomn_x)$ & $\Rightarrow$ & $\nomi_a \leq \nomj_b \backslash \Box \cnomn_x$,
\end{tabular}
\par}}
\smallskip

\noindent which satisfies Definition \ref{def:crypto_inductive_inverse_corr}. By Theorem \ref{thm:invcorr}, this quasi-inequality is readily shown to be equivalent to  $\Box(\Diamondblack v_1 \backslash v_2) \leq v_1 \backslash \Box v_2$, which, as discussed in Example \ref{eg:vssaxiomk_to_inductive}, is crypto-inductive and is equivalent to
\begin{equation}
\label{eq:axiomk}
\Box(p \backslash q) \leq \Box p \backslash \Box q.
\end{equation}
This axiom is another example of a  generalized modal reduction principle, and the inclusion of relations 
 \eqref{eq:axiomk_incl_rel} is the lifted version of the following inclusion of relations on Kripke $\langbase$-frames:
\[
R_\backslash \, \circ (\Delta, R_\Box) \subseteq R_\Box \, \circ (R_\backslash \, \circ (R_{\Diamondblack}, \Delta)),
\]
where $\Delta$ is the identity relation, and, for any ternary relation $R$ and binary relations $S$ and $T$, $(R \circ (S, T)) xyz$ iff $\exists a \exists b(Rxab \metaand Say \metaand Tbz)$. The inclusion of relations above can be rewritten as follows in the language of distributive $\langbase$-frames:
\smallskip

{{\centering
\begin{tabular}{rrcl}
& $R_\backslash \, \circ (\Delta, R_\Box)$ & $\subseteq$ & $R_\Box \, \circ (R_\backslash \, \circ (R_{\Diamondblack}, \Delta))$ \\
iff  & $(R_\backslash \, \circ (\Delta, R_\Box))xyz$ & $\Rightarrow$ & $(R_\Box \, \circ (R_\backslash \, \circ (R_{\Diamondblack}, \Delta)))xyz$ \\
iff  & $\exists b (R_\backslash xyb \metaand R_\Box bz)$ & $\Rightarrow$ & $\exists c(R_\Box xc \metaand (R_\backslash \, \circ (R_{\Diamondblack}, \Delta))cyz )$ \\
iff  & $R_\backslash xyb \metaand R_\Box bz$ & $\Rightarrow$ & 
$\exists c(R_\Box xc \metaand 
    \exists d(R_\backslash cdz 
    \metaand 
    R_{\Diamondblack} dy) ).$ \\
\end{tabular}
\par}}

\noindent In the setting of Kripke frames, the first order formula above can be translated in $\langmeta$ as follows:
\[
\nomj_y \backslash \cnomm_b \leq \cnomm_x
\metaand
\Box \cnomm_z \leq \cnomm_b
\Rightarrow
\exists \cnomm_c \exists \nomj_d(\Box \cnomm_c \leq \cnomm_x \metaand \nomj_d \backslash \cnomm_z \leq \cnomm_c \metaand \nomj_d \leq \Diamondblack \nomj_y ).
\]
By the complete join (resp.\ meet) primeness of the completely join (resp.\ meet) irreducibles in the distributive setting, the quasi-inequality above can be rewritten as follows (see Section \ref{sec:starting_example} and \cite[Lemma 44]{inversedle}):

{{\centering 
\begin{tabular}{rl}
& $
\nomj_y \backslash \cnomm_b \leq \cnomm_x
\metaand
\Box \cnomm_z \leq \cnomm_b
\Rightarrow
\exists \cnomm_c \exists \nomj_d(\Box \cnomm_c \leq \cnomm_x \metaand \nomj_d \backslash \cnomm_z \leq \cnomm_c \metaand \nomj_d \leq \Diamondblack \nomj_y )$\\ 
iff & $
\nomj_y \backslash \Box \cnomm_z  \leq \cnomm_x
\Rightarrow
\exists \cnomm_c \exists \nomj_d(\Box \cnomm_c \leq \cnomm_x \metaand \nomj_d \backslash \cnomm_z \leq \cnomm_c \metaand \nomj_d \leq \Diamondblack \nomj_y )$ \\
iff & $
\nomj_y \backslash \Box \cnomm_z  \leq \cnomm_x
\Rightarrow
\exists \cnomm_c(\Box \cnomm_c \leq \cnomm_x \metaand \Diamondblack \nomj_y  \backslash \cnomm_z \leq \cnomm_c)$ \\
iff & $
\nomj_y \backslash \Box \cnomm_z  \leq \cnomm_x
\Rightarrow
\Box (\Diamondblack \nomj_y  \backslash \cnomm_z) \leq \cnomm_x $ \\
iff & $
\Box (\Diamondblack \nomj_y  \backslash \cnomm_z) \leq \nomj_y \backslash \Box \cnomm_z$.
\end{tabular}
\par}}

\noindent  The pure inequality above is identical to the one obtained in the polarity-based setting, and as discussed in Example \ref{eg:vssaxiomk_to_inductive}, it is equivalent to \eqref{eq:axiomk}.

\subsection{Dunn and Fischer Servi axioms}

Consider the distributive setting of intuitionistic modal logic, where $\mathcal{F} = \{\Diamond, \wedge\}$ and $\mathcal{G} = \{\Box, \rightarrow, \vee\}$, and the non-distributive setting of substructural modal logic, where $\mathcal{F} = \{\Diamond\}$ and $\mathcal{G} = \{\Box, \backslash \}$. Relational frames are partially ordered sets $X = (W, \leq)$ endowed with binary relations $R_\Box$ and $R_\Diamond$ such that ${\leq \circ {R_\Box} \circ \leq} \subseteq R_\Box$ and ${\geq \circ {R_\Diamond} \circ \geq} \subseteq R_\Diamond$. In this semantics, nominals (resp.\ conominals) correspond to principal up-sets (resp.\ complements of principal down-sets); therefore, in what follows, we will index nominals and conominals with  individual variables of the frame correspondence language to facilitate the translation. Also, recall that, in perfect  distributive algebras, completely join-irreducibles and completely meet irreducible elements are order isomorphic to each other via the assignments $\nomj\mapsto \bigvee \{\nomi\mid \nomj\nleq \nomi\} $ and $\cnomm\mapsto \bigwedge \{\cnomn\mid \cnomn\nleq \cnomm\} $. We will assign the same index to nominals and conominals corresponding to each other via these assignments. With this notation in place, pure inequalities can be translated to first order formulas as follows:
\smallskip

{{\centering
\begin{tabular}{rl}
& $\nomj_x \leq \nomj_y$ \\
i.e. & ${\uparrow}x \subseteq {\uparrow}y$ \\
iff & $y \leq x$
\end{tabular}
\begin{tabular}{rl}
& $\nomj_x \leq \Diamond \nomj_y$ \\
i.e. & ${\uparrow}x \subseteq R_\Diamond^{-1}[{\uparrow}y]$ \\
iff & $R_\Diamond xy$
\end{tabular}
\begin{tabular}{rl}
& $\Box \cnomm_x \leq \cnomm_y$ \\
i.e. & ${\downarrow} y \subseteq R_\Box^{-1}[{\downarrow}x]$ \\
iff & $R_\Box yx$
\end{tabular}
\begin{tabular}{rl}
     & $\nomj_x \leq \Diamondblack \nomj_y $\\
iff & $R_{\Diamondblack} xy$ \\
iff & $R_\Box yx$.
\end{tabular}
\par}}
\smallskip

\noindent Consider now the following first order formula (where all variables are universally quantified)
\smallskip

{{\centering
\begin{tabular}{rl}
     &  $ R_\Diamond zx  \metaand y \leq w \metaand z \leq w \Rightarrow \exists t(R_\Diamond wt \metaand x \leq t \metaand R_\Box yt)$ \\
i.e. & $\nomj_z \leq \Diamond \nomj_x  \metaand \nomj_w \leq \nomj_y \metaand \nomj_w \leq \nomj_z \Rightarrow \exists \nomj_t(\nomj_w \leq \Diamond \nomj_t \metaand \nomj_t \leq \nomj_x \metaand \nomj_t \leq \Diamondblack \nomj_y)$ \\
iff & $\nomj_z \leq \Diamond \nomj_x  \metaand \nomj_w \leq \nomj_y\wedge \nomj_z \Rightarrow \exists \nomj_t(\nomj_w \leq \Diamond \nomj_t \metaand \nomj_t \leq \nomj_x \wedge \Diamondblack \nomj_y)$ \\
\end{tabular}
\par}}
\smallskip

\noindent by complete join primeness of the completely join-irreducibles, it is equivalent to
\begin{equation}
\label{eq:fischerservi_halfprocessed}
\nomj_w \leq \nomj_y\wedge \Diamond \nomj_x \Rightarrow \nomj_w \leq \Diamond (\nomj_x \wedge \Diamondblack \nomj_y).
\end{equation}
By residuation and the connection between nominals and conominals in the distributive setting, 
\[
\nomj_w \leq \nomj_y \wedge \Diamond \nomj_x
\quad \mbox{iff} \quad
\nomj_y \wedge \Diamond \nomj_x \nleq \cnomm_w
\quad \mbox{iff} \quad
\nomj_y \nleq \Diamond \nomj_x \rightarrow \cnomm_w
\quad \mbox{iff} \quad
\Diamond \nomj_x \rightarrow \cnomm_w \leq \cnomm_y.
\]
Hence, \eqref{eq:fischerservi_halfprocessed} is equivalent to the following quasi inequality:
\[
\Diamond \nomj_x \rightarrow \cnomm_w \leq \cnomm_y\Rightarrow \nomj_w \leq \Diamond (\nomj_x \wedge \Diamondblack \nomj_y),
\]
which can be rewritten as the following one, which satisfies Definition \ref{def:crypto_inductive_inverse_corr}.
\begin{equation}
\label{eq:fischerservi_crypto_inverse_corr}
\Diamond (\nomj_x \wedge \Diamondblack \nomj_y) \leq \cnomm_w
\Rightarrow 
\nomj_y \leq \Diamond \nomj_x \rightarrow \cnomm_w
\end{equation}
Hence, by Theorem \ref{thm:cryptoinvcorr}, it is equivalent to a crypto-inductive $\langbase$-inequality computed as follows:
\smallskip

{{\centering
\begin{tabular}{rl}
     & $\Diamond (\nomj_x \wedge \Diamondblack \nomj_y) \leq \cnomm_w
\Rightarrow 
\nomj_y \leq \Diamond \nomj_x \rightarrow \cnomm_w$ \\
iff & $\nomj_x \leq v_x \metaand \nomj_y \metaand \Diamond (v_x \wedge \Diamondblack v_y) \leq \cnomm_w
\Rightarrow \nomj_y \leq \Diamond \nomj_x \rightarrow \cnomm_w$ \\
iff & $v_y \leq \Diamond v_x \rightarrow \Diamond (v_x \wedge \Diamondblack v_y)$
\end{tabular}
\par}}
\smallskip

\noindent The last inequality is  crypto $\varepsilon$-inductive for $\varepsilon(v_y) = 1 = \varepsilon(v_x) = 1$. The $\mathsf{MVTree}$ of the non-critical occurrence of $v_x$ contains only $v_x$, while that of the non-critical occurrence of $v_y$ contains $\Diamondblack v_y$. Hence, Proposition \ref{prop:cryptotoinductive} yields
\smallskip

{{\centering 
\begin{tabular}{rl}
     & $v_y \leq \Diamond v_x \rightarrow \Diamond (v_x \wedge \Diamondblack v_y)$ \\
iff & $v_x \leq p \metaand \Diamondblack v_y \leq q \Rightarrow v_y \leq \Diamond v_x \rightarrow \Diamond (p \wedge q)$\\
iff & $v_x \leq p \metaand v_y \leq \Box q \Rightarrow v_y \leq \Diamond v_x \rightarrow \Diamond (p \wedge q)$\\
iff & $\Box q \leq \Diamond p \rightarrow \Diamond (p \wedge q)$.\\
iff & $\Diamond p\wedge \Box q \leq  \Diamond (p \wedge q)$.
\end{tabular}
\par}}
\smallskip

\noindent If we do not consider $\wedge$ to be part of the language, is also crypto-inductive for $\varepsilon(v_x) = \varepsilon(v_y) = 1$. The procedure in Proposition \ref{prop:cryptotoinductive} yields
\smallskip

{{\centering 
\begin{tabular}{rl}
     & $v_y \leq \Diamond v_x \rightarrow \Diamond (v_x \wedge \Diamondblack v_y)$ \\
iff & $\Diamondblack v_y \leq q \Rightarrow v_y \leq \Diamond v_x \rightarrow \Diamond (v_x \wedge q)$\\
iff & $v_x \wedge q \leq p \metaand \Diamondblack v_y \leq q \Rightarrow v_y \leq \Diamond v_x \rightarrow \Diamond p$\\
iff & $v_x \leq q \rightarrow p \metaand v_y \leq \Box q \Rightarrow v_y \leq \Diamond v_x \rightarrow \Diamond p$\\
iff & $\Box q \leq \Diamond (q \rightarrow p) \rightarrow \Diamond p$ \\
iff & $\Diamond (q \rightarrow p) \leq \Box q \rightarrow \Diamond p$. 
\end{tabular}
\par}}
\smallskip

\section{Conclusions}
\label{sec:conclusions}

\paragraph{Contributions of the present paper.} 
In the present paper, we generalize Kracht-Kikot's inverse correspondence theory from classical modal logic to arbitrary LE-logics. Just like Kracht's theory can be regarded as the converse of classical Sahlqvist theory, the present result can be understood as the converse of unified correspondence theory for LE-logics: indeed, this result uniformly applies across arbitrary LE-signatures, thanks to the fact that it is set in the algebraic environment corresponding to LE-logics, and it hinges on the order-theoretic properties of the interpretation of the logical connectives.

\paragraph{Future work.} A natural refinement of the present results concerns the setting of normal  distributive lattice expansions (normal DLEs), initiated in \cite{inversedle}. The different order-theoretic properties of this environment allow for a number of technical and methodological differences between the present setting and this refinement.
In particular, the complete primeness of the completely join- and meet-irreducible elements can be exploited in the distributive setting, as well as the fact  that the completely join- and meet-irreducibles bijectively correspond to each other. These two properties allow for a finer-grained reduction of the relational atomic constituents, and for a more straightforward translation into commonly used frame-correspondence languages.

An open question concerns the extension of the present results to logics with fixpoints. These logics have already been accounted for within unified correspondence \cite{CoGhPa14,alba_dmucalculus,conradie2017constructive}, and the frame correspondence language is a first order language with fixpoints. 

Another natural open question concerns the extension of these results to non-normal logics. These logics have also been covered by the general theory of unified correspondence \cite{palmigiano2017sahlqvist,palmigiano2020constructive,chen2022non}. A possible way to tackle this problem might be via suitable embeddings into multi-type normal environments in the style of what was done in \cite{chen2022non}.

A third natural direction concerns the  relationship of the present results with the recent theory of {\em parametric correspondence}  \cite{dairapaper}, as discussed in Section \ref{ssec:simpleexamples}. This theory establishes systematic connections between the first order correspondents of certain axioms across different semantic settings. In particular, the results in \cite{dairapaper} hinge on the possibility to define relations associated with (unary) definite skeleton formulas thanks to suitable  definitions of relational compositions (which depend on different semantic settings), and rewrite the first order correspondents of definite Sahlqvist/inductive inequalities (modal reduction principles) in terms of inclusions of such relations. In the present paper, we have followed a very similar strategy, based on the definition of relations associated to (non-necessarily unary) definite skeleton formulas (see Section \ref{ssec:kripke} and Section \ref{ssec:polarity_based_semantics}): the definition of inductive inverse correspondent serves as a connection point which allows to recognize the syntactic shape, based on such relations, of the various first order correspondents across the different semantic settings. Hence, a natural question is whether these systematic connections can be extended to the characterization of the first order correspondents of inductive inequalities.

Finally, we noticed that the syntactic shape of inductive inverse correspondents (cf.\ Definition \ref{def:inductiveinvcorr}) is a subclass of generalized geometric implications in the chosen first language. This syntactic shape has been used in structural proof theory in the context a research program aimed at generating analytic calculi for large classes of axiomatic extensions of classical normal modal logic \cite{negri2016sysrul}. In the light of this, the present result can pave the way to the extension of these proof theoretic results to DLE and LE logics. 
\bibliographystyle{plain}
\bibliography{references}
\appendix

\section{Flattening skeleton inequalities in the LE setting}
\label{sssec:flatteningskeleton}

In Remark \ref{remark:flatinequalities}, we have discussed and justified our choice to include complex skeleton formulas in the atoms of $\langmeta$, due to ease of presentation, and because they do not add more power to the usual first order correspondence language in the setting of distributive (and classical) modal logic. In the present section, we show how the flattening procedure can be carried on under certain circumstances even when the inequalities of language contain at most one connective. 

The flattening procedure can be applied as in the rest of Section \ref{ssec:flattening}, and Corollary \ref{cor:firstapprox_inductive} can be easily used to flatten inequalities $\nomj \leq \varphi$ and $\psi \leq \cnomm$, where $\varphi$ and $\psi$ are positive and negative skeleton formulas respectively. It remains to show how to flatten inequalities of shape
$\nomj \leq \varphi$ and $\psi \leq \cnomm$, where $\varphi$ and $\psi$ are negative and positive skeleton formulas respectively, which appear in antecedents of implications whose antecedent contains $\nomj$ (resp.\ $\cnomm$) only once, and whose consequent contains some inequality $\varphi'\leq\psi'$ in which $\nomj$ (resp.\ $\cnomm$) occurs exactly once, and $\varphi'$ and $\psi'$ are skeleton formulas. We show how to recursively flatten inequalities $\nomj \leq \varphi$ (and the similar case $\psi \leq \cnomm$) when $f$ and $g$ connectives do not occur under the scope of $\wedge$ and $\vee$ in $\varphi$ (resp.\ $\psi$).

If $\varphi$ is a finite meet of pure variables $\bigwedge\purev_i$, then $\nomj \leq \varphi$ iff $\bigmetaand\nomj\leq\purev_i$.

Consider now the case in which $\varphi = g(\alpha_{1},\ldots,\alpha_{n},\beta_{1},\ldots,\beta_{m})$, where $g$ is some connective in $\mathcal{G}$ whose first $n$ coordinates are positive, and whose last $m$ coordinates are negative. Since $\varphi'$ and $\psi'$ are skeleton formulas and one of them (without loss of generality $\varphi'$), by Corollary \ref{cor:skeletonadjoint}, it is possible to put $\nomj$ on display: $\varphi' \leq \psi'$ iff $\nomj \leq \mathsf{RA}(\varphi')_{\psi'}$. Hence, since $\nomj$ occurs only twice in the implication
\smallskip

{{\centering
\begin{tabular}{rlr}
& $\forall \nomj(\xi \metaand 
\nomj \leq g(\alpha_{1},\ldots,\alpha_{n},\beta_{1},\ldots,\beta_{m}))
\Rightarrow
\nomj \leq \mathsf{RA}(\varphi')_{\psi'}$ & \\
iff & $\xi \Rightarrow 
g(\alpha_{1},\ldots,\alpha_{n},\beta_{1},\ldots,\beta_{m})
\leq \mathsf{RA}(\varphi')_{\psi'}$ & $\mathbb{A}$ dense \\
iff & $\forall \cnomn(\xi \metaand \mathsf{RA}(\varphi')_{\psi'} \leq \cnomn \Rightarrow 
g(\alpha_{1},\ldots,\alpha_{n},\beta_{1},\ldots,\beta_{m})
\leq \cnomn)$ & $\mathbb{A}$ dense\\
iff & $\forall \cnomn,\overline \nomk,\overline \cnomo(\xi \metaand 
\mathsf{RA}(\varphi')_{\psi'} \leq \cnomn 
\metaand
\bigmetaand_{i=1}^n \alpha_i\leq\cnomo_i
\metaand
\bigmetaand_{i=1}^m \nomk_i \leq \beta_i
\Rightarrow 
g(\overline \cnomo, \overline\nomk)
\leq \cnomn)$ & Corollary \ref{cor:firstapprox_inductive}\\
iff & $\forall \overline \nomk,\overline \cnomo(\xi \metaand 
\bigmetaand_{i=1}^n \alpha_i\leq\cnomo_i
\metaand
\bigmetaand_{i=1}^m \nomk_i \leq \beta_i
\Rightarrow 
g(\overline \cnomo, \overline\nomk)
\leq \mathsf{RA}(\varphi')_{\psi'})$ & $\mathbb{A}$ dense\\
iff & $\forall \nomj, \overline \nomk,\overline \cnomo(\xi \metaand 
\bigmetaand_{i=1}^n \alpha_i\leq\cnomo_i
\metaand
\bigmetaand_{i=1}^m \nomk_i \leq \beta_i
\metaand
\nomj\leq g(\overline \cnomo, \overline\nomk)
\Rightarrow 
\nomj
\leq \mathsf{RA}(\varphi')_{\psi'})$ & $\mathbb{A}$ dense\\
\end{tabular}
\par}}
\smallskip

It is then possible to proceed on $\xi$ and on the newly introduced inequalities $\nomk_i\leq\beta_i$ (resp.\ $\alpha_i\leq\cnomo_i$) from the last and penultimate step of the chain of equivalences above, respectively. The same argument cannot be applied when meets and joins occur arbitrarily in $\varphi$ and $\psi$, as, whenever an inequality $\nomj \leq \varphi_1 \wedge \varphi_2$ (resp.\ $\psi_1 \vee \psi_2 \leq \cnomm$) is split into $\nomj \leq \varphi_1 \metaand \nomj \leq \varphi_2$ (resp.\ $\psi_1 \leq \cnomm \metaand \psi_2 \leq \cnomm$), there is no clear way to rewrite the first formula above, in a shape close the third one, i.e., there is no clear way in which one can rewrite the formula by substituting $\nomj$ (resp.\ $\cnomm$) with a conominal (resp.\ nominal) and making the inequalities in which they occur switch sides of the meta-implication. Example \ref{eg:flattenalsonegskeright} shows one case in which meets and joins make any rewriting of the formula non-obvious.

\begin{example}
\label{eg:flattification:alsoinsideskeleton}
For instance, consider the inequality $\nomj_1 \leq (\Box\Box \nomj_2 \circ \nomj_3) \circfor \Box \Diamond {\rhd} \nomj_2$ which is equivalent to
\[
\forall \cnomo_1, \nomh_1, \nomh_2
(
\nomh_1 \leq \Box\Box \nomj_2 
\metaand
\nomh_2 \leq \nomj_3
\metaand
\forall \nomh_3(\nomh_3 \leq {\rhd} \nomj_2 \Rightarrow \Diamond\nomh_3 \leq \cnomo_1)
\Rightarrow
\nomj_1 \leq (\nomh_1 \circ \nomh_2) \circfor \cnomo_1
),
\]
as shown in Example \ref{eg:flattification1}. If $\nomh_1 \leq \Box\Box \nomj_2$ is not considered an atom (predicate on $\nomh_1$ and $\nomj_2$), then it needs to be further flattified as explained above: the first step is residuate the inequality on the right of the meta-implication to put $\nomh_1$ on display. Therefore,  $\nomj_1 \leq (\nomh_1 \circ \nomh_2) \circfor \cnomo_1$ is equivalently rewritten as $\nomh_1 \circ \nomh_2 \leq \nomj_1 \circfor \cnomo_1$, which is equivalent to $\nomh_1 \leq \nomh_2 \circfor (\nomj_1 \circfor \cnomo_1)$, by residuation. Thus, the starting inequality can be rewritten as
\smallskip

{{\centering
\begin{tabular}{rl}
& $\forall \cnomo_1, \nomh_1, \nomh_2
(
\nomh_1 \leq \Box\Box \nomj_2 
\metaand
\nomh_2 \leq \nomj_3
\metaand
\forall \nomh_3(\nomh_3 \leq {\rhd} \nomj_2 \Rightarrow \Diamond\nomh_3 \leq \cnomo_1)
\Rightarrow
\nomh_1 \leq \nomh_2 \circfor (\nomj_1 \circfor \cnomo_1)
)$ \\
& $\forall \cnomo_1, \cnomo_2, \nomh_2
(
\nomh_2 \circfor (\nomj_1 \circfor \cnomo_1) \leq \cnomo_2
\metaand
\nomh_2 \leq \nomj_3
\metaand
\forall \nomh_3(\nomh_3 \leq {\rhd} \nomj_2 \Rightarrow \Diamond\nomh_3 \leq \cnomo_1)
\Rightarrow
\Box\Box \nomj_2 \leq \cnomo_2
)$
\end{tabular}
\par}} 
\end{example}

\begin{example}
\label{eg:flattenalsonegskeright}
The inequality $((\Diamond p \wedge {\lhd} \Box {\rhd} q) \circ p) \circfor {\rhd} p \leq  {\rhd}(p \circ q)$ in the language of the Lambek-Grishin calculus enriched with unary modalities is $(\Omega,\varepsilon)$-inductive for
$\varepsilon(p) = 1$ and $\varepsilon(q) = 1$, and any $\Omega$. Corollary \ref{cor:firstapprox_inductive} yields that it is equivalent to
\[
\begin{array}{rl}
& \forall \nomj_1,\nomj_2,\nomj_3
\left[
\nomj_1 \leq ((\Diamond p \wedge {\lhd} \Box {\rhd} q) \circ p) \circfor {\rhd} p
\metaand
\nomj_2 \leq p
\metaand 
\nomj_3 \leq q
\Rightarrow
\nomj_1 \leq {\rhd}(\nomj_2 \circ \nomj_3)
\right],
\end{array}
\]
and by Lemma \ref{lemma:ackermann}, its ALBA output is
\[
\forall \nomj_1,\nomj_2,\nomj_3
\left[
\nomj_1 \leq ((\Diamond \nomj_2 \wedge {\lhd} \Box {\rhd} \nomj_3) \circ \nomj_2) \circfor {\rhd} \nomj_2
\Rightarrow
\nomj_1 \leq {\rhd}(\nomj_2 \circ \nomj_3)
\right].
\]
The usual flattification process on $\nomj_1 \leq ((\Diamond \nomj_2 \wedge {\lhd} \Box {\rhd} \nomj_3) \circ \nomj_2) \circfor {\rhd} \nomj_2$ yields
\smallskip

{{\centering
\begin{tabular}{rl}
& $\nomj_1 \leq ((\Diamond \nomj_2 \wedge {\lhd} \Box {\rhd} \nomj_3) \circ \nomj_2) \circfor {\rhd} \nomj_2$ \\
iff & $\forall \cnomo_1,\nomh_1({\rhd} \nomj_2 \leq \cnomo_1 \metaand 
\nomh_1 \leq \Diamond \nomj_2 \wedge {\lhd} \Box {\rhd} \nomj_3
\Rightarrow \nomj_1 \leq (\nomh_1 \circ \nomj_2)\circfor \cnomo_1)$ \\
iff & $\forall \cnomo_1,\nomh_1({\rhd} \nomj_2 \leq \cnomo_1 \metaand 
\nomh_1 \leq \Diamond \nomj_2 \metaand \nomh_1 \leq {\lhd} \Box {\rhd} \nomj_3
\Rightarrow \nomj_1 \leq (\nomh_1 \circ \nomj_2)\circfor \cnomo_1)$ \\
\end{tabular}
\par}}

If $\nomh_1 \leq {\lhd} \Box {\rhd} \nomj_3$ is not considered an atom (predicate on $\nomj_3$ and $\cnomo_1$), then it needs to be further flattified as explained above: the first step is residuate the inequality on the right of the meta-implication to put $\nomh_1$ on display. Therefore,  $\nomj_1 \leq (\nomh_1 \circ \nomj_2) \circfor \cnomo_1$ is equivalently rewritten as $\nomh_1 \circ \nomj_2 \leq \nomj_1 \circfor \cnomo_1$, which is equivalent to $\nomh_1 \leq \nomj_2 \circfor (\nomj_1 \circfor \cnomo_1)$, by residuation. Thus, $\nomj_1 \leq ((\Diamond \nomj_2 \wedge {\lhd} \Box {\rhd} \nomj_3) \circ \nomj_2) \circfor {\rhd} \nomj_2$ can be rewritten as
\smallskip

{{\centering
\begin{tabular}{rl}
&  $\forall \cnomo_1,\nomh_1({\rhd} \nomj_2 \leq \cnomo_1 \metaand 
\nomh_1 \leq \Diamond \nomj_2 \metaand \nomh_1 \leq {\lhd} \Box {\rhd} \nomj_3
\Rightarrow \nomh_1 \leq \nomj_2 \circfor (\nomj_1 \circfor \cnomo_1))$.
\end{tabular}
\par}}
\smallskip

\noindent There is no clear way in which it is possible to make the inequality $\nomh_1 \leq \nomj_2 \circfor (\nomj_1 \circfor \cnomo_1)$ switch position with ${\rhd} \nomj_2 \leq \cnomo_1 \metaand 
\nomh_1 \leq \Diamond \nomj_2 \metaand \nomh_1 \leq {\lhd} \Box {\rhd} \nomj_3$, unless the two inequalities in the antecedent are merged again.
\end{example}

\end{document}